\documentclass[pdflatex,sn-mathphys,Numbered]{sn-jnl}

\usepackage{tikz}
\usepackage{tikz-cd}
\usetikzlibrary{decorations.pathreplacing, decorations.pathmorphing}
\usetikzlibrary{intersections}
\usetikzlibrary{calc}
\usetikzlibrary{arrows.meta}

\usepackage{graphicx}%
\usepackage{multirow}%
\usepackage{amsmath,amssymb,amsfonts}%
\usepackage{amsthm}%
\usepackage{mathrsfs}%
\usepackage[title]{appendix}%
\usepackage{xcolor}%
\usepackage{textcomp}%
\usepackage{manyfoot}%
\usepackage{booktabs}%
\usepackage{algorithm}%
\usepackage{algorithmicx}%
\usepackage{algpseudocode}%
\usepackage{listings}%
\usepackage{enumitem}
\usepackage{mathtools}
\usepackage{esint}
\usepackage{bbm}

\graphicspath{{figures/}}

\usepackage{caption}
\usepackage{subcaption}

\newtheorem{thm}{Theorem}
\newtheorem{prop}[thm]{Proposition}%

\newtheorem{lem}[thm]{Lemma}

\newtheorem{cor}[thm]{Corollary}

\theoremstyle{remark}
\newtheorem{rem}{Remark}%

\newtheorem{defi}{Definition}%

\newtheorem{hyp}{Hypothesis}

\newcommand{\D}{\mathcal{D}}
\newcommand{\E}{\mathbb{E}}
\newcommand{\F}{\mathcal{F}}

\newcommand{\Pro}{\mathbb{P}}
\newcommand{\R}{\mathbb{R}}
\newcommand{\dst}{\displaystyle}
\newcommand{\eps}{\varepsilon}
\newcommand{\Var}{\textrm{Var}}
\newcommand{\Cov}{\textrm{Cov}}

\newcommand{\vsd}{\vspace{0.2cm}}

\newcommand{\norme}[1]{\left\Vert #1\right\Vert}

\newcommand{\intInter}[1]{[\![ #1 ]\!]}

\definecolor{color1}{rgb}{0.22,0.45,0.70}
\definecolor{color2}{rgb}{0.9, 0.17, 0.31}
\definecolor{color3}{rgb}{0.8, 0.25, 0.33}
\definecolor{blue2}{rgb}{0.2,0.2,0.7}
\definecolor{green2}{rgb}{0.13, 0.7, 0.67}
\definecolor{green3}{rgb}{0.3,0.6,0.4}
\definecolor{mygray}{gray}{0.40}
\definecolor{mylightgray}{gray}{0.70}

\newcommand{\bl}[1]{\textcolor{color1}{#1}}

\newcommand{\re}[1]{\textcolor{color3}{#1}}

\begin{document}

\title[Article Title]{Scattered wavefield in the stochastic homogenization regime}


\author[1]{\fnm{Garnier} \sur{Josselin}}
\email{josselin.garnier@polytechnique.edu}

\author[2]{\fnm{Giovangigli} \sur{Laure}}\email{laure.giovangigli@ensta-paris.fr}

\author*[1,2]{\fnm{Goepfert} \sur{Quentin}}\email{quentin.goepfert@ensta-paris.fr}

\author[3]{\fnm{Millien} \sur{Pierre}}\email{pierre.millien@espci.fr}

\affil[1]{\orgdiv{CMAP}, \orgname{CNRS, Ecole polytechnique, Institut Polytechnique de Paris}, \orgaddress{\city{91120 Palaiseau}, \country{France}}}
\affil[2]{\orgdiv{POEMS}, \orgname{CNRS, Inria, ENSTA Paris, Institut Polytechnique de Paris}, \orgaddress{\city{91120 Palaiseau}, \country{France}}}
\affil[3]{\orgdiv{Institut Langevin, ESPCI Paris, PSL University, CNRS}, 
\orgaddress{1 rue Jussieu, F-75005 \city{Paris}, \country{France}}}


\abstract{ 
In the context of providing a mathematical framework for the propagation of ultrasound waves in a random multiscale medium, we consider the scattering of classical waves (modeled by a divergence form scalar Helmholtz equation) by a bounded object with a random composite micro-structure embedded in an unbounded homogeneous background medium. Using \emph{quantitative}  stochastic homogenization techniques, we provide asymptotic expansions of the scattered field in the background medium with respect to a scaling parameter describing the spatial random oscillations of the micro-structure. Introducing a boundary layer corrector to compensate the breakdown of stationarity assumptions at the boundary of the scattering medium, we prove quantitative $L^2$- and $H^1$- error estimates for the asymptotic first-order expansion. The theoretical results are supported by numerical experiments. 
}

\keywords{Helmholtz equation, quantitative stochastic homogenization, transmission problem, boundary layer}

\maketitle

\section{Introduction and context}

The emergence of quantitative medical imaging techniques that can map the numerical value of a physical parameter in a biological tissue constitutes  a major shift of paradigm for the theory of inverse problems. Imaging modalities are now expected not only to produce images that are anatomically accurate (structural images)  but also stably and quantitatively reconstruct parameters of interest that can help discriminate pathological states.

Medical ultrasound imaging is a powerful, safe, portable and cheap imaging modality that is used in countless physical exams.  Ultrasonic pulses (in the MHz range) are transmitted into the region of interest and the images are obtained by numerically backpropagating the echoes generated by the tissues and recorded on a receiver array. Each tissue and its pathological state will be characterized by a distinct type of speckle on the image.  

The technique relies on the fact that most soft tissues have a mass density and compressibility close to those of water (and ultrasonic waves travel in these tissues almost as in water) yet have echogenic properties that  can be explained by the presence of acoustic heterogeneities of characteristic size much smaller than the wavelength, see \cite{shung1992ultrasonic}.
 
The quantification of these echogenic properties  (known in the literature as \emph{backscattering coefficient estimation} \cite{ueda1985spectral})  relied until now on the introduction of an \emph{ultrasonic reflectivity}  \cite{norton1981ultrasonic} and approximations of the scattered field derived under a set of restrictive hypotheses that do not hold in many practical situations (usually assumptions of the low scatterer concentration, single scattering regime, strictly homogeneous mass density in the medium, uniformity of the excitation beam \ldots \cite{mamou2013quantitative}). Recently, using a formal approach based on a separation of scale in the scattering process, Aubry $\&$  al. have recently obtained spectacular results in quantitative speed of sound imaging on experimental data \cite{lambert2020reflection}.

In this paper, we aim at providing a mathematical framework for the propagation of ultrasound waves in random multiscale media. Using the tools of \emph{stochastic homogenization}, we provide a mathematical model for the acoustic properties of a soft tissue as well as \emph{quantitative} asymptotic expansions of the scattered field with respect to the scale of the acoustic heterogeneities in the medium.

\subsection{State of the art}

Homogenization techniques are an essential tool to address the study of partial differential equations with rapidly oscillating coefficients that exhibit periodic \cite{bensoussan2011asymptotic} or stochastic  \cite{papanicolaou1979boundary,jikov2012homogenization} variations, allowing to derive effective coefficients or asymptotic expansions for the solutions.

In the case of stochastic homogenization, the recent quantification of convergence rates in the case of the Poisson equation in unbounded domains, obtained independently by  Armstrong et al \cite{armstrong2017additive} as well as Gloria and Otto \cite{gloria2015quantification,gloria2015corrector} has initiated a leap in results on the subject, relaxing some of the hypotheses of the aforementioned papers.  Using multiscale inequalities to quantify ergodicity,  quantitative convergence rates were obtained for correlated coefficients with long range correlations \cite{gloria2021quantitative,duerinckx2020multiscale} or in the case of bounded domains \cite{armstrong2019quantitative,josien2021quantitative}. 
Additionally, the emergence of a theory of fluctuations in stochastic homogenization has lead to introducing a new quantity : the homogenization commutator.  In the series of articles \cite{duerinckx2020robustness,duerinckx2020structure} the authors have shown that the fluctuations of the two-scale expansion error of the \emph{so-called} commutator characterizes the fluctuations of all the observables of interest (flux, gradient\ldots).

Classical wave scattering by a medium containing periodically distributed penetrable objects has attracted a lot of attention in the recent years \cite{cakoni2016homogenization,chaumont2023scattering, vinoles2016problemes} with a particular focus on the asymptotic analysis of the boundary corrector \cite{cakoni2019scattering} and the construction of effective transmission conditions \cite{beneteau2021modeles}.  We also refer to  \cite{allaire1999boundary,gerard2012homogenization,prange2013asymptotic} for major contributions to the study of the boundary layers for the Poisson equation.

\subsection{Main contribution}
 In this paper,  we are interested in the scattering of classical waves by a bounded object with a random composite micro-structure embedded in an unbounded homogeneous background medium.  The problem considered is modeled by a divergence form scalar Helmholtz equation with discontinuous rapidly oscillating (at some scale $\varepsilon$ much smaller than the wavelength)  stochastic coefficients.

Building on the methods developed in \cite{gloria2014regularity,gloria2021quantitative} we  establish a first-order (with respect to the parameter $\varepsilon$)  asymptotic expansion of the scattered field inside the object (proposition \ref{prop:homog}).  Introducing a boundary layer corrector to enforce transmission conditions at the boundary of the object we prove $L^2$- and $H^1$-norm convergence rates (proposition \ref{prop:2scErrDecayBC}).   Using the Lippman-Schwinger equation and results on fluctuations of the commutator \cite{duerinckx2020robustness}, we derive a quantitative first-order expansion of the scattered wave outside the object.  We also present numerical illustrations of the solution of the multiscale problem as well as the correctors, and the first-order expansion of the solution. Numerical convergence rates are computed to support the theoretical claims.  

The article is organized as follows:
\begin{itemize}
\item In section \ref{section:model} we present the model for the propagation medium and the stochastic framework required to prove stochastic homogenization results.
\item Section \ref{section:2scErr} is devoted to proving \emph{$L^2$- and $H^1$- quantitative estimates} of the error between the solution of the original problem and the first-order two-scale expansion (proposition \ref{prop:2scErrDecay}).
\item Using the expansions of the solution and its gradient inside the composite medium established in the previous section in conjonction with the Lippman-Schwinger equation satisfied by the scattered field eq \eqref{eq:LSEQ}, we derive in section \ref{section:IRfield} an explicit integral representation formula for an $H^1$- approximation of the scattered field outside the composite medium of order $(d+1)/2$ (Theorem \ref{thm:Rdecay}), where $d$ is the dimension. This theorem along with Corollary~\ref{cor:R1decay} is the main result of the paper. It makes it possible to relate the small-scale fluctuations of the composite medium and the scattered wavefield that can be measured outside the medium. This paves the way towards the resolution of quantitative inverse problems that aim at characterizing the statistics of the composite medium from the statistics of the scattered field.
\item In section \ref{section:numericalSimul}, we show numerical results on the original problem, the effective coefficients and the homogenized problem,  as well as the different correctors.  
We compute the different norm errors between the solution of the original problem and its various approximations to confirm the claims of  proposition \ref{prop:2scErrDecayL2} and theorem \ref{thm:Rdecay}.

\end{itemize}

\newpage
\tableofcontents
\newpage

\section{Presentation of the model} \label{section:model}

We consider a bounded acoustic medium $D \subset \R^d$, $d \in  [\![1,3]\!]$ with a $\mathcal{C}^4$- boundary $\partial D$ and we study the scattering of a time-harmonic plane wave 
\begin{equation}
u^{inc}(x) := exp(i k \theta \cdot x)~\textrm{ for } x\in\R^d
\end{equation} with wave number $k$ and direction $\theta \in \mathbb{S}$.  We assume that a set $S^{\eps}$ of randomly distributed inclusions of characteristic size $\eps > 0$ lies inside the medium $D$. $\eps$ is small compared to the wavelength of the incoming field $2 \pi k^{-1}$. 

\begin{figure} 
\centering
\input{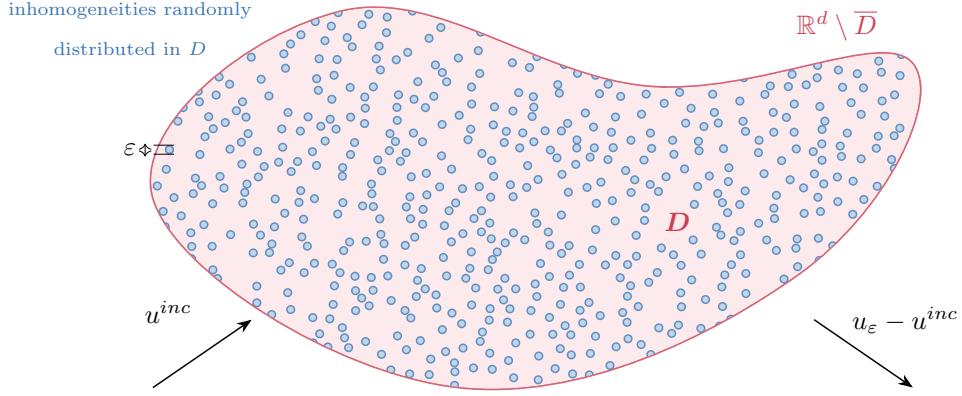}
\caption{Scattering by an obstacle in the stochastic homogenization regime}
\label{fig:schema}
\end{figure}

The outer medium $\R^d \setminus \overline{D}$, the background $D \setminus \overline{S^\eps}$ and the scatterers $S^\eps$ are assumed to be homogeneous with respective parameters $(Id, n_0)$, $(a_M, n_M)$ and $(a_S, n_S)$. The medium parameters are then given by
\begin{equation}
\left\{\begin{array}{l}
\dst \vsd a_\eps := Id \mathbbm{1}_{\R^d \setminus \overline{D}} + a_M\mathbbm{1}_{D\setminus\bar{S^\eps}} + a_S\mathbbm{1}_{S^\eps},\\
\dst \vsd n_\eps := n_0 \mathbbm{1}_{\R^d \setminus \overline{D}} + n_M\mathbbm{1}_{D\setminus\bar{S^\eps}} + n_S\mathbbm{1}_{S^\eps},
\end{array}\right.
\end{equation}
where $a_M$ and $a_S$ are positive definite matrices of $\mathcal{M}_d(\R)$ and $n_0$, $n_M$, $n_S$ are positive. \\
The total field $u_\eps$ is then the unique solution a.s. in $H^1_{loc}(\R^d)$ of the following problem:
\begin{equation}
\left\{\begin{aligned}
& -\nabla\cdot\left(a_\eps(x) \nabla u_\eps(x) \right) - k^2 n_\eps(x) u_\eps(x) = 0 && \textrm{for }x \in \R^d, \\
& \lim_{|x|\to +\infty}|x|^{\frac{d-1}{2}}\left(\frac{\partial (u_\eps - u^{inc})}{\partial |x|}(x)-ik~(u_\eps - u^{inc})(x)\right)=0.
\end{aligned} \right.
\end{equation}

\begin{rem}
In the context of acoustics $u_\eps$ is the pressure. $a_\eps$ and $n_\eps$ relate to the mass density and the bulk modulus of the inner and outer media \cite[Chapter~3.3]{pierce2007basic}. We choose identical parameters for all scatterers. The study can easily be extended to independent and identically distributed parameters as long as the assumptions of uniform ellipticity for $a_\eps$ and uniform boundness from below and above for $n_\eps$ are satisfied.
\end{rem}

We will derive an asymptotic expansion of $u_\eps(x)$ with respect to $\eps$ for $x \in \R^d \setminus \overline{D}$ using quantitative stochastic homogenization techniques. Before doing so, let us specify in this section the different assumptions that we make on the random distribution of scatterers. 

\subsection{Description of the distribution of scatterers} \label{subsec:scatterers}
Let $(x_i)_{i \in \mathbb{N}}$ be the point process in $\R^d$ corresponding to the centers of the scatterers. A scatterer $s_i$, $i \in \mathbb{N}$ centered at $x_i$ consists in an open connected Lipschitz domain $\mathcal{O}$ of radius $r := \underset{x,y \in \mathcal{O}}{\max}\vert x-y \vert$. We denote by $S := \underset{i \in \mathbb{N}}{\cup} s_i$ the set of scatterers of radius $r=1$ in $\R^d$. Let $(\Omega, \F, \Pro)$ be a probability space. We make the following assumptions on $(x_i)_{i \in \mathbb{N}}$:
\begin{itemize}
\item[-]$(x_i)_{i \in \mathbb{N}}$ is stationary, \textit{i.e.} its distribution law is invariant by translation and ergodic;
\item[-] the scatterers lie at a distance at least $\delta > 0$ from one another, \textit{i.e.} there exists $\delta > 0$ such that  $$\forall i \neq j,~dist(s_i, s_j) > \delta\quad \textrm{a.s.} $$ 
\end{itemize} 
We introduce the parameters
\begin{equation}
\left\{\begin{array}{l}
\dst \vsd a := a_M\mathbbm{1}_{\R^d \setminus \overline{S}} + a_S\mathbbm{1}_{S},\\
\dst \vsd n := n_M\mathbbm{1}_{\R^d \setminus \overline{S}} + n_S\mathbbm{1}_{S}.\\
\end{array}\right.
\end{equation}
For $\eps > 0$, we define
$$\mathcal{N}_\eps := \{i \in  \mathbb{N}~\vert~\eps x_i \in D \}.$$
$\mathcal{N}_\eps$ corresponds to the collection of scatterers of size $\eps$ that lie in $D$. We subsequently denote 
\begin{equation}
S^\eps := \underset{i \in \mathcal{N}_\eps}{\cup} \eps s_i\cap D.
\end{equation}
Note that we have then 
$$\forall x \in D,~a_\eps(x) = a\biggl(\frac{x}{\eps}\biggr) \textrm{ and } n_\eps(x) =  n\biggl(\frac{x}{\eps}\biggr).$$

Figure~\ref{fig:aeps} illustrate an example of a realization of $a_\eps$ in $D$.

\subsection{Stochastic setting} \label{subsection:randomness}

As it is customary in stochastic homogenization, we define stationarity and ergodicity through an action $(\tau_x)_{x\in\R^d}$ of the group $(\R^d , +)$ on $(\Omega, \F)$. \\
We thus equip $(\Omega, \F)$ with $(\tau_x)_{x\in\R^d}$ that verifies: 
\begin{itemize}
\item[-]the map $\tau: \dst\left\{\begin{array}{ccc}
\vsd\R^d\times \Omega&\to&\Omega
\\({x},\omega)&\mapsto&\tau_{x} \omega
\end{array}\right.$ is measurable,
\item[-] $\forall x,y\in\R^d, \tau_{x+y}=\tau_x \circ \tau_y,$ 
\item[-]For all $ x\in\R^d$, $\tau_x$ preserves $\Pro$,  \textit{i.e.}
$$\forall A\in \F,\,\, \Pro(\tau_{x}A)=\Pro(A).$$
\end{itemize}

\begin{defi}[Stationarity] \label{def:stationary}
In the rest of the paper, a random process $f : \R^d\times \Omega \to \R^p$ is said to be stationary (with respect to $\tau$) if
\begin{equation}
\forall x, y\in\R^d, \, \text{a.e.} \,\omega\in\Omega,\,\, f(x+y,\omega)=f(x, \tau_{y}\omega).
\end{equation}
\end{defi}

\noindent Moreover, we assume that the action $(\tau_{x})_{x\in\R^d}$ is ergodic.
\begin{defi}[Ergodicity] \label{def:ergodicity}
Any $\tau$-invariant event has probability 0 or 1, that is,
\begin{equation}\forall A\in\F, \,\,(\forall x\in\R^d, \,\, \tau^{-1}_{x} A=A)\implies (\Pro(A)\in\{0,1\}).
\end{equation}
\end{defi}
We can now write in terms of $\tau$ the stationary and ergodic assumption on $\left\{x_i^\omega\right\}_{i \in \mathbb{N}}$ the centers of the scatterers for the realization $\omega\in \Omega$.
\begin{equation}
\forall \omega \in \Omega,~\forall y \in \R^d,~\left\{x_i^{\omega} + y \right\}_{i \in \mathbb{N}} = \left\{x_i^{\tau_y \omega} \right\}_{i \in \mathbb{N}} .
\end{equation}
These two assumptions are the minimal and classical assumptions on the distribution of scatterers $(x_i)_{i \in \mathbb{N}}$ that we require for qualitative stochastic homogenization. 
In the rest of the paper the dependency on randomness $\omega \in \Omega$ is not mentioned explicitly.

We also assume that the process $(x_i)_{i \in \mathbb{N}}$ or equivalently $S$ verifies a quantitative mixing condition. 
We choose to express this condition as a multiscale variance inequality as introduced in \cite{DuerinckxGloria2017}. This assumption is verified by the most common hardcore point processes such as Mat\`ern point process \cite[Section~6.5.2]{Matern} as proved in \cite[Section~3]{DuerinckxGloria2017}.

\begin{hyp}[Mixing hypothesis] \label{hyp:mix_hyp}
There exists a non-increasing weight function \mbox{ $\pi:\R^+\to\R^+$ }with exponential decay such that $S$ verifies for all $\sigma(S)$- measurable random variable $F(S)$,
\begin{equation} \label{eq:decorr}
\Var\left[F(S)\right]\leq\E\left[\int_1^{+\infty}\int_{\R^d} \left(\partial^{osc}_{S, B_{\ell}(x)} F(S)\right)^2\mathrm{d}x \ell^{-d} \pi(\ell-1)\mathrm{d}\ell\right],
\end{equation}
where $B_\ell(x)$ is the ball with radius $\ell\geq 0$ and center $x\in\R^d$ and the oscillation $\partial^{osc}_{S, B_l(x)} F(S)$
 of $F(S)$ with respect to $S$ on $B_\ell(x)$ is defined by: \begin{multline*}\partial^{osc}_{S, B_l(x)} F(S):=\textrm{sup ess}\left\{F(S')|S'\cap(\R^d\setminus B_\ell(x))=S\cap(\R^d\setminus B_\ell(x))\right\}\\-\textrm{inf ess}\left\{F(S')|S'\cap(\R^d\setminus B_\ell(x))=S\cap(\R^d\setminus B_\ell(x))\right\}.
\end{multline*}
\end{hyp}

\begin{prop}[Mat\`ern process {\cite[Proposition~3.3]{DuerinckxGloria2017}}] \label{prop:mattern}
The Mat\`ern process verifies \eqref{eq:decorr} with the weight function $\pi$:
\begin{equation}
\pi(l) = C e^{-\frac{1}{C}l}
\end{equation}
for some $C>0$.
\end{prop}

\begin{rem}
Note that the Hypothesis~\ref{hyp:mix_hyp} implies that the covariance function of $S$ : $C_S(x) := \Cov(S(0), S(x))$ satisfies \cite[Proposition~1.3]{duerinckx2017multiscale}
\begin{equation}
\forall x \in \R^d,~\vert C_S(x) \vert \lesssim \int_{\max(\frac{1}{2}(\vert x \vert - 2), 0)}^\infty \pi(\ell) \mathrm{d}\ell.
\end{equation}
where the notation $\lesssim$ stands for "inferior up to a multiplicative constant dependent only on the dimension and possibly other controlled quantities" and will be used throughout the paper. 
For a Mat\`ern process this implies that $C_S$ has an exponential decay.
\end{rem}

\section{Two-scale asymptotic expansion of the field} \label{section:2scErr}

\subsection{Homogenized problem}

We restrict our domain of study to $B_R$ the ball of radius $R > 0$ centered at $0$, via the Dirichlet-to-Neumann operator $\Lambda: H^{\frac{1}{2}}(\partial B_R) \to H^{-\frac{1}{2}}(\partial B_R)$. This operator takes a Dirichlet data $g \in H^{\frac{1}{2}}(B_R)$ and maps it to the Neumann trace of $u$ on $\partial B_R$ \textit{i.e.} ${\Lambda g = \nabla u \cdot \nu \in H^{-\frac{1}{2}}(\partial B_R)}$ where $u$ is the outgoing the solution of 
$$- \Delta u - k^2 n_0 u = 0 \textrm{ in } \R^d\setminus{\overline{B_R}}, \textrm{ satisfying } u\vert_{\partial B_R} =  g.$$
$\Lambda$ is continuous, self-adjoint and non-positive and its expansion  in terms of Hankel functions can be found for example in \cite{chandler2008wave}, \cite[Section~2.6.3]{nedelec2001acoustic} and \cite{melenk2010convergence}. \\
We thus consider $u_\eps$ the a.s. unique solution in $H^1(B_R)$ to 
\begin{equation}\label{eq:mainBR}
\left\{\begin{aligned}
& -\nabla\cdot\left(a_\eps \nabla u_\eps \right) - k^2 n_\eps u_\eps = 0 && \textrm{in } B_R, \\
& \nabla (u_\eps - u^{inc})\cdot \nu = \Lambda(u_\eps - u^{inc}) && \textrm{on } \partial B_R. \\
\end{aligned} \right.
\end{equation}
The well-posedness of \eqref{eq:mainBR} for non-smooth coefficients is a difficult problem in $3d$.  We refer to \cite{ball2012uniqueness} for the proof in the $L^\infty$ case.  However,  the proof relies on Fredholm theory and unique continuation principle and therefore does not yield a uniform explicit control with respect to $\varepsilon$ and $\omega$.  To obtain this type of uniform control that will be necessary for the homogenization process,  we have to add some additional assumptions on the coefficients.  For $kR$ sufficiently small or $\Im k>0$  the sesquilinear form associated to \eqref{eq:mainBR} can be proved to be coercive and the uniform bound in $\eps$ and $\omega$ of $u_\eps$ can be achieved \cite{spence}. We also point out that some other methods were developed in \cite{spence} and \cite{bouchitte2004homogenization} to obtain uniform control of the solution, but they do not apply to our specific problem. 
Here, we assume that the sesquilinear form associated to \eqref{eq:mainBR} is coercive so that Proposition~\ref{prop:stabAbsorb_smallKR} holds. The following homogenization theorem follows directly.
\begin{prop}[Qualitative homogenization theorem] \label{prop:homog}
A.s. the unique solution $u_\eps \in H^1(B_R)$ of \eqref{eq:mainBR} converges weakly in $H^1(B_R)$ towards $u_0$, the unique solution in $H^1(B_R)$ of the following problem
\begin{equation}\label{eq:u0BR}
\left\{\begin{aligned}
& -\Delta u_0 - k^2 n_0 u_0 = 0 && \textrm{in } B_R \setminus \overline{D}, \\
& -\nabla\cdot\left(a^{hom} \nabla u_0 \right) - k^2 n^{hom} u_0 = 0 && \textrm{in } D, \\
& u_0^- - u_0^+ = 0 && \textrm{on } \partial D, \\
& \nabla u_0^- \cdot \nu - a^{hom} \nabla u_0^+ \cdot \nu = 0 && \textrm{on } \partial D, \\
& \nabla (u_0 - u^{inc})\cdot \nu = \Lambda(u_0 - u^{inc}) && \textrm{on } \partial B_R, \\
\end{aligned} \right.
\end{equation}
where the superscripts $^-$ and $^+$ denote the traces outside and inside $D$.\\
The homogenized coefficients $a^{hom} \in \mathcal{M}_d(\R)$ and $n^{hom} \in (0, +\infty)$ are defined as follows
\begin{equation} \label{eq:homcoeff}
\left\{\begin{aligned}
a^{hom}_{i,j} & = \E[e_j \cdot a (e_i + \nabla \phi_i)], \\
n^{hom} & =  \E[n],
\end{aligned}
\right.
\end{equation}
and for $i \in \intInter{1,d}$ the corrector $\phi_i$ is defined as in the forthcoming Definition~\ref{def:corrector}.
\end{prop}
We define for any integrable function $f$ and any domain $B$ the notation $\fint_B f$ as 
$$\fint_B f := \frac{1}{|B|}\int_B f. $$
\begin{defi}[Corrector] \label{def:corrector}
Let $(\phi_i)_{i \in \intInter{1,d}}$ be the unique vector field such that for all $i \in \intInter{1,d}$
\begin{enumerate}[label=(\alph*)]
\item a.s. $\phi_i \in H^1_{loc}(\R^d)$ is the solution in $\D'(\R^d)$ of
\begin{equation} \label{eq:phi}
-\nabla \cdot \biggl( a (\nabla \phi_i + e_i) \biggr) = 0 \quad \textrm{in } \R^d,
\end{equation} 
with the anchoring condition $$ \fint_{\square_0} \phi_i = 0, $$
where $\square_x$ denotes the unit square centered at $x$
\begin{equation}
\square_x := [-\frac{1}{2} + x, \frac{1}{2} + x]^d. 
\end{equation} 
\item $\nabla \phi$ is stationary, has finite second moments and vanishing expectation.
\end{enumerate}
\end{defi}

Once the uniform bound on $u_\eps$ is established, the proof of Proposition~\ref{prop:homog} follows from the classical steps of stochastic homogenization using Tartar's method \cite{tartar} of oscillating test functions. For the sake of completeness, we detail it in Appendix~\ref{app:qualitative_hom}. 
$a^{hom}$ is definite positive ensuring the well-posedness of the homogenized problem \eqref{eq:u0BR}.

\subsection{Two-scale expansion error and boundary layer}
The qualitative homogenization theory implies that a.s. $u_\eps$ converges to $u_0$ strongly in $L^2(B_R)$ and weakly in $H^1(B_R)$. In order to upgrade this result to strong convergence in $H^1(B_R)$ and get a quantitative rate of convergence, one needs to consider the contribution of the first-order corrector $u_{1,\eps} \in H^1(B_R \setminus \overline{D}) \times H^1(D)$.

\begin{defi}[first-order corrector] \label{def:u1}

Let $u_{1,\eps} \in H^1(B_R \setminus \overline{D}) \times H^1(D)$ be the first-order corrector defined by: 
\begin{equation} \label{eq:u1}
u_{1,\eps}(x) := \mathbbm{1}_{D}(x) \sum_{i=1}^d \phi_i \biggl(\frac{x}{\eps} \biggr) \partial_i u_0(x) \quad \textrm{for } x \in B_R.
\end{equation}
\end{defi}

This definition of $u_{1,\eps}$ corresponds to the usual definition inside $D$. Since there is no micro-structure outside of $D$, we extend it to $u_{1,\eps} = 0$ in $B_R \setminus \overline{D}$. \\
Since $u^{inc} \in \mathcal{C}^\infty(\R^d)$ and $ \partial D$ is $\mathcal{C}^4$, ${u_0}_{|_D}$ is in $H^2(D)$ (see Appendix~\ref{app:regularity}, Proposition~\ref{prop:regHk}). Therefore, ${u_{1, \eps}}_{|_D}$ is indeed in $H^1(D)$. Similarly, we introduce ${w_\eps \in H^1(B_R \setminus \overline{D}) \times H^1(D)}$, the two-scale expansion defined as follows : 
\begin{defi}[Two-scale expansion] \label{def:twoScExp}
\begin{equation} \label{eq:weps}
\begin{split}
w_\eps(x) & := u_0(x) + \eps u_{1,\eps}(x) \quad \textrm{for } x \in B_R. \\
\end{split}
\end{equation}
\end{defi}

\subsection{Two-scale error - boundary corrector} \label{subsec:homExpansion}

We want to quantify the error $Z_\eps := u_\eps - w_\eps$ between the solution of \eqref{eq:mainBR} and its two-scale expansion \eqref{eq:weps} in $H^1(B_R \setminus \overline{D}) \times H^1(D)$. \\
In a bounded Lipschitz domain $U$, it has been shown \cite[Chapter~6]{armstrong2019quantitative} in dimension $3$ that for the Poisson equation, both with Dirichlet and Neumann boundary conditions, the following holds
$$\E[\norme{\nabla u_\eps - \nabla w_\eps}_{L^2(U)}^2]^{\frac{1}{2}} \lesssim \eps^{1/2} \norme{u_0}_{W^{2, \infty}(U)}.$$ 
The order $1/2$ of the error is due to the fact that $u_{1,\eps}$ and thus $w_\eps$ do not satisfy the Dirichlet or Neumann boundary conditions on $\partial U$. To obtain an error of order $\eps$, one needs to take into account what happens at the boundary and add the correct boundary corrector \cite{armstrong2019quantitative}. We establish a similar result for the Helmholtz transmission problem. Let us define two extended correctors which appear naturally while deriving the problem verified by $Z_\eps$.

\begin{defi}[Extended corrector] \label{def:ExtendedCorrector}
Let $\beta := (\beta_i)_{i \in \intInter{1,d}}$ be the unique vector field and let $\sigma :=  (\sigma_{i,jm})_{i,j,m \in \intInter{1,d}}$ be the unique tensor field  such that for all $i,j,m \in \intInter{1,d}$, 
\begin{enumerate}[label=(\alph*)]
\item A.s. $\beta_i \in H^1_{loc}(\R^d)$ and $\sigma_{i,jm} \in H^1_{loc}(\R^d)$ are the solutions in $\D'(\R^d)$ of: 
\begin{equation} \label{eq:extended_corrector}
\left\{\begin{aligned}
- \Delta \beta_i(y) & = \partial_i (n(y) - n^{hom}),  \\
- \Delta \sigma_{i,jm}(y) & = \partial_j q_{im}(y) - \partial_m q_{ij}(y), \\
\end{aligned}
\right.
\end{equation}
with $$q_i := a(e_i + \nabla \phi_i) - a^{hom} e_i$$
and are anchored with the condition: 
\begin{equation}
\fint_{\square_0} \sigma_{i,jm} = \fint_{\square_0} \beta_i = 0.
\end{equation} 
Furthermore, $\sigma_i$ is skew-symmetric and verifies a.s.
\begin{equation}
\nabla \cdot \sigma_i =  q_i 
\end{equation} 
where for $i,j \in \intInter{1,d}$,
\begin{equation}
(\nabla \cdot \sigma_i)_j := \sum_{m=1}^d \partial_m \sigma_{i, jm}
\end{equation} 
and $\beta$ verifies a.s. 
\begin{equation}
\nabla \cdot \beta = n - n^{hom}. 
\end{equation} 
\item $\nabla \beta_i$ and $\nabla \sigma_{i,jm}$ are stationary, have finite second moments and vanishing expectation.
\end{enumerate}

\end{defi} 

$\sigma$ is the classical extended corrector in stochastic homogenization of the operator $-\nabla \cdot a \nabla$ and can be found for example in \cite[Lemma~1]{gloria2014regularity}. The well-posedness of $\beta$ is proven in the exact same manner. \\

We can now write the problem verified by $Z_\eps$. 
\begin{lem}[Two-scale error] \label{lem:twoScError}
$Z_\eps := u_\eps - w_\eps$ is a.s. the unique solution in~$H^1(B_R \setminus \overline{D})~\times~H^1(D)$ of:
\begin{equation} \label{eq:Zeps}
\left\{\begin{aligned} 
& -\Delta Z_\eps - k^2 n_0 Z_\eps = 0 &&\textrm{in}\, B_R \setminus\bar{D},\\
& -\nabla \cdot a_\eps \nabla Z_\eps - k^2 n_\eps Z_\eps = \nabla \cdot F_\eps + k^2 G_\eps &&\textrm{in}\, D,\\
& Z_\eps^- - Z_\eps^+ = \eps u_{1,\eps} &&\textrm{on}\, \partial D, \\
& \nabla Z_\eps^- \cdot \nu - a_\eps \nabla Z_\eps^+ \cdot \nu \\
& \hspace{1cm} = F_\eps \cdot \nu + \eps \sum_{i=1}^d \left(\nabla \cdot(\sigma_i^\eps \partial_i u_0)^+\right) \cdot \nu - k^2 \eps (\beta^\eps u_0)^+ \cdot \nu   &&\textrm{on}\, \partial D, \\
& \nabla Z_\eps \cdot \nu = \Lambda(Z_\eps) &&\textrm{on}\, \partial B_R,
 \end{aligned}
\right.
\end{equation}
where $F_\eps \in H^1(D)$, $G_\eps \in H^1(D)$ are defined as follows
\begin{equation} \label{Feps}
F_\eps := \eps \sum_{i=1}^d (a_\eps \phi_i^\eps -\sigma_i^\eps)\nabla (\partial_i u_0) + \eps k^2 \beta^\eps u_0,
\end{equation}
\begin{equation}\label{Geps}
G_\eps := \eps \sum_{i=1}^d (n_\eps \phi_i^\eps - \beta_i^\eps) \partial_i u_0 .
\end{equation}
\end{lem}
Here, $\phi^\eps$ denotes $\phi^\eps(\cdot) := \phi(\frac{\cdot}{\eps}) $. $\beta_i^\eps$ and $\sigma_i^\eps$ are defined similarly from $\beta_i$ and $\sigma_i$ and for the rest of the paper. Note that we have then 
$$\nabla \phi^\eps(\cdot) = \frac{1}{\eps}(\nabla \phi) (\frac{\cdot}{\eps}).$$

\begin{proof}
Let us first derive the problem satisfied by $Z_\eps$ before proving well-posedness. In $B_R \setminus \overline{D}$ and on $\partial B_R$, $u_\eps$ and $w_\eps$ verify the same equation and so does $Z_\eps$.
Using the equation \eqref{eq:mainBR} for $u_\eps$ and equation \eqref{eq:u0BR} for $u_0$, we have moreover
\begin{equation}
\begin{split}
-\nabla \cdot & (a_\eps \nabla Z_\eps) -k^2 n_\eps Z_\eps \\
& = -\nabla \cdot (a^{hom} \nabla u_0) - k^2 n^{hom} u_0 + \nabla \cdot \left(a_\eps (\nabla u_0 + \eps \nabla u_{1,\eps}) \right) + k^2 n_\eps (u_0 + \eps u_{1,\eps})  \quad \textrm{in } D.
\end{split}
\end{equation}
By the definition of the extended corrector of Definition~\ref{def:ExtendedCorrector},
\begin{equation} \label{eq:Geps}
\begin{split}
- k^2 n^{hom} & u_0 + k^2 n_\eps (u_0 + \eps u_{1,\eps}) \\
& = k^2 (n_\eps - n^{hom}) u_0 + \eps k^2 n_\eps \sum_{i=1}^d \phi_i^\eps \partial_i u_0  \\
& = k^2 (\nabla \cdot \beta) (\frac{\cdot}{\eps}) u_0 + \eps k^2 n_\eps \sum_{i=1}^d \phi_i^\eps \partial_i u_0 \\
& = k^2 \eps \nabla \cdot (\beta^\eps u_0) + \eps k^2 \sum_{i=1}^d - \beta^\eps_i \partial_i u_0 +  n_\eps \phi_i^\eps \partial_i u_0 \\
& = k^2 G_\eps + \eps k^2 \nabla \cdot (\beta^\eps u_0).
\end{split}
\end{equation}
By the skew-symmetry of $\sigma_i$, $i \in \intInter{1,d}$, for all $x \in D$,
\begin{equation} 
\begin{split}
\nabla \cdot \left(\sum_{i=1}^d (\nabla \cdot \sigma_i)(\frac{x}{\eps}) \partial_i u_0(x) \right) & = \nabla \cdot (\eps \sum_{i=1}^d (\nabla \cdot \sigma_i^\eps(x)) \partial_i u_0(x) ) \\
& = \eps \sum_{i,j,m=1}^d \partial_{j} ((\partial_m \sigma_{i,jm}^\eps(x)) \partial_i u_0(x) ) \\
& = \eps \sum_{i,j,m=1}^d \partial_{jm} (\sigma_{i,jm}^\eps(x) \partial_i u_0(x) ) - \partial_{j} ( \sigma_{i,jm}^\eps(x) (\partial_{im} u_0(x)))  \\
& = -\eps \sum_{i,j,m=1}^d \partial_{j} ( \sigma_{i,jm}^\eps(x) (\partial_{im} u_0(x) ))  \\
& = -\eps \nabla \cdot \left(\sum_{i=1}^d \sigma_i^\eps(x) \nabla \partial_i u_0(x) \right).
\end{split}
\end{equation}
Therefore, we obtain
\begin{equation} \label{eq:Feps}
\begin{split}
-\nabla \cdot  & (a^{hom} \nabla u_0) + \nabla \cdot (a_\eps (\nabla u_0 + \eps \nabla u_{1,\eps})) \\
& = \sum_{i=1}^d \nabla \cdot \biggl(\biggl((a (e_i + \nabla \phi_i) - a^{hom} e_i\biggr)(\frac{\cdot}{\eps}) \partial_i u_0 \biggr) + \eps \nabla \cdot \left(a_\eps \phi_i^\eps \nabla (\partial_i u_0 )\right)  \\
& = \sum_{i=1}^d \nabla \cdot \left((\nabla \cdot \sigma_i)(\frac{\cdot}{\eps}) \partial_i u_0 \right) + \eps \nabla \cdot \left(a_\eps \phi_i^\eps \nabla (\partial_i u_0)\right) \\
& = \nabla \cdot \left(\eps \sum_{i=1}^d \left((a_\eps \phi_i^\eps - \sigma_i^\eps) \nabla \partial_i u_0 \right)\right)  \\
& = \nabla \cdot F_\eps - \eps k^2 \nabla \cdot (\beta^\eps u_0).
\end{split}
\end{equation}
By combining \eqref{eq:Geps} and \eqref{eq:Feps}, we obtain the error satisfied inside $D$. \\
Using the jump conditions of $u_\eps$ and $u_0$ across $\partial D$, one gets:
\begin{equation}
Z_\eps^- - Z_\eps^+ = \eps u_{1,\eps},
\end{equation}
and the flux jump:
\begin{equation}
\begin{split}
\nabla & Z_\eps^- \cdot \nu  - a_\eps \nabla Z_\eps^+ \cdot \nu \\
& = \sum_{i=1}^d ((a_\eps (e_i + \nabla \phi_i^\eps) - a^{hom} e_i) \partial_i u_0)^+ \cdot \nu + \eps (a_\eps \phi_i^\eps \nabla  (\partial_i u_0))^+ \cdot \nu \\
& = \sum_{i=1}^d (\eps (\nabla \cdot \sigma_i^\eps) \partial_i u_0)^+ \cdot \nu + \eps (a_\eps \phi_i^\eps \nabla  \partial_i u_0)^+ \cdot \nu \\
& =  \eps \sum_{i=1}^d((a_\eps \phi_i^\eps - \sigma_i^\eps) \nabla (\partial_i u_0) )^+ \cdot \nu + \nabla \cdot (\sigma_i^\eps \partial_i u_0)^+ \cdot \nu  \\
& = F_\eps \cdot \nu + \eps \sum_{i=1}^d \nabla \cdot (\sigma_i^\eps \partial_i u_0)^+ \cdot \nu - k^2 \eps (\beta^\eps u_0)^+ \cdot \nu. \\
\end{split}
\end{equation}
The well-posedness of \eqref{eq:Zeps} is a direct consequence of Proposition~\ref{prop:stabAbsorb_smallKR}. Since a.s. for $i,j,k \in \intInter{1,d}$, $\phi_i$, $\beta_i$, $\sigma_{i,jk} \in H^1_{loc}(\R^d)$ then $\phi_i$, $\beta_i$, $\sigma_{i,jk} \in H^1(D)$. Moreover, $u_0 \in H^2(D)$ and $u_{1, \eps} \in H^1(D)$. Therefore, $F_\eps$, $G_\eps \in H^1(D)$, ${u_{1, \eps}}_{\vert \partial D} \in H^{\frac{1}{2}}(\partial D)$ and $F_\eps \cdot \nu + \eps \sum_{i=1}^d \nabla \cdot (\sigma_i^\eps \partial_i u_0)^+ \cdot \nu - k^2 \eps (\beta^\eps u_0)^+ \cdot \nu \in H^{-\frac{1}{2}}(\partial D)$, and we can apply Proposition~\ref{prop:stabAbsorb_smallKR}.
\end{proof}

As it is customary in homogenization in the presence of boundary since \cite{allaire1999boundary}, we introduce the boundary corrector $v_\eps$ also called boundary layer, that ensures that $Z_\eps - v_\eps$ verifies the transmission conditions on $\partial D$.

\begin{defi}[Boundary corrector] \label{def:BC}
Let $v_\eps$ be the a.s. unique solution in $H^1(B_R \setminus \overline{D}) \times H^1(D)$ of
\begin{equation} \label{eq:veps}
\left\{\begin{aligned} 
& - \Delta v_\eps - k^2 n_0 v_\eps = 0 &&\textrm{in}\, B_R\setminus\overline{D},\\
&-\nabla\cdot a_\eps \nabla v_\eps - k^2 n_\eps v_\eps = 0&&\textrm{in}\, D,\\
& v_\eps^- - v_\eps^+ = \eps u_{1,\eps} &&\textrm{on}\, \partial D, \\
& \nabla v_\eps^- \cdot \nu - a_\eps \nabla v_\eps^+ \cdot \nu = \eps \sum_{i=1}^d \left(\nabla \cdot(\sigma_i^\eps \partial_i u_0)^+\right) \cdot \nu - k^2 \eps (\beta^\eps u_0)^+ \cdot \nu   &&\textrm{on}\, \partial D, \\
& \nabla v_\eps \cdot \nu = \Lambda(v_\eps) && \textrm{on}\, \partial B_R. \\
 \end{aligned}
\right.
\end{equation}
\end{defi}
The well-posedness is once again a consequence of Proposition~\ref{prop:stabAbsorb_smallKR}.

\begin{rem}
The definition of this corrector is very similar to the boundary layer introduced in \cite{Cakoni2016OnTH} which deals with the periodic case. \\
However, all the analysis done in \cite{Cakoni2016OnTH} cannot be applied here as it uses the $L^\infty(\R^d)$-boundness of the corrector which does not hold here. \\
\end{rem}

\begin{prop}[Two-scale error with the boundary layer] \label{prop:twoscerrBC}
A.s. $Z_\eps - v_\eps$ is the unique solution in $H^1(B_R)$ of
\begin{equation} \label{eq:Zepsveps}
\left\{\begin{aligned} 
&-\nabla\cdot a_\eps \nabla (Z_\eps - v_\eps) - k^2 n_\eps (Z_\eps - v_\eps) = \nabla \cdot F_\eps + k^2 G_\eps && \textrm{in} \, B_R,\\
& \nabla (Z_\eps - v_\eps) \cdot \nu = \Lambda(Z_\eps - v_\eps) && \textrm{on} \, \partial B_R.\\
 \end{aligned}
\right.
\end{equation}
Moreover, $Z_\eps - v_\eps$ verifies a.s.

\begin{equation} \label{eq:H1err}
\norme{Z_\eps - v_\eps}_{H^1(B_R)} \lesssim \norme{F_\eps}_{L^2(D)} + \norme{G_\eps}_{L^2(D)}.    
\end{equation}

\end{prop}
Once again, we apply the Proposition~\ref{prop:stabAbsorb_smallKR} for the well-posedness of \eqref{eq:Zepsveps}. \\
In order to quantify the convergence of $u_\eps$ towards $u_0 + \eps u_{1, \eps} + v_\eps$, we are now left with estimating the right hand side of \eqref{eq:H1err}. This is easily done with the a.s. corrector estimates established in \cite[Theorem~2]{Gloria2019} for coefficients verifying the mixing Hypothesis~\ref{hyp:mix_hyp} as we show in the next section.

\subsection{Convergence rate of the two-scale expansion} \label{subsec:homThm}
In this section we estimate the convergence of the two-scale expansion error with and without the boundary layer $v_\eps$ both in $H^1$- and $L^2$-norms. \\
As mentioned above, the proof relies on the corrector bounds established in \cite[Theorem~4]{Gloria2019} for correlated fields satisfying the Hypothesis~\ref{hyp:mix_hyp}. We recall these results below. 
\begin{prop}[Corrector bounds] \label{prop:corrbounds}
Under the mixing Hypothesis~\ref{hyp:mix_hyp}, 
\begin{enumerate}[label=(\alph*)]
\item  ($\Pro-a.s.$ corrector bound): \\
There exists an a.s. finite (non-stationary) random field $x \mapsto \mathcal{C}(x)$ such that for all $x \in \R^d$,
\begin{equation} \label{eq:corrector_bounds}
\left(\fint_{\square_x}|\phi|^2 + |\sigma|^2 + |\beta|^2 \right)^{\frac{1}{2}} \leq \mathcal{C}(x) \mu_d(|x|),
\end{equation}

where for all $y \in \R^+$,
 \begin{equation} \label{eq:mu_d}
\mu_d(y) = \left\{
\begin{aligned}
& \sqrt{y} && \textrm{if } d = 1,\\
& |\log(2 + y)|^{\frac{1}{2}}&& \textrm{if } d = 2, \\
& 1 && \textrm{if } d = 3.
\end{aligned} \right.
\end{equation}

\item (Corrector bound in average): \\
Furthermore, for all $y \in \R^d$, $\mathcal{C}(y)$ satisfies the following stochastic integrability
\begin{equation} \label{eq:Exp_corrector_bounds}
\E[\exp(\frac{1}{C}\mathcal{C}(y)^\gamma)] \leq 2
\end{equation}
for some constant $C >0$ depending on $d$, $n_0$, $n_M$, $n_S$, $a_M$, $a_S$ and exponent $\gamma > 0$ depending on $d$ and the exponential decay rate of $\pi$.
\item (Mean-value property): \\
There exists a stationary $\frac{1}{8}$-Lipschitz continuous random field $r_*>1$ (the so-called minimal radius) satisfying \eqref{eq:Exp_corrector_bounds} such that for all $\ell \geq 1$
\begin{equation} \label{eq:grad_corr}
\int_{B_\ell(x)} |\nabla \phi|^2 \lesssim (\ell + r_*(x))^d.
\end{equation}
\end{enumerate}
\end{prop}

For any bounded domain $B\subset\R^d$, we consider the covering of $B$ with squares of size $\eps$ and define $P_\eps(B)$ as the set of centers of those squares, \emph{i.e.}
$$P_\eps(B) := \{ x \in \mathbb{Z}^d, \eps \square_x \cap B \neq \emptyset \} .$$
We prove the following estimate for the two-scale expansion error with the boundary layer in $H^1(B_R)$.
\begin{prop}[$H^1$- convergence of the two-scale expansion with the boundary corrector] \label{prop:2scErrDecayBC}
Let $u_\eps \in H^1(B_R)$ be the a.s. solution of \eqref{eq:mainBR} and ${u_0 \in H^1(B_R)}$ such that ${{u_0}_{|_D} \in W^{2, \infty}(D)}$ be the solution of \eqref{eq:u0BR}. Let $u_{1,\eps}$ be the corrector defined by \eqref{eq:u1} and $v_\eps$ be the boundary corrector solution of \eqref{eq:veps}. \\ 
Then a.s.
\begin{equation}
\norme{u_\eps - u_0 - \eps u_{1,\eps} - v_\eps}_{H^1(B_R)} \lesssim \eps \mu_d(\frac{1}{\eps}) \chi_\eps \norme{u_0}_{W^{2, \infty}(D)},
\end{equation}
where $\mu_d$ is defined in Proposition~\ref{prop:corrbounds} and $\chi_\eps$ is the random variable defined as: 
\begin{equation}
\chi_\eps := \left(\eps^d \sum_{z \in P_\eps(D)}  \mathcal{C}(z)^2 \right)^{\frac{1}{2}},
\end{equation}
with $\mathcal{C}$ also defined in Proposition~\ref{prop:corrbounds}. 
In particular, $\chi_\eps$ satisfies the stochastic integrability \eqref{eq:Exp_corrector_bounds} and it holds
\begin{equation}
\E\left[\norme{u_\eps - u_0 - \eps u_{1,\eps} - v_\eps}_{H^1(B_R)}^2\right]^{\frac{1}{2}} \lesssim \eps \mu_d(\frac{1}{\eps}) \norme{u_0}_{W^{2, \infty}(D)}.
\end{equation}

\end{prop}

\begin{rem}
This result is an equivalent of the result obtained in \cite[Chapter~6]{armstrong2019quantitative}, both for Dirichlet and Neumann boundary conditions on $\partial D$. The proof of Proposition~\ref{prop:2scErrDecayBC} follows similar steps as the proofs in \cite{armstrong2019quantitative}, that dealt with the case ${u_0}_{|_D} \in W^{2, \infty}(D)$. \\
In Appendix~\ref{appendix:homLessReg}, Proposition~\ref{prop:2scErrDecayBClessRegular}, we extend the result of the proposition to the case where ${u_0}_{|_D} \in W^{1 + \alpha, p}(D)$ for $\alpha \in (0,1]$ and $p \in (2, \infty]$.  It will be needed in the proof of Proposition~\ref{prop:2scErrDecayL2}. However, for the sake of simplicity, we choose to display here the proof in the more regular setting as it relies on the same ideas but requires less technicity. \\
Note that ${u_0}_{|_D}$ is indeed in $W^{2, \infty}(D)$. By Proposition~\ref{prop:regHk}, since $D$ has a $\mathcal{C}^4$- boundary, and $u^{inc} \in \mathcal{C}^\infty(B_R)$, ${u_0}_{|_D}$ is in $H^4(D)$. By \cite[Corollary~9.15]{brezis2011functional}, one has the embedding $H^4(D) = W^{4, 2}(D) \hookrightarrow \mathcal{C}^2(\overline{D})$. 
\end{rem}

\begin{proof}
By \eqref{eq:H1err}, one only needs to show that a.s.
$$\norme{F_\eps}_{L^2(D)} + \norme{G_\eps}_{L^2(D)}  \lesssim \eps \mu_d(\frac{1}{\eps}) \chi_\eps \norme{u_0}_{W^{2, \infty}(D)}. $$
The definition of $F_\eps$ by \eqref{Feps} and $G_\eps$ by \eqref{Geps} implies that a.s.
\begin{equation}
\begin{split}
\norme{F_\eps}_{L^2(D)} + \norme{G_\eps}_{L^2(D)} & \lesssim \eps \left(\norme{ \phi^\eps }_{L^2(D)} + \norme{\sigma^\eps }_{L^2(D)} + \norme{ \beta^\eps}_{L^2(D)}\right) \norme{u_0}_{W^{2, \infty}(D)}.
\end{split}
\end{equation}
It then suffices to bound a.s. the norm of each corrector on the right hand side by $\mu_d(\frac{1}{\eps})$ (up to a constant) to obtain the desired estimate. \\
We prove the bound for $\norme{\phi^\eps}_{L^2(D)}$. The two other estimates for $\norme{\beta^\eps}_{L^2(D)}$ and $\norme{\sigma^\eps}_{L^2(D)}$ are established in a similar manner. \\
We pave $D$ with squares of size $\eps$, change of variable and use the a.s. corrector bounds \eqref{eq:corrector_bounds} to obtain
\begin{equation}
\begin{split}
\norme{\phi^\eps}_{L^2(D)}^2 & \lesssim \sum_{z \in P_\eps(D)} \norme{ \phi^\eps }_{L^2(\eps \square_z)}^2 \\
& = \sum_{z \in P_\eps(D)} \int_{\eps \square_z}|\phi(\frac{\cdot}{\eps})|^2 \\
& \lesssim \eps^d \sum_{z \in P_\eps(D)} \fint_{\square_z}|\phi|^2 \\
& \lesssim \mu_d(\frac{1}{\eps})^2 \left(\eps^d \sum_{z \in P_\eps(D)} \mathcal{C}(z)^2 \right) \\
& \lesssim \mu_d(\frac{1}{\eps})^2 \chi_\eps^2.
\end{split}
\end{equation}
By \cite[Lemma~A4]{armstrong2019quantitative}, $\chi_\eps$ satisfies \eqref{eq:Exp_corrector_bounds} which concludes our proof.
\end{proof}

The boundary corrector $v_\eps$, defined by \eqref{eq:veps} solves an a.s. comparable problem as the one verified by $u_\eps$ in $D$, with an oscillatory boundary data on $\partial D$. The resulting complexity drives us to derive convergence rates of the two-scale expansion error without $v_\eps$. We start with the estimate in the $H^1$-norm.
\begin{prop}[$H^1$- convergence rate of the two-scale expansion] \label{prop:2scErrDecay}
Let $u_\eps \in H^1(B_R)$ be the solution of \eqref{eq:mainBR}, ${u_0 \in H^1(B_R)}$ such that ${{u_0}_{|_D} \in W^{2, \infty}(D)}$ be the solution of \eqref{eq:u0BR} and $u_{1,\eps}$ be defined by \eqref{eq:u1}. \\ 
Then, it holds a.s.
\begin{equation} \label{eq:errH1}
\norme{u_\eps - u_0}_{H^1(B_R \setminus\overline{D})} + \norme{u_\eps - u_0 - \eps u_{1,\eps}}_{H^1(D)} \lesssim \eps^{\frac{1}{2}} \mu_d(\frac{1}{\eps})^{\frac{1}{2}} \widehat{\chi_\eps} \norme{u_0}_{W^{2, \infty}(D)},
\end{equation}
where $\widehat{\chi_\eps}$ is a random variable that satisfies the stochastic integrability \eqref{eq:Exp_corrector_bounds}. In particular 
\begin{equation}
\E\left[\norme{u_\eps - u_0}_{H^1(B_R \setminus\overline{D})}^2\right]^{\frac{1}{2}}+\E\left[\norme{u_\eps - u_0 - \eps u_{1,\eps}}_{H^1(D)}^2\right]^{\frac{1}{2}} \lesssim \eps^{\frac{1}{2}} \mu_d(\frac{1}{\eps})^{\frac{1}{2}} \norme{u_0}_{W^{2, \infty}(D)}.
\end{equation}

\end{prop}

In the rest of the article $\widehat{\chi_\eps}$ denotes a random variable satisfying the stochastic integrability \eqref{eq:Exp_corrector_bounds}. Its expression in the specific estimate \eqref{eq:errH1} is made explicit in the proof. 

\begin{proof}
Thanks to Proposition~\ref{prop:2scErrDecayBC}, we only need to estimate the norm of $v_\eps$ in $H^1(B_R~\setminus~\overline{D})~\times~H^1(D)$, and the conclusion follows by the triangle inequality. \\
We lift the trace jump of $v_\eps$ across $\partial D$ by considering 
$\tilde{v}_\eps := v_\eps - \eps \eta_\eps u_{1,\eps} $ where $\eta_\eps$ is a smooth cutoff satisfying for all $x \in B_R$
$$0 \leq \eta_\eps(x) \leq 1, \hspace{0.2cm} \eta_\eps = 0 \hspace{0.2cm} \textrm{in } D_{2 \eps \mu_d(\frac{1}{\eps})}, \hspace{0.2cm} \eta_\eps = 1\, \textrm{in } B_R \setminus \overline{D_{\eps \mu_d(\frac{1}{\eps})}}, \hspace{0.2cm} \vert \nabla \eta_\eps(x) \vert \lesssim \frac{1}{\eps \mu_d(\frac{1}{\eps})},$$
where $D_r = \{x \in B_R \, \vert \, dist(x, \partial D) > r \}$ for $r>0$. \\
By Proposition~\ref{prop:stabAbsorb_smallKR}, since $\eta_\eps u_{1, \eps} \in H^1(D)$, $\tilde{v}_\eps$ is a.s. the unique solution in $H^1(B_R)$ to
\begin{equation} \label{eq:tildeVeps}
\left\{\begin{aligned} 
&- \Delta \tilde{v}_\eps - k^2 n_0 \tilde{v}_\eps = 0 \qquad && \textrm{in}\, B_R \setminus \overline{D},\\
&-\nabla\cdot a_\eps \nabla \tilde{v}_\eps - k^2 n_\eps \tilde{v}_\eps = \eps \nabla\cdot a_\eps \nabla (\eta_\eps u_{1, \eps}) + \eps k^2 n_\eps (\eta_\eps u_{1, \eps})\qquad && \textrm{in}\, D,\\
& \nabla \tilde{v}_\eps^- \cdot \nu - a_\eps \nabla \tilde{v}_\eps^+ \cdot \nu = a_\eps \nabla (\eta_\eps u_{1, \eps}) \cdot \nu  \\ & + \eps \sum_{i=1}^d \left(\nabla \cdot(\sigma_i^\eps \partial_i u_0)^+\right) \cdot \nu - k^2 \eps (\beta^\eps u_0)^+ \cdot \nu \qquad  && \textrm{on}\, \partial D, \\
& \nabla \tilde{v}_\eps \cdot \nu = \Lambda(\tilde{v}_\eps)  \qquad && \textrm{on}\, \partial B_R. \\
 \end{aligned}
\right.
\end{equation}
On the boundary it holds that
\begin{equation}
\sum_{i=1}^d \nabla \cdot(\sigma_i^\eps \partial_i u_0)^+) \cdot \nu - k^2 \eps (\beta^\eps u_0)^+ \cdot \nu = \sum_{i=1}^d \nabla \cdot(\sigma_i^\eps \partial_i u_0 \eta_\eps)^+) \cdot \nu - k^2 \eps (\beta^\eps u_0 \eta_\eps)^+ \cdot \nu.
\end{equation}
Furthermore, for all $i \in \intInter{1,d}$, 
$$\nabla \cdot (\nabla \cdot (\sigma_i^\eps \partial_i u_0 \eta_\eps)) = \sum_{j,m=1}^d \partial_{jm} (\sigma_{i,jm}^\eps \partial_i u_0 \eta_\eps)) = 0.$$ 
Thus, $\tilde{v}_\eps$ verifies for all $w \in H^1(B_R)$
\begin{equation} \label{eq:FVtildeVeps}
\begin{split}
\int_{B_R} & a_\eps \nabla \tilde{v}_\eps \cdot \nabla \overline{w}  - k^2 n_\eps \tilde{v}_\eps \overline{w} - \left \langle \Lambda(\tilde{v}_\eps), w \right\rangle_{H^{-\frac{1}{2}}(\partial B_R), H^{\frac{1}{2}}(\partial B_R)}  \\
& =  \int_D -\eps a_\eps \nabla (\eta_\eps u_{1, \eps}) \cdot \nabla \overline{w} + \eps k^2 n_\eps (\eta_\eps u_{1, \eps})\overline{w} - \eps \sum_{i=1}^d \nabla \cdot (\sigma_i^\eps \partial_i u_0 \eta_\eps) \cdot \nabla \overline{w} \\
& \hspace{2cm} - k^2 \eps \nabla \cdot (\beta^\eps u_0 \eta_\eps) \overline{w} + k^2 \eps \beta^\eps u_0 \eta_\eps \cdot \nabla \overline{w},
\end{split}
\end{equation}
where $\overline{w}$ denotes the conjugate of $w$ and $\langle \cdot, \cdot \rangle_{H^{-\frac{1}{2}}(\partial B_R), H^{\frac{1}{2}}(\partial B_R)}$ denotes the duality product $H^{-\frac{1}{2}}(\partial B_R), H^{\frac{1}{2}}(\partial B_R)$. By the coercivity of the sesquilinear form, we then obtain
\begin{equation}
\begin{split}
\norme{\tilde{v}_\eps}_{H^1(B_R)} & \lesssim \eps \norme{\eta_\eps u_{1,\eps}}_{H^1(D)} + \eps \norme{\sum_{i=1}^d \nabla \cdot (\sigma_i^\eps \partial_i u_0 \eta_\eps)}_{L^2(D)} \\
& \hspace{1cm} + \eps \norme{\nabla \cdot (\beta^\eps u_0 \eta_\eps)}_{L^2(D)} + \eps \norme{\beta^\eps u_0 \eta_\eps}_{L^2(D)}. \\
\end{split}
\end{equation}
Let $\mathcal{S}_{\eta_\eps} := \textrm{supp}(\eta_\eps)$ denote the support of $\eta_\eps$. By definition of $\tilde{v}_\eps$, it holds
\begin{equation}
\begin{split}
& \norme{v_\eps}_{H^1(B_R \setminus \overline{D})} + \norme{v_\eps}_{H^1(D)} \\
&  \hspace*{0.7cm} \lesssim \eps \norme{\eta_\eps u_{1, \eps}}_{H^1(D)} + \eps \norme{\sum_{i=1}^d \nabla \cdot (\sigma_i^\eps \partial_i u_0 \eta_\eps)}_{L^2(D)} + \eps \norme{\nabla \cdot (\beta^\eps u_0 \eta_\eps)}_{L^2(D)} + \eps \norme{\beta^\eps u_0 \eta_\eps}_{L^2(D)} \\
&  \hspace*{0.7cm} \lesssim \left(\eps\norme{\nabla \phi^\eps}_{L^2(\mathcal{S}_{\eta_\eps})} + \eps \norme{\phi^\eps}_{L^2(\mathcal{S}_{\eta_\eps})} + \frac{1}{\mu_d(\frac{1}{\eps})}\norme{\phi^\eps}_{L^2(\mathcal{S}_{\eta_\eps})}\right.\\
& \hspace*{1.7cm} + \eps\norme{\nabla \sigma^\eps}_{L^2(\mathcal{S}_{\eta_\eps})} + \eps \norme{\sigma^\eps}_{L^2(\mathcal{S}_{\eta_\eps})} + \frac{1}{\mu_d(\frac{1}{\eps})} \norme{\sigma^\eps}_{L^2(\mathcal{S}_{\eta_\eps})} \\
& \hspace*{1.7cm} \left. + \eps\norme{\nabla \beta^\eps}_{L^2(\mathcal{S}_{\eta_\eps})} + \eps \norme{\beta^\eps}_{L^2(\mathcal{S}_{\eta_\eps})} + \frac{1}{\mu_d(\frac{1}{\eps})}\norme{\beta^\eps}_{L^2(\mathcal{S}_{\eta_\eps})} \right) \norme{u_0}_{W^{2, \infty}(D)}.
\end{split}
\end{equation}
We prove the bounds on the corrector $\phi$, the two other estimates are established in a similar manner. 
Let $\widetilde{\chi_{\eps}^1}$ and $\widetilde{\chi_\eps^2}$ be the random variables defined by
\begin{equation}
\left\{\begin{aligned} 
& \widetilde{\chi_\eps^1} := \left(\frac{\eps^d}{\eps \mu_d(\frac{1}{\eps})} \sum_{z \in P_\eps(\mathcal{S}_{\eta_\eps})} \mathcal{C}(z)^2  \right)^{\frac{1}{2}}, \\
& \widetilde{\chi_\eps^2} := \left(\frac{\eps^d}{\eps \mu_d(\frac{1}{\eps})} \sum_{z \in P_\eps(\mathcal{S}_{\eta_\eps})} (1 + r_*(z))^{2d} \right)^{\frac{1}{2}}.
\end{aligned}
\right.
\end{equation}
These two random variables satisfy the stochastic integrability \eqref{eq:Exp_corrector_bounds} since the prefactor $\frac{\eps^d}{\eps \mu_d(\frac{1}{\eps})} \approx \frac{1}{|P_\eps(\mathcal{S}_{\eta_\eps})|}$ is the appropriate renormalization.
\\
By following the proof of Proposition~\ref{prop:2scErrDecayBC}, it holds a.s.
\begin{equation} \label{eq:normphi}
\begin{split}
\eps\norme{\phi^\eps}_{L^2(\mathcal{S}_{\eta_\eps})} & \lesssim \eps^{\frac{3}{2}} \mu_d(\frac{1}{\eps})^{\frac{3}{2}} \widetilde{\chi_{\eps}^1}.
\end{split}
\end{equation}
It remains to estimate $\eps\norme{\nabla \phi^\eps}_{L^2(\mathcal{S}_{\eta_\eps})}$. Using the mean value property of Proposition~\ref{prop:corrbounds}, we get a.s.
\begin{equation} \label{eq:normnablaphi}
\begin{split}
\eps\norme{\nabla \phi^\eps}_{L^2(\mathcal{S}_{\eta_\eps})} & \lesssim \eps^{\frac{1}{2}} \mu_d(\frac{1}{\eps})^{\frac{1}{2}} \widetilde{\chi_{\eps}^2}.
\end{split}
\end{equation}
By combining \eqref{eq:normphi} and \eqref{eq:normnablaphi}, we obtain
\begin{equation} \label{eq:chi3}
\begin{split}
\eps\norme{\nabla \phi^\eps}_{L^2(\mathcal{S}_{\eta_\eps})} & + \eps \norme{\phi^\eps}_{L^2(\mathcal{S}_{\eta_\eps})} + \frac{1}{\mu_d(\frac{1}{\eps})}\norme{\phi^\eps}_{L^2(\mathcal{S}_{\eta_\eps})} \\
& \lesssim \eps^{\frac{1}{2}}\mu_d(\frac{1}{\eps})^{\frac{1}{2}}\widetilde{\chi_{\eps}^2} + \eps^{\frac{3}{2}}\mu_d(\frac{1}{\eps})^{\frac{3}{2}} \widetilde{\chi_{\eps}^1} + \eps^{\frac{1}{2}}\mu_d(\frac{1}{\eps})^{\frac{1}{2}}\widetilde{\chi_{\eps}^1}  \\
& \lesssim \eps^{\frac{1}{2}} \mu_d(\frac{1}{\eps})^{\frac{1}{2}} \widetilde{\chi_{\eps}^3},
\end{split}
\end{equation}
where we define $\widetilde{\chi_{\eps}^3}$ as
\begin{equation} \label{eq:Chi3}
\widetilde{\chi_{\eps}^3} := \widetilde{\chi_{\eps}^2} + (\eps \mu_d(\frac{1}{\eps}) + 1)\widetilde{\chi_{\eps}^1}.
\end{equation}
As $\widetilde{\chi_{\eps}^1}$ and $\widetilde{\chi_{\eps}^2}$ satisfy the stochastic integrability \eqref{eq:Exp_corrector_bounds} 
, $\widetilde{\chi_{\eps}^3}$ also satisfies \eqref{eq:Exp_corrector_bounds}. \\
This gives us the following estimate for $v_\eps$ in $H^1(B_R \setminus\overline{D}) \times H^1(D)$,
\begin{equation}
\norme{v_\eps}_{H^1(B_R \setminus \overline{D})} + \norme{v_\eps}_{H^1(D)} \lesssim \eps^{\frac{1}{2}}\mu_d(\frac{1}{\eps})^{\frac{1}{2}}\widetilde{\chi_\eps^3} \norme{u_0}_{W^{2, \infty}(D)}.
\end{equation}
Therefore, we conclude by the triangle inequality
\begin{equation}
\begin{split}
\norme{u_\eps - u_0}_{H^1(B_R \setminus\overline{D})} + &  \norme{u_\eps - u_0 - \eps u_{1,\eps}}_{H^1(D)} \\
& \lesssim \norme{u_\eps - u_0 - \eps u_{1,\eps} - v_\eps}_{H^1(B_R)} + \norme{v_\eps}_{H^1(D)} + \norme{v_\eps}_{H^1(B_R \setminus \overline{D})} \\
& \lesssim \eps^{\frac{1}{2}} \mu_d(\frac{1}{\eps})^{\frac{1}{2}}\left(\eps^{\frac{1}{2}}\mu_d(\frac{1}{\eps})^{\frac{1}{2}} \chi_\eps + \widetilde{\chi_{\eps}^3} \right) \norme{u_0}_{W^{2, \infty}(D)} \\
& \lesssim \eps^{\frac{1}{2}} \mu_d(\frac{1}{\eps})^{\frac{1}{2}} \widehat{\chi_\eps} \norme{u_0}_{W^{2, \infty}(D)}, \\
\end{split}
\end{equation}
where the random variable $\widehat{\chi_\eps}$ defined as
\begin{equation}
\widehat{\chi_\eps} := \eps^{\frac{1}{2}}\mu_d(\frac{1}{\eps})^{\frac{1}{2}} \chi_\eps +\widetilde{\chi_{\eps}^2} + (\eps \mu_d(\frac{1}{\eps}) + 1) \widetilde{\chi_{\eps}^1},
\end{equation}
satisfies the stochastic integrability \eqref{eq:Exp_corrector_bounds}. 
\end{proof}

We expect the homogenization error $u_\eps - u_0$ to be of order $O(\eps)$ in $L^2(B_R)$ as $u_0$ verifies the proper transmission conditions on $\partial D$. This is the subject of the next proposition.
\begin{prop}[$L^2$- rate of convergence of the homogenization error] \label{prop:2scErrDecayL2}
Let ${u_\eps \in H^1(B_R)}$ be the a.s. solution of \eqref{eq:mainBR} and $u_0 \in H^1(B_R)$ such that ${{u_0}_{|_D} \in W^{2, \infty}(D)}$ be the solution of \eqref{eq:u0BR}. \\
Then, it holds a.s.
\begin{equation} \label{eq:L2decayPS}
\norme{u_\eps - u_0}_{L^2(B_R)} \lesssim \eps \mu_d(\frac{1}{\eps}) \widehat{\chi_\eps}\norme{u_0}_{W^{2, \infty}(D)},
\end{equation}
where $\widehat{\chi_\eps}$ is a random variable satisfying the stochastic integrability \eqref{eq:Exp_corrector_bounds}. In particular
\begin{equation}
\E \left[\norme{u_\eps  - u_0}_{L^2(B_R)}^2\right]^{\frac{1}{2}} \lesssim \eps \mu_d(\frac{1}{\eps})\norme{u_0}_{W^{2, \infty}(D)}.
\end{equation}

\end{prop}

\begin{rem}
A similar result has been shown in the periodic case in \cite{Cakoni2016OnTH}. Though, we cannot adapt the proof, since it uses the $L^\infty$-bound of the corrector, that does not hold in the stochastic setting. In \cite[Theorem~6.14]{armstrong2019quantitative}, the result is shown for Poisson equation in the Dirichlet case. Our proof is an adaptation of the latter result.
\end{rem}

\begin{proof}
To prove \eqref{eq:L2decayPS}, we use Proposition~\ref{prop:2scErrDecayBC} and the bounds on the correctors, which imply that a.s.
\begin{equation} \label{eq:normeErrL2}
\begin{split}
\norme{u_\eps - u_0}_{L^2(B_R)} & \lesssim \norme{u_\eps - u_0 - \eps u_{1,\eps} - v_\eps}_{L^2(B_R)} + \eps \norme{u_{1,\eps}}_{L^2(D)} + \norme{v_\eps}_{L^2(B_R)}  \\
& \lesssim \eps \mu_d(\frac{1}{\eps}) \chi_\eps \norme{u_0}_{W^{2,\infty}(D)} + \eps \norme{\phi^\eps}_{L^2(D)}\norme{u_0}_{W^{2,\infty}(D)} + \norme{v_\eps}_{L^2(B_R)} \\
& \lesssim \eps \mu_d(\frac{1}{\eps}) \chi_\eps \norme{u_0}_{W^{2,\infty}(D)} + \norme{v_\eps}_{L^2(B_R)}. \\
\end{split}
\end{equation}
It remains to estimate the $L^2$-norm of the boundary corrector, which we do by using a duality argument as in \cite{Cakoni2016OnTH}. \\
Let $h \in L^2(B_R)$. We wish to estimate
$$\left\vert\int_{B_R} v_\eps \overline{h} \right\vert.$$
To do so, we introduce the auxiliary function $W_\eps \in H^1(B_R)$ solution of
\begin{equation} \label{eq:Weps}
\left\{\begin{aligned} 
&-\nabla\cdot a_\eps^* \nabla W_\eps - k^2 n_\eps W_\eps = \overline{h} &&\textrm{in}\, B_R,\\
& \nabla W_\eps \cdot \nu = \Lambda(W_\eps) &&\textrm{on}\, \partial B_R. \\ 
\end{aligned}
\right.
\end{equation}
Here $a_\eps^*$ denote the transpose of $a_\eps$. \\
We write the variational formulation verified by $W_\eps$ in $H^1(B_R)$ and choose $\overline{\tilde{v}_\eps}$ as a test function. Recall that $\tilde{v}_\eps := v_\eps - \eps \eta_\eps u_{1,\eps}$ is the unique solution in $H^1(B_R)$ to \eqref{eq:tildeVeps}. We obtain
\begin{equation}
\begin{split}
\int_{B_R} \tilde{v}_\eps \overline{h} & = \int_{B_R} a_\eps^* \nabla W_\eps \cdot \nabla \tilde{v}_\eps - k^2 n_\eps W_\eps \tilde{v}_\eps - \left \langle \Lambda(\tilde{v}_\eps), \overline{W_\eps} \right\rangle_{H^{-\frac{1}{2}}(\partial B_R), H^{\frac{1}{2}}(\partial B_R)}. \\
\end{split}
\end{equation}
We choose in the variational formulation of $\tilde{v}_\eps$ \eqref{eq:FVtildeVeps}, $\overline{W_\eps}$ as a test function and subtract the two expressions to obtain
\begin{equation}
\begin{split}
\int_{B_R} \tilde{v}_\eps \overline{h} & = \int_D -\eps a_\eps \nabla (\eta_\eps u_{1, \eps}) \cdot \nabla W_\eps + \eps k^2 n_\eps (\eta_\eps u_{1, \eps}) W_\eps - \eps \left(\sum_{i=1}^d \nabla \cdot (\sigma_i^\eps \partial_i u_0 \eta_\eps) \right) \cdot \nabla W_\eps \\
& \hspace{2cm} - k^2 \eps \nabla \cdot (\beta^\eps u_0 \eta_\eps) W_\eps + k^2 \eps \beta^\eps u_0 \eta_\eps \cdot \nabla W_\eps.
\end{split}
\end{equation}
We use the estimate \eqref{eq:chi3} and get
\begin{equation} 
\begin{split} \label{eq:normeVeps}
\left \vert \int_{B_R} v_\eps \overline{h} \right \vert & \lesssim \eps\left \vert \int_{D} \eta_\eps u_{1,\eps} \overline{h} \right \vert + \eps^{\frac{1}{2}}\mu_d(\frac{1}{\eps})^{\frac{1}{2}}\widetilde{\chi_{\eps}^3} \norme{W_\eps}_{H^1(\mathcal{S}_\eps)} \norme{u_0}_{W^{2, \infty}(D)} \\ 
& \lesssim \eps^{\frac{3}{2}}\mu_d(\frac{1}{\eps})^{\frac{3}{2}}\widetilde{\chi_{\eps}^1}  \norme{h}_{L^2(D)} + \eps^{\frac{1}{2}}\mu_d(\frac{1}{\eps})^{\frac{1}{2}}\widetilde{\chi_{\eps}^3} \norme{W_\eps}_{H^1(\mathcal{S}_\eps)} \norme{u_0}_{W^{2, \infty}(D)}. \\
\end{split}
\end{equation}
It remains to show that 
$$\norme{W_\eps}_{H^1(\mathcal{S}_{\eta_\eps})} \lesssim \eps^{\frac{1}{2}} \mu_d(\frac{1}{\eps})^{\frac{1}{2}} \widetilde{\chi_{\eps}^4} \norme{h}_{L^2(B_R)}, $$  
for a random variable $\widetilde{\chi_{\eps}^4}$ satisfying the correct stochastic integrability. 
Following \cite{Cakoni2016OnTH}, we apply homogenization results to $W_\eps$ to obtain the desired estimate. \\
We thus introduce $W_0 \in H^2(B_R \setminus \overline{D}) \times H^2(D)$ solution of
\begin{equation} 
\left\{\begin{aligned} 
& -\Delta W_0 - k^2 n_0 W_0 = \overline{h} &&\textrm{in}\, B_R \setminus\overline{D},\\
&-\nabla\cdot a^{hom} \nabla W_0 - k^2 n^{hom} W_0 = \overline{h} &&\textrm{in}\, D,\\
& \nabla W_0^- \cdot \nu - a^{hom} \nabla W_0^+ \cdot \nu = 0 &&\textrm{on}\, \partial D, \\ 
& \nabla W_0 \cdot \nu  = \Lambda(W_0)  &&\textrm{on}\, \partial B_R. \\ 
 \end{aligned}
\right.
\end{equation}
The regularity of $W_0$ comes from Proposition~\ref{prop:regHk}. Moreover the following estimate holds
\begin{equation}
\norme{W_0}_{H^2(D)} \lesssim \norme{h}_{L^2(B_R)}.
\end{equation}
However, we have no guarantee that ${W_0}_{|_D}$ is in $W^{2, \infty}(D)$, since $h$ is only in $L^2(B_R)$. Therefore we cannot apply the result of Proposition~\ref{prop:2scErrDecay}. Note that even if we could, this would yield a control with the $W^{2, \infty}(D)$ norm of $W_0$, that we cannot directly link to $\norme{h}_{L^2(B_R)}$. Instead, Proposition~\ref{prop:2scErrDecaylessRegular} gives a $W^{\frac{3}{2}, 2 + s}(D)$-control, for $s > 0$, of the two-scale expansion error. \\
By the fractional Sobolev embedding (cf. \cite[Theorem~7.58]{adams2003sobolev}), there exists an exponent $s(d) > 0$ such that we have the embedding:
$$W^{2, 2}(D) \hookrightarrow W^{\frac{3}{2},2 + s}(D). $$
In particular, this yields
\begin{equation}
\norme{W_0}_{W^{\frac{3}{2}, 2 + s}(D)} \lesssim \norme{W_0}_{H^2(D)} \lesssim \norme{h}_{L^2(B_R)}.
\end{equation}
For $\alpha = \frac{1}{2}$, $p =  2+ s > 2$, Proposition~\ref{prop:2scErrDecaylessRegular} implies then
\begin{equation}
\begin{split}
\norme{W_\eps}_{H^1(\mathcal{S}_{\eta_\eps})} & \lesssim \norme{W_\eps - W_0 - \eps W_{1,\eps}}_{H^1(\mathcal{S}_{\eta_\eps})} + \norme{W_0 + \eps W_{1,\eps}}_{H^1(\mathcal{S}_{\eta_\eps})} \\
& \lesssim \eps^{\frac{1}{2}} \mu_d(\frac{1}{\eps})^{\frac{1}{2}} \widehat{\chi_{\eps, p}} \norme{W_0}_{W^{\frac{3}{2}, p}(D)} + \norme{W_0 + \eps W_{1,\eps}}_{H^1(\mathcal{S}_{\eta_\eps})}, \\ 
\end{split}
\end{equation}
where $$W_{1,\eps} := \sum_{i=1}^d \phi_i^\eps \partial_i \widehat{W_0}*\xi_\eps,$$
and $\widehat{\chi_{\eps, p}}$ verifies \eqref{eq:Exp_corrector_bounds}. The mollifier $\xi_\eps$ is defined by \eqref{eq:mollifier} and $\widehat{W_0}$ is the Sobolev extension in $H^1(\R^d)$ of ${W_0}_{|_D}$ (cf Lemma~\ref{lem:extensionThm}). \\
Using Lemma~\ref{lem:lemma612}, with 
$f = \nabla W_0$, $r = \eps \mu_d(\frac{1}{\eps})$, $p = 2+ s$, $\alpha = \frac{1}{2}$, $q = 2$, $\beta = \frac{1}{q} = \frac{1}{2}$, we obtain
\begin{equation}
\norme{\nabla W_0}_{L^2(\mathcal{S}_{\eta_\eps})} \lesssim \eps^{\frac{1}{2}} \mu_d(\frac{1}{\eps})^{\frac{1}{2}}\norme{W_0}_{W^{\frac{3}{2}, p}(D)}.
\end{equation} 
It also holds by the combination of \eqref{eq:tildeu12} and \eqref{eq:tildeu14} that
\begin{equation}
\begin{split}
\eps \norme{\nabla W_{1,\eps}}_{L^2(\mathcal{S}_{\eta_\eps})} & \lesssim \eps^{\frac{1}{2}}\mu_d(\frac{1}{\eps})^{\frac{1}{2}} \widehat{\chi_{\eps, p}} \norme{W_0}_{W^{\frac{3}{2}, p}(D)}.
\end{split}
\end{equation} 
with 
$$\widehat{\chi_{\eps, p}} := \widetilde{\chi_{\eps, p}^2} + \eps^{\frac{1}{2} - \frac{1}{p}} \mu_d(\frac{1}{\eps})^{1-\frac{1}{p}} \widetilde{\chi_{\eps, p}^1},$$ 
and $\widetilde{\chi_{\eps, p}^1}$, $\widetilde{\chi_{\eps, p}^2}$ are defined in \eqref{eq:ChiEpsP} and satisfy \eqref{eq:Exp_corrector_bounds}. With the same arguments, similar estimates can be derived for $\norme{W_0}_{L^2(\mathcal{S}_{\eta_\eps})}$ and $\eps\norme{W_{1,\eps}}_{L^2(\mathcal{S}_{\eta_\eps})}$. This yields that
\begin{equation} \label{eq:WepsH1}
\begin{split}
\norme{W_\eps}_{H^1(\mathcal{S}_{\eta_\eps})} & \lesssim \eps^{\frac{1}{2}}\mu_d(\frac{1}{\eps})^{\frac{1}{2}} \widetilde{\chi_{\eps, p}^4} \norme{h}_{L^2(D)}, \\
\end{split}
\end{equation}
where $\widetilde{\chi_{\eps, p}^4} := 1 + \widehat{\chi_{\eps,p}} + \widetilde{\chi_{\eps,p}^1}$ satisfies \eqref{eq:Exp_corrector_bounds}. \\
We combine \eqref{eq:normeVeps} and \eqref{eq:WepsH1} to obtain
\begin{equation}
\begin{split}
\norme{v_\eps}_{L^2(B_R)} & \lesssim \eps \mu_d(\frac{1}{\eps})\left(\eps^{\frac{1}{2}}\mu_d(\frac{1}{\eps})^{\frac{1}{2}} \widetilde{\chi_{\eps}^1} + \widetilde{\chi_{\eps}^3}\widetilde{\chi_{\eps, p}^4} \right) \norme{u_0}_{W^{2, \infty}(D)}.
\end{split}
\end{equation} 
Therefore, by \eqref{eq:normeErrL2} we get 
\begin{equation}
\norme{u_\eps - u_0}_{L^2(B_R)} \lesssim \eps \mu_d(\frac{1}{\eps})\widehat{\chi_\eps} \norme{u_0}_{W^{2, \infty}(D)},
\end{equation}
where $\widehat{\chi_\eps}$ defined by
$$\widehat{\chi_\eps} := \chi_\eps + \eps^{\frac{1}{2}}\mu_d(\frac{1}{\eps})^{\frac{1}{2}} \widetilde{\chi_{\eps}^1} + \widetilde{\chi_{\eps}^3}\widetilde{\chi_{\eps, 2 + \delta}^4},$$ 
satisfies the desired stochastic integrability thanks to the following version of H\"older's inequality \cite{duerinckx2020robustness}.
\begin{lem}[H\"older's inequality] \label{lem:holder}
For all random variables $Y_1, Y_2$, given $\kappa_1, \kappa_2 > 0$,
\begin{equation} \label{eq:holder}
\begin{split}
& \textrm{if } \E[\exp(Y_1^{\kappa_1})] \leq 2 \textrm{ and } \E[\exp(Y_2^{\kappa_2})] \leq 2, \\
& \hspace{1cm} \textrm{then there exists }C > 0, \textrm{ such that }
\E\biggl[\exp\left(\frac{1}{C} (Y_1 Y_2)^{\frac{\kappa_1 \kappa_2}{\kappa_1 + \kappa_2}}\right)\biggr] < \infty.
\end{split}
\end{equation}
\end{lem}
\end{proof}

\section{Asymptotic expansion of the scattered field} \label{section:IRfield}

\subsection{Main result}
The convergence estimates that we established in the previous section provide an asymptotic expansion of the field at order $\eps$. Outside $D$, $u_\eps$ is approximated at first-order by $u_0$ according to Proposition~\ref{prop:2scErrDecay}. Physically, $u_0$ corresponds to the wave that interacts with the effective medium of parameters $a^{hom}$ and $n^{hom}$. It depends on the distribution of the scatterers as $a^{hom}$ does but it is deterministic and thus is not characteristic of one realization in a given medium. In the context of ultrasounds the measurements are usually done using the same sensor array that transmits the plane wave excitation (ultrasonic transducers can be used as transmitters and as receivers). So $u_0$ contains only the contribution from the boundary $\partial D$ while we would like to characterize the speckle field generated by the small heterogeneities. 
We are then interested in this section in obtaining the next order term in the expansion of the field outside of $D$. \\
We introduce $G_0$ the Green function associated to the homogenized problem \eqref{eq:u0BR} \textit{i.e.} $G_0$ verifies in $\mathcal{D}'(B_R)$ for all $y \in B_R$,
\begin{equation}\label{eq:G0}
\left\{\begin{aligned}
&- \Delta G_0(\cdot, y) - k^2 n_0 G_0(\cdot, y) = \delta(\cdot - y) && \textrm{in}\, B_R \setminus \overline{D}, \\
& -\nabla\cdot\left(a^{hom} \nabla G_0(\cdot, y) \right) - k^2 n^{hom} G_0(\cdot, y) = \delta(\cdot - y) && \textrm{in}\, D,\\
&G_0(\cdot, y)^- = G_0(\cdot, y)^+ && \textrm{on}\, \partial D,\\
& \nabla G_0(\cdot, y)^- \cdot \nu - a^{hom} \nabla G_0(\cdot, y)^+ \cdot \nu = 0  && \textrm{on}\, \partial D,\\
& \nabla G_0 \cdot \nu = \Lambda(G_0) &&  \textrm{on}\, \partial B_R.\\
\end{aligned} \right.
\end{equation}
For all $\alpha > 0$, we define $D^\alpha := \{x \in B_R \, \vert \, dist(x, D) < \alpha \}$. For $z \in B_R \setminus \overline{D^\alpha}$, $u_\eps$  verifies the following Lippman-Schwinger equation
\begin{equation}\label{eq:LSEQ}
\begin{split}
u_\eps(z) = u_0(z) + \int_D & (a^{hom} - a_\eps(x)) \nabla u_\eps(x) \cdot \nabla G_0(x,z) \mathrm{d}x \\
& - k^2 \int_D (n^{hom} - n_\eps(x)) u_\eps(x) G_0(x,z) \mathrm{d}x.
\end{split}
\end{equation}
We make use of the asymptotic expansion of $u_\eps$ in $H^1(D)$ and obtain for all ${z \in B_R \setminus \overline{D^\alpha}}$
\begin{equation} \label{eq:asympExp}
\begin{split} 
u_\eps(z) = u_0(z) + \int_D & \sum_{i=1}^d \biggl(a^{hom} - a_\eps(x) (e_i + \eps \nabla \phi_i^\eps(x) \biggr)\partial_i u_0(x) \cdot \nabla G_0(x,z) \mathrm{d}x \\
& - k^2 \int_D (n^{hom} - n_\eps(x)) u_0(x) G_0(x,z) \mathrm{d}x + R^\eps(z),
\end{split}
\end{equation}
where 
\begin{equation}
\begin{split}
R^\eps(z) := \int_D & (a^{hom} - a_\eps(x) \left(\nabla u_\eps(x) - \sum_{i=1}^d (e_i + \eps \nabla \phi_i^\eps(x))\partial_i u_0(x) \right) \cdot \nabla G_0(x,z) \mathrm{d}x \\
& - k^2 \int_D (n^{hom} - n_\eps(x)) (u_\eps(x)- u_0(x)) G_0(x,z) \mathrm{d}x.
\end{split} 
\end{equation}
Using the strong convergence estimates established in Proposition~\ref{prop:2scErrDecay} and Proposition~\ref{prop:2scErrDecayL2} leads to 
$$\E[\norme{R^\eps(z)}_{L^2(B_R \setminus \overline{D^\alpha})}] \lesssim \eps^{\frac{1}{2}}\mu_d(\frac{1}{\eps})^{\frac{1}{2}}, $$
which is not sufficient since $u_\eps - u_0$ is of order $\eps$ in $L^2(B_R \setminus \overline{D})$. We thus need to estimate more sharply the weak convergence of the two quantities $(a^{hom} - a_\eps(x)) \left(\nabla u_\eps - \sum_{i=1}^d (e_i +\eps \nabla \phi_i^\eps)\partial_i u_0 \right)$ and $(n^{hom} - n_\eps) (u_\eps - u_0)$. \\
In \cite{duerinckx2020structure} and \cite{duerinckx2020robustness}, the authors study the fluctuations of $\nabla u_\eps$ in the context of the Poisson equation in $\R^d$. They prove that the fluctuations of $\nabla u_\eps$ and $a_\eps \nabla u_\eps$ can be recovered from the fluctuations of the commutator $\Xi \in L^2_{loc}(\R^d)^d$ defined by:
\begin{equation} \label{eq:Xi}
\forall i \in \intInter{1,d}, \, \Xi_i := (a - a^{hom})(e_i + \nabla \phi_i) .
\end{equation}
This leads to estimate
\begin{equation}
\begin{split}
\mathcal{R}^\eps := \eps^{-\frac{d}{2}}\int_{\R^d} (a_\eps - a^{hom}) \nabla u_\eps \cdot g - \sum_{i=1}^d \Xi_i(\frac{\cdot}{\eps}) \partial_i u_0 \cdot g
\end{split}
\end{equation}
for all $g \in \mathcal{C}^\infty_c(\R^d)^d$. They show that, for all $g \in \mathcal{C}^\infty_c(\R^d)^d$
\begin{equation} \label{eq:decayMitia}
\begin{split}
\Var[\mathcal{R}^\eps]^{\frac{1}{2}} 
& \lesssim \eps \mu_d(\frac{1}{\eps}).
\end{split}
\end{equation}
We extend this result to our situation where the Poisson equation is replaced by the Helmholtz equation leading to a second term in $R^\eps$ and where we have to take into account the boundary of $D$ as the support of $G_0$ is not compactly supported in $D$. We deal with this last point in a similar manner as in Section~\ref{section:2scErr} by introducing the appropriate boundary layer. However the rate of convergence is now $1/2$ order smaller. \\
Our main result is stated in the following theorem.
\begin{thm}[Pointwise convergence of $R^\eps$] \label{thm:Rdecay}
Let $u_\eps \in H^1(B_R)$ be the a.s. solution of \eqref{eq:mainBR}, ${u_0 \in H^1(B_R)}$ such that ${{u_0}_{|_D} \in W^{2, \infty}(D)}$ be the solution of \eqref{eq:u0BR} and for $y\in B_R \setminus \overline{D^\alpha}$, let ${G_0(\cdot, y) \in H^1(B_R \setminus \{y\})}$ such that  ${G_0(\cdot, y)_{_D} \in W^{2, \infty}(D)}$ be the solution of \eqref{eq:G0}. Define $\mathcal{U}_1 \in H^1(B_R \setminus \{y\})$ as:
\begin{equation}
\begin{split}
& \mathcal{U}_1 := \E[u_\eps - u_0]  + \sum_{i=1}^d \int_{D} (a^{hom} - a_\eps(x)) (e_i + \eps \nabla \phi_i^\eps(x)) \partial_i u_0(x) \cdot \nabla G_0(x,\cdot)\mathrm{d}x \\ 
& \hspace{3cm} - k^2 \int_D (n^{hom} - n_\eps(x)) u_0(x) G_0(x, \cdot)\mathrm{d}x.
\end{split}
\end{equation} 
Then
\begin{equation} \label{eq:L2decayE}
\begin{split} 
& \E\left[ \left\vert u_\eps(y) - u_0(y) - \mathcal{U}_1 (y) \right\vert^2 \right]^{\frac{1}{2}}  \\ 
& \hspace{3cm}\lesssim \eps^{\frac{d+1}{2}} \mu_d(\frac{1}{\eps})^{\frac{1}{2}} \norme{u_0}_{W^{2, \infty}(D)} \norme{G_0(\cdot, y)}_{W^{2, \infty}(D)},
\end{split}
\end{equation}
and if we further assume that $x \mapsto G(x, y)$ is in $W^{3, \infty}(D)$ for $y \in B_R \setminus \overline{D^\alpha}$, then
\begin{equation} \label{eq:H1decayE}
\begin{split} 
& \E\left[ \left\vert \nabla u_\eps(y) - \nabla u_0(y) - \nabla \mathcal{U}_1 (y) \right\vert^2 \right]^{\frac{1}{2}}  \\ 
& \hspace{3cm}\lesssim \eps^{\frac{d+1}{2}} \mu_d(\frac{1}{\eps})^{\frac{1}{2}} \norme{u_0}_{W^{2, \infty}(D)} \norme{G_0(\cdot, y)}_{W^{3, \infty}(D)}.
\end{split}
\end{equation}
\end{thm}

\begin{rem}
Note that for $y \in B_R \setminus \overline{D^\alpha}$, ${G_0}(\cdot ,y)_{|_D}$ belongs to $W^{2, \infty}(D)$ in view of Proposition~\ref{prop:regHk}. The regularity $G_0(\cdot, y)_{|_D} \in W^{3, \infty}(D)$ can be obtained by assuming that the boundary of $D$ is $\mathcal{C}^5$ by the Sobolev embeddings \cite[Corollary~9.15]{brezis2011functional}.
\end{rem}

\begin{cor}[$L^2$- and $H^1$convergence of $R^\eps$] \label{cor:R1decay}
For all $y \in B_R \setminus \overline{D^\alpha}$,
\begin{equation} 
\begin{split} 
& \E\left[ \norme{u_\eps - u_0 - \mathcal{U}_1}^2_{L^2(B_R \setminus \overline{D^\alpha})} \right]^{\frac{1}{2}}  \\ 
& \hspace{3cm}\lesssim_{\alpha} \eps^{\frac{d+1}{2}} \mu_d(\frac{1}{\eps})^{\frac{1}{2}} \norme{u_0}_{W^{2, \infty}(D)} \left( \int_{B_R \setminus \overline{D^\alpha}} \norme{G_0(\cdot, y)}_{W^{2, \infty}(D)}^2 \mathrm{d}y \right)^{\frac{1}{2}},
\end{split}
\end{equation}
and if we further assume that $G_0(\cdot,y)_{|_D}$ is in $W^{3, \infty}(D)$ for $y \in B_R \setminus \overline{D^\alpha}$, then
\begin{equation}
\begin{split} 
& \E\left[ \norme{u_\eps - u_0 - \mathcal{U}_1}^2_{H^1(B_R \setminus \overline{D^\alpha})} \right]^{\frac{1}{2}}  \\ 
& \hspace{3cm}\lesssim_{\alpha} \eps^{\frac{d+1}{2}} \mu_d(\frac{1}{\eps})^{\frac{1}{2}} \norme{u_0}_{W^{2, \infty}(D)} \left(\int_{B_R \setminus \overline{D^\alpha}} \norme{G_0(\cdot, y)}_{W^{3, \infty}(D)}^2 \mathrm{d}y \right)^{\frac{1}{2}}.
\end{split}
\end{equation}
\end{cor}

Moreover we denote by $a^*$ the transpose of $a$ and $\phi^*$, $\sigma^*$ the adjoint correctors that solves respectively \eqref{eq:phi} and \eqref{eq:extended_corrector} with $a^*$ instead of $a$. Finally, we write $\phi^{\eps,*} := \phi^*(\frac{\cdot}{\eps})$ and $\sigma^{\eps,*} := \sigma^*(\frac{\cdot}{\eps})$. \\
Note that, from \eqref{eq:asympExp}, for all $y \in B_R \setminus \overline{D^\alpha}$,
\begin{equation} 
\begin{split} 
u_\eps(y) - u_0(y) - \mathcal{U}_1(y) = \\ &\hspace{-1.4cm} \int_{D} (a^{hom} - a_\eps(x)) (\nabla u_\eps(x) - \sum_{i=1}^d (e_i + \eps \nabla \phi_i^\eps(x) \partial_i u_0(x)) \cdot \nabla G_0(x,y) \mathrm{d}x \\
&\hspace{-0.3cm} - k^2 \int_D (n^{hom} - n_\eps(x)) (u_\eps(x) - u_0(x)) G_0(x, y)\mathrm{d}x - \E[u_\eps - u_0] \\
& = R^\eps (y) - \E[R^\eps (y)].
\end{split}
\end{equation}
We follow the strategy of \cite{duerinckx2020robustness}, to show that
\begin{equation} \label{eq:mainBRCommutator}
\begin{split} 
\Var\biggl[R_\eps(y)\biggr] & \lesssim \eps^{d + 1} \mu_d(\frac{1}{\eps}) \norme{u_0}_{W^{2, \infty}(D)}^2 \norme{G_0(\cdot, y)}_{W^{2, \infty}(D)}^2,
\end{split}
\end{equation}
which will yield the desired result by integrating over $y$. \\
In \cite{duerinckx2020robustness}, three main tools are used to show \eqref{eq:decayMitia}:
\begin{itemize}
\item the \textit{multiscale functional inequality} Hypothesis~\ref{hyp:mix_hyp} that also holds here.
\item the bounds on the corrector (Proposition~\ref{prop:corrbounds}) and the convergence of the two-scale expansion (without the boundary corrector) that we showed in Proposition~\ref{prop:2scErrDecay}
\item the large-scale (weighted) Calder\'on-Zygmund estimates stated in  \cite{gloria2014regularity}. 
\end{itemize}
In our configuration, we can use the two first tools. However the large-scale Calder\'on-Zygmund estimates were developed for the Poisson equation, not for the Helmholtz equation. Instead of deriving similar estimates for Helmholtz equation, we take advantage of the boundness of the our domain $D$ to establish the following Lemma~\ref{lem:CSestimates}.

\begin{lem} \label{lem:CSestimates}
\begin{enumerate}[label=(\alph*)]
\item There exists a constant $C$ depending only on $d$ such that, for any $U \in L^1(D)$ and $t>0$, 
\begin{equation} \label{eq:CZ-1}
\int_{\R^d} \left(\int_{B_t(x)\cap D} |U|\right) \mathrm{d}x \leq C t^d \int_{D} |U|.
\end{equation}

\item For $T > 0$, let $\rho_T(x)$ be the radial weight:
$$\rho_T(x) := \frac{x}{T} + 1.$$
Then, for $U \in L^1(D)$ and $\alpha > 0$,
\begin{equation} \label{eq:CZ-2}
\int_{\R^d} \rho_T(x)^\alpha \left(\int_{B_t(x)\cap D} |U|\right)^2 \mathrm{d}x \leq C \sup_{y \in D}\biggl(\frac{t + y}{T} + 1\biggr)^{\alpha} t^d  \left( \int_{D} |U| \right)^2.
\end{equation}
\end{enumerate}
\end{lem}
The proof can be found in Appendix~\ref{sec:lemCSestimates}. \\

The first step of the proof of Theorem~\ref{thm:Rdecay} consists in applying the mixing condition Hypothesis~\ref{hyp:mix_hyp} to $R^\eps(z)$ for $z \in B_R \setminus \overline{D^\alpha}$. 
To simplify notations, we introduce
\begin{equation}
\begin{split}
\mathcal{P}(S) := \int_{D} &(a^{hom} - a_\eps(x)) \biggl(\nabla u_\eps(x) - (e_i + \nabla \phi_i(\frac{x}{\eps})) \partial_i u_0(x) \biggr) \cdot \nabla g(x) \mathrm{d}x \\
& - k^2 \int_D (n^{hom} - n_\eps(x)) (u_\eps(x) - u_0(x)) g(x) \mathrm{d}x. \\
\end{split}
\end{equation}
where $a_\eps := a_M + (a_S - a_M)\mathbbm{1}_{S^\eps}$, $n_\eps := n_M + (n_S - n_M)\mathbbm{1}_{S^\eps}$ and $g \in W^{3, \infty}(D)$. \\
By definition, we have then $\mathcal{P}(S) = R^\eps(z)$ if $g(\cdot) = G_0(\cdot, z)$, and $\mathcal{P}(S) = \partial_i R^\eps(z)$ if $g(\cdot) = \partial_i G_0(\cdot, z)$, if $\partial_i G_0(\cdot, z)_{|_D} \in W^{3, \infty}(D)$ where the derivative applies to the second variable. \\
We introduce some additional notations before considering $\partial^{osc} \mathcal{P}(S)$. Let $\ell \geq 1$ and $x \in \R^d$. Let $S$ be a given realization of the scatterer process. We consider another distribution of scatterers $S'$ satisfying the assumptions of Section~\ref{subsec:scatterers} and such that $S \cap (\R^d \setminus B_{\ell}(x)) = S' \cap (\R^d \setminus B_{\ell}(x))$. We name $\mathcal{A}_{\ell}(x) := \{S'~|S~\cap (\R^d \setminus B_{\ell}(x)) = S' \cap (\R^d \setminus B_{\ell}(x))\}.$ \\
For any $S$-dependent measurable random variable $F$, we denote by $F'$ and $\delta F$ the random variables:
$$F' := F(S'),$$
$$\delta F := F(S') - F(S) := F' - F.$$
By definition,
\begin{equation}
|\partial_{S, B_\ell(x)}^{osc}\mathcal{P}(S)| \lesssim \sup_{ S' \in \mathcal{A}_\ell(x)} |\delta \mathcal{P}|
\end{equation}
Here, the notation $\phi^{\eps, \prime}$ stands for $\phi'(\frac{\cdot}{\eps})$. \\
The proof is split into two. We start by deriving a representation formula for $\delta \mathcal{P}$. We then bound each term of the representation formula to get our estimate. 

\subsection{Representation formula for $\delta \mathcal{P}$}

\begin{lem}[Representation formula for $\delta \mathcal{P}$] \label{lem:repreformula}

For $ g \in W^{3, \infty}(D)$,
\begin{equation} \label{eq:repreFormula}
\begin{split}
\delta \mathcal{P} &  =  - \sum_{j=1}^d \int_D \partial_j g (\eps \nabla \phi_j^{\eps,*} + e_j) \cdot \delta a_\eps (\nabla u_\eps' - \sum_{i=1}^d (e_i + \eps \nabla \phi_i^{\eps,\prime}) \partial_i u_0) \\
& \hspace*{0.5cm} +\sum_{j=1}^d  \int_{D} - (\eps \phi_j^{\eps, *} \nabla \partial_j g + \nabla r_j)  \cdot \delta a_\eps \nabla u_\eps' + k^2  \delta n_\eps u_\eps'(r_j + \eps \phi_j^{\eps, *} \partial_j g) \\
& \hspace{0cm} + \sum_{i,j=1}^d\int_{\R^d} (\eps \phi_j^{\eps, *} \nabla \partial_j g \partial_i u_0 \mathbbm{1}_D - \eps \nabla \cdot(\eta_\eps \phi_j^{\eps,*} \partial_j g \partial_i u_0) \mathbbm{1}_D + \nabla R_{ij}) \cdot \delta a_\eps (\eps \nabla \phi_i^{\eps,\prime} + e_i) \\
& \hspace*{0.5cm} + \int_D k^2 \delta n_\eps (u_\eps' - u_0) g - \eps k^2 \beta^{\eps} \delta  u_\eps  \cdot \nabla g \\
\end{split}
\end{equation}
where for $j \in \intInter{1,d}$, $r_j := -\eps  \partial_j g \phi_j^{\eps, *} \mathbbm{1}_D + \tilde{r}_j$ and $\tilde{r}_j$ is the a.s. unique solution in $H^1(B_R)$ of:
\begin{equation} \label{eq:rEps}
\left\{\begin{aligned}
&-\Delta \tilde{r}_j - k^2 n_0 \tilde{r}_j = 0 && \textrm{in}\, B_R \setminus \overline{D}, \\
& -\nabla\cdot\left(a_\eps^* \nabla \tilde{r}_j \right) - k^2 n_\eps \tilde{r}_j = - \eps \nabla \cdot ((a_\eps^* \phi_j^{\eps, *} - \sigma_j^{\eps, *}) \nabla \partial_j g) \\
& \hspace{5cm}+ \eps k^2 \nabla \cdot (\beta^{\eps} g) && \textrm{in}\, D,\\
& \nabla \tilde{r}_j^- \cdot \nu - a_\eps^* \nabla \tilde{r}_j^+ \cdot \nu = -\eps a_\eps^* \phi_j^{\eps, *} \nabla \partial_j g \cdot \nu + \eps (\nabla \cdot \sigma_j^{\eps,*}) \partial_j g^+ \cdot \nu && \textrm{on}\, \partial D,\\
& \nabla \tilde{r}_j \cdot \nu = \Lambda(\tilde{r}_j) && \textrm{on}\, \partial B_R,\\
\end{aligned} \right.
\end{equation} 
and for $i,j \in \intInter{1,d}$, $R_{ij}$ is the a.s. unique solution in $\dot{H}(\R^d) := \{v \in H^1_{loc}(\R^d)~|~\nabla v \in L^2(\R^d)\} / \R$ of:
\begin{equation} \label{eq:rij}
\left\{\begin{aligned}
&-\nabla \cdot a(\frac{\cdot}{\eps})^* \nabla R_{ij} = 0 && \textrm{in}\, \R^d \setminus \overline{D}, \\
& -\nabla\cdot a_\eps^* \nabla R_{ij}  = -\eps \nabla \cdot (a_\eps^* \phi_j^{\eps, *} - \sigma_j^{\eps, *}) \nabla(\partial_j g \partial_i u_0) \\
& \hspace{3cm} + \nabla \cdot  a_\eps^* \nabla(\eta_\eps \phi_j^{\eps, *} \partial_j g \partial_i u_0 )  && \textrm{in}\, D,\\
& a_\eps^* \nabla R_{ij}^- \cdot \nu - a_\eps^* \nabla R_{ij}^+ \cdot \nu = -\eps (a_\eps^* \phi_j^{\eps, *} - \sigma_j^{\eps, *}) \nabla(\partial_j g \partial_i u_0) \cdot \nu \\
& \hspace{5cm} + \eps \nabla \cdot (\sigma_j^{\eps, *} \partial_j g \partial_i u_0) \cdot \nu \\
& \hspace{5cm} + \eps a_\eps^* \nabla(\phi_j^{\eps, *} \partial_j g \partial_i u_0 ) \cdot \nu  && \textrm{on}\, \partial D.\\
\end{aligned} \right.
\end{equation}

\end{lem}

\begin{rem}
Note that by the divergence theorem \cite[theorem 3.24]{monk2003finite}, since $\nabla \cdot (\nabla \cdot \sigma_j) = 0$ for all $j \in \intInter{1,d}$, we have that the normal trace $(\nabla \cdot \sigma_j) \cdot \nu \in H^{-1/2}(\partial D)$.
\end{rem}

\begin{proof}[Proof of the Lemma] 

By direct computation, 
\begin{equation}
\begin{split}
\delta \mathcal{P} & =  -\int_D \delta a_\eps \left(\nabla u_\eps' - \sum_{i=1}^d (e_i + \eps \nabla \phi_i^{\eps, \prime})\partial_i u_0  \right) \cdot \nabla g \\
& + \int_D (a^{hom} - a_\eps) \delta \left(\nabla u_\eps - \sum_{i=1}^d (e_i + \eps \nabla \phi_i^\eps)\partial_i u_0\right) \cdot \nabla g \\
& \hspace{1cm} + \int_D k^2 \delta n_\eps(u_\eps' - u_0) - k^2 (n^{hom} - n_\eps) \delta u_\eps g
\end{split}
\end{equation}
First notice that for $j \in \intInter{1,d}$, $(a^{hom} - a_\eps)e_i$ can be rewritten as: 
$$(a^{hom} - a_\eps)^*e_j = \eps a_\eps \nabla \phi_j^{\eps,*} - \eps \nabla \cdot \sigma_j^{\eps, *}.$$
Moreover $\delta \phi_i$ verifies in $\R^d$
$$ - \nabla \cdot a \nabla \delta \phi_i = \nabla \cdot \delta a(\nabla \phi_i' + e_i) $$
and $\delta u_\eps$ is a.s. the unique solution in $H^1(B_R)$ of:
\begin{equation} \label{eq:deltaUeps}
\left\{\begin{aligned}
&-\Delta \delta u_\eps - k^2 n_0 \delta u_\eps = 0 && \textrm{in}\, \R^d \setminus \overline{D}, \\
& -\nabla\cdot\left(a_\eps \nabla \delta u_\eps \right) - k^2 n_\eps \delta u_\eps = \nabla \cdot (\delta a_\eps \nabla u_\eps') + k^2 \delta n_\eps u_\eps' && \textrm{in}\, D,\\
& \nabla \delta u_\eps^- \cdot \nu - a_\eps \nabla \delta u_\eps^+ \cdot \nu = \delta a_\eps \nabla {u_\eps'}^+ \cdot \nu  && \textrm{on}\, \partial D,\\
& \nabla \delta u_\eps \cdot \nu = \Lambda(\delta u_\eps)  && \textrm{on}\, \partial B_R.\\
\end{aligned} \right.
\end{equation}
We thus get for $i,j \in \intInter{1,d}$
\begin{equation}
\begin{split}
\nabla & \phi_j^{\eps, *} \cdot a_\eps (\nabla \delta u_\eps - \nabla \delta \phi_i^\eps) \\
& = \nabla \cdot (\phi_j^{\eps, *} a_\eps (\nabla \delta u_\eps - \nabla \delta \phi_i^\eps)) - \phi_j^{\eps, *} \nabla \cdot(a_\eps (\nabla \delta u_\eps - \nabla \delta \phi_i^\eps) ) \\
& = \nabla \cdot (\phi_j^{\eps, *} a_\eps(\nabla \delta u_\eps - \nabla \delta \phi_i^\eps))  \\
& \hspace{1cm}+ \phi_j^{\eps, *} (\nabla \cdot (\delta a_\eps \nabla u_\eps') + k^2 \delta n_\eps u_\eps' + k^2 n_\eps \delta u_\eps) \\
& \hspace{1cm} - \phi_j^{\eps, *} \nabla \cdot ( \delta a_\eps (\nabla \phi_i^{\eps, \prime} + e_i)) \\
& = \nabla \cdot (\phi_j^{\eps, *} a_\eps(\nabla \delta u_\eps - \nabla \delta \phi_i^\eps)) \\
& \hspace{1cm}+ \nabla \cdot (\phi_j^{\eps, *} \delta a_\eps \nabla u_\eps') - \nabla \phi_j^{\eps, *} \cdot \delta a_\eps \nabla u_\eps' + k^2 \delta n_\eps u_\eps' \phi_j^{\eps, *} + k^2 n_\eps \delta u_\eps \phi_j^{\eps, *} \\
& \hspace{1cm} - \nabla \cdot (\phi_j^{\eps, *} \delta a_\eps (\nabla \phi_i^{\eps,\prime} + e_i)) + \nabla \phi_j^{\eps, *} \cdot \delta a_\eps (\nabla \phi_i^{\eps, \prime} + e_i).
\end{split}
\end{equation}
By skew-symmetry, it also holds for $i,j \in \intInter{1,d}$
$$(\nabla \cdot \sigma_j^{\eps, *}) \cdot \nabla  (\delta  u_\eps - \nabla \delta \phi_i^\eps) =  - \nabla \cdot \left(\sigma_j^{\eps, *} (\nabla \delta u_\eps - \nabla \delta \phi_i^\eps)\right).$$
Similarly $n^{hom} - n_\eps = - \eps \nabla \cdot \beta^\eps$ and thus,
$$(n^{hom} - n_\eps) \delta u_\eps = - \eps \nabla \cdot (\beta^\eps \delta u_\eps) + \eps \beta^\eps \cdot \nabla \delta u_\eps. $$
Therefore we obtain
\begin{equation}
\begin{split}
\delta \mathcal{P} & = \sum_{j=1}^d - \int_D (\eps \nabla \phi_j^{\eps, *} + e_j) \partial_j g \cdot \delta a_\eps (\nabla u_\eps' - \sum_{i=1}^d (\eps \nabla \phi_i^{\eps,\prime} + e_i)\partial_i u_0) \\
& \hspace{0.5cm} - \eps \int_D \phi_j^{\eps, *} \nabla \partial_j g \cdot \delta a_\eps \nabla u_\eps'  + \eps \int_{\partial D} \phi_j^{\eps, *} \partial_j g \cdot \delta a_\eps \nabla u_\eps' \cdot \nu \\  
& \hspace{0.5cm} + \eps \int_D  k^2 \delta n_\eps u_\eps' \phi_j^{\eps, *} \partial_j g + k^2 \int_D \delta n_\eps (u_\eps' - u_0) g \\
& \hspace{0.5cm} + \sum_{i=1}^d \eps \int_D \phi_j^{\eps, *} \nabla (\partial_j g \partial_i u_0) \cdot \delta a_\eps (\eps \nabla \phi_i^{\eps,\prime} + e_i) \\
& \hspace{0.5cm} -\eps \int_{\partial D} \phi_j^{\eps, *} \partial_j g \partial_i u_0 \delta a_\eps (\nabla \phi_i^{\eps,\prime} + e_i) \cdot \nu \\
& \hspace{0.5cm} + \eps^2 \int_D \nabla (\partial_j g \partial_i u_0) \cdot (\phi_j^{\eps, *} a_\eps + \sigma_j^{\eps, *}) \nabla \delta \phi_i^\eps \\
& \hspace{0.5cm} - \eps^2 \int_{\partial D} \partial_j g \partial_i u_0 (\phi_j^{\eps, *} a_\eps + \sigma_j^{\eps, *} ) \nabla \delta \phi_i^{\eps} \cdot \nu \\
& \hspace{0.5cm} - \eps \int_D \nabla \partial_j g \cdot (a_\eps \phi_j^{\eps, *} + \sigma_j^{\eps, *}) \nabla \delta u_\eps + \eps \int_{\partial D} \partial_j g (a_\eps \phi_j^{\eps, *} + \sigma_j^{\eps, *}) \nabla \delta u_\eps \cdot \nu \\
& \hspace{0.5cm} + \eps \int_D k^2 n_\eps \delta u_\eps \phi_j^{\eps, *} \partial_j g - \eps \int_D k^2 \beta^\eps \cdot \nabla \delta u_\eps g \\
& \hspace{0.5cm} - \eps \int_D k^2 \beta^\eps \delta u_\eps \cdot \nabla g + \eps \int_{\partial D} k^2  g \beta^\eps \delta u_\eps \cdot \nu. \\
\end{split}
\end{equation}
We simplify the terms depending on $\delta u_\eps$ by introducing the adjoint problem \eqref{eq:rEps}. By Proposition~\ref{prop:stabAbsorb_smallKR}, since $\phi_j^{\eps, *}$, $\sigma_j^{\eps, *}$, $\beta^\eps \in H^1_{loc}(\R^d)$ and $g \in H^2(D)$, there exists a unique solution $\tilde{r}_j$ to \eqref{eq:rEps}. \\
For $h \in H^1(B_R)$, $\tilde{r}_j$ verifies:    
\begin{equation} \label{eq:repsFV}
\begin{split}
\int_{B_R \setminus \overline{D}} \nabla \tilde{r}_j & \cdot \nabla \overline{h} - k^2 n_0 \tilde{r}_j \overline{h} - \left \langle \Lambda(\tilde{r}_j), h \right\rangle_{H^{-\frac{1}{2}}(\partial B_R), H^{\frac{1}{2}}(\partial B_R) }  + \int_D a_\eps^* \nabla \tilde{r}_j \cdot \nabla \overline{h} - k^2 n_\eps \tilde{r}_j \overline{h}\\
& = \eps \int_D ( a_\eps^* \phi_j^{\eps, *} - \sigma_j^{\eps, *}) \nabla \partial_j g \cdot \nabla \overline{h} - k^2 \beta^{\eps} g \cdot \nabla \overline{h} \\
& \hspace{1cm} + \eps \int_{\partial D} \sigma_j^{\eps, *} \partial_j g \nabla \overline{h}\cdot \nu + k^2 \beta^\eps g \cdot \nu \overline{h}. \\
\end{split}
\end{equation}
Note that we used the skew-symmetry of $\sigma_j$ to get the integration by parts
\begin{equation} \label{eq:IPP_sigma}
\left\langle \sigma_j^{\eps, *}, \nabla (\partial_j g \overline{h})\right\rangle_{-\frac{1}{2},\frac{1}{2}}= - \left\langle (\nabla \cdot \sigma_j^{\eps, *})\cdot\nu,   \partial_j g \overline{h} \right\rangle_{-\frac{1}{2},\frac{1}{2}}
\end{equation}
Moreover $\delta u_\eps$ verifies for $h \in H^1(B_R)$
\begin{equation}
\begin{split}
\int_{B_R \setminus \overline{D}} \nabla \delta u_\eps & \cdot \nabla \overline{h} - k^2 n_0 \delta u_\eps \overline{h} - \left \langle \Lambda(\delta u_\eps), h \right\rangle_{H^{-\frac{1}{2}}(\partial B_R), H^{\frac{1}{2}}(\partial B_R) } \\ & \hspace{-1cm} + \int_D a_\eps \nabla \delta u_\eps \cdot \nabla \overline{h} - k^2 n_\eps \delta u_\eps \overline{h}
 = \int_D - \delta a_\eps \nabla u_\eps' \cdot \nabla \overline{h} + k^2 \delta n_\eps u_\eps' \overline{h},
\end{split}
\end{equation}
and
\begin{equation}
\begin{split}
- \int_D a_\eps & \nabla \delta u_\eps \cdot \nabla (\phi_j^{\eps, *} \partial_j g) + \int_D k^2 n_\eps \delta u_\eps \phi_j^{\eps, *} \partial_j g\\
&  = - \int_{\partial D} a_\eps \phi_j^{\eps, *} \partial_j g \nabla \delta u_\eps \cdot \nu + \int_D \delta a_\eps \nabla u_\eps' \cdot \nabla (\phi_j^{\eps, *} \partial_j g) \\ & \hspace{1cm} - \int_D k^2 \delta n_\eps u_\eps' \phi_j^{\eps, *} \partial_j g 
- \int_{\partial D} \delta a_\eps \nabla u_\eps' \cdot \nu \phi_j^{\eps, *} \partial_j g.
\end{split}
\end{equation}
Therefore $\delta u_\eps$ satisfies
\begin{equation} \label{eq:deltaUepsFV}
\begin{split}
\int_{B_R \setminus \overline{D}}&  \nabla \delta u_\eps \cdot \nabla r_j - k^2 n_0 \delta u_\eps r_j  - \left \langle \Lambda(\delta u_\eps), \overline{r_j} \right\rangle_{H^{-\frac{1}{2}}(\partial B_R), H^{\frac{1}{2}}(\partial B_R) }  \\
& \hspace{-0.5cm}+ \int_D a_\eps \nabla \delta u_\eps \cdot \nabla r_j - k^2 n_\eps \delta u_\eps r_j 
=\int_D - \delta a_\eps \nabla u_\eps' \cdot \nabla r_j + k^2 \delta n_\eps u_\eps' r_j \\& \hspace{1cm} - \int_{\partial D} a_\eps \phi_j^{\eps, *} \partial_j g \nabla \delta u_\eps \cdot \nu - \int_{\partial D} \delta a_\eps \nabla u_\eps' \cdot \nu \phi_j^{\eps, *} \partial_j g.
\end{split}
\end{equation}
We combine \eqref{eq:repsFV} for $\overline{h} = \delta u_\eps$ and \eqref{eq:deltaUepsFV} to get
\begin{equation}
\begin{split}
- \int_D \delta a_\eps & \nabla u_\eps' \cdot \nabla r_j + k^2 \delta n_\eps u_\eps' r_j 
= \eps \int_{\partial D} \phi_j^{\eps, *} \partial_j g^+\delta a_\eps \nabla u_\eps' \cdot \nu 
 \\& \hspace{-1cm} - \eps \int_D \nabla \partial_j g \cdot (a_\eps \nabla \phi_j^{\eps, *} + \sigma_j^{\eps, *})\nabla \delta u_\eps + \eps \int_{\partial D} \partial_j g (a_\eps \nabla \phi_j^{\eps, *} + \sigma_j^{\eps, *})\nabla u_\eps \cdot \nu\\
&  \hspace{-1cm }+ \eps \int_D k^2 \delta u_\eps \phi_j^{\eps, *} \partial_j g - \eps \int_D k^2 \beta^{\eps} \cdot \nabla \delta u_\eps g +\eps \int_{\partial D} k^2 g \beta^\eps \delta u_\eps \cdot \nu.
\end{split}
\end{equation}
We deal now with the terms depending on $\delta \phi_i$. $R_{ij}$ satisfies for all $h \in \dot{H}(\R^d)$,
\begin{equation} \label{eq:FVri}
\begin{split}
\int_{\R^d}a(\frac{\cdot}{\eps})^* \nabla R_{ij} \cdot \nabla \overline{h}  \\ & \hspace{-1cm}=\eps \int_D (a_\eps^* \phi_j^{\eps, *} - \sigma_j^{\eps, *}) \nabla(\partial_j g \partial_i u_0) \cdot \nabla \overline{h} - (\nabla \cdot (\eta_\eps \sigma_j^{\eps, *} \partial_j g \partial_i u_0)) \cdot \nabla \overline{h} \\
&  -\eps\int_D a_\eps^* \nabla (\eta_\eps \phi_j^{\eps, *} \partial_j g \partial_i u_0) \cdot \nabla \overline{h} 
\end{split}
\end{equation} 
where we used that $\nabla \cdot \left(\nabla \cdot (\eta_\eps  \sigma_j^{\eps, *} \partial_j g \partial_i u_0)\right) = 0$ by the skew-symmetry of $\sigma_j$. Subsequently the sesquilinear and linear form associated to \eqref{eq:rij} are respectively coercive and continuous in $\dot{H}(\R^d)$ equipped with the semi-norm $\norme{\cdot}_{\dot{H}(\R^d)} = \norme{\nabla \cdot}_{L^2(\R^d)}$ (cf \cite[Chapter~2.5]{nedelec2001acoustic} for more details). Moreover,
\begin{equation} \label{eq:BordDeltaR}
\begin{split}
-\int_{\partial D} \phi_j^{\eps, *}  & \partial_j g \partial_i u_0 \delta a_\eps (\nabla \phi_i^{\eps,\prime} + e_i) \cdot \nu - \int_{\partial D} \partial_j g \partial_i u_0 \phi_j^{\eps, *} a_\eps \nabla \delta \phi_i^{\eps} \cdot \nu \\
& = - \int_{D} \delta a_\eps (\nabla \phi_i^{\eps,\prime} + e_i) \cdot \nabla (\eta_\eps \phi_j^{\eps, *}  \partial_j g \partial_i u_0) - \nabla \delta \phi_i^{\eps}  \cdot a_\eps^* \nabla (\eta_\eps \phi_j^{\eps} \partial_j g \partial_i u_0)
\end{split}
\end{equation}
and by the skew-symmetry of $\sigma_j$,
\begin{equation} \label{eq:IPP_sigma2}
\begin{split}
\int_{D} (\nabla \cdot (\eta_\eps \sigma_j^{\eps, *} \partial_j g \partial_i u_0)) \cdot \nabla \delta \phi_i^\eps & = \int_{\partial D} (\nabla \cdot (\sigma_j^{\eps, *} \partial_j g \partial_i u_0)) \cdot \nu \delta \phi_i^\eps \\
& = - \int_{\partial D} \partial_j g \partial_i u_0 \sigma_j^{\eps, *} \nabla \delta \phi_i^\eps \cdot \nu.\\
\end{split}
\end{equation}
We combine \eqref{eq:FVri} for $\overline{h} = \eps \delta \phi_i^\eps$ (which is a suitable test function), \eqref{eq:BordDeltaR} and \eqref{eq:IPP_sigma2} to get the desired result.

\end{proof}

\subsection{Proof of Theorem~\ref{thm:Rdecay}}

\begin{proof}[Proof of Theorem~\ref{thm:Rdecay}]

Let 
$$\norme{I}_\ell^2 := \int_{\R^d} \ell^{-d} \sup_{S' \in \mathcal{A}_\ell(x)} \biggl| \sum_{j=1}^d \int_{D\cap B_{\eps \ell}(x)}  \partial_j g (\eps \nabla \phi_j^{\eps, *} + e_j) \cdot \delta a_\eps (\nabla u_\eps' - \sum_{i=1}^d (e_i + \eps \nabla \phi_i^{\eps,\prime}) \partial_i u_0) \biggr|^2 \mathrm{d}x.$$
By Cauchy-Schwarz inequality we obtain
\begin{equation}
\begin{split}
\norme{I}_\ell^2 & \lesssim \int_{\R^d} \ell^{-d}  \biggl(\sum_{j=1}^d  \int_{D\cap B_{\eps \ell}(x)} |\partial_j g (\eps \nabla \phi_j^{\eps,*} + e_j)|^2 \biggr) \\ & \hspace{1cm}\sup_{S' \in \mathcal{A}_\ell(x)} \biggl(\int_D |\nabla u_\eps' - \sum_{i=1}^d(e_i + \eps \nabla \phi_i'(\frac{\cdot}{\eps})) \partial_i u_0|^2 \biggr)\mathrm{d}x  \\
& \lesssim \int_{\R^d} \ell^{-d} \biggl(\sum_{j=1}^d \int_{D\cap B_{\eps \ell}(x)} |\partial_j g (\eps \nabla \phi_j^{\eps,*} + e_j)|^2 \biggr) \\ & \hspace{1cm} \sup_{S' \in \mathcal{A}_\ell(x)} \biggl(\norme{u_\eps' - u_0 - \sum_{i=1}^d\eps \phi_i^{\eps,\prime} \partial_i u_0}_{H^1(D)}^2 \biggr)\mathrm{d}x. 
\end{split}
\end{equation}
Using Proposition~\ref{prop:2scErrDecay}, we have moreover
$$\sup_{S' \in \mathcal{A}_\ell(x)} \norme{u_\eps' - u_0 - \sum_{i=1}^d \eps \phi_i^{\eps,\prime} \partial_i u_0}_{H^1(D)}^2 \lesssim \eps \mu_d(\frac{1}{\eps}) \sup_{S' \in \mathcal{A}_\ell(x)} (\chi_\eps')^2 \norme{u_0}_{W^{2, \infty}(D)}^2. $$
As mentioned in \cite[Remark~2.1]{duerinckx2020robustness}, by following the proof of \cite[Theorem~4]{Gloria2019}, one has that 
\begin{equation} \label{eq:Cprime}
\mathcal{C}'(z) \lesssim \mathcal{C}(z).
\end{equation} 
where $\mathcal{C}$ is defined in Proposition~\ref{prop:corrbounds}. It particular, this implies that $\sup_{S' \in \mathcal{A}_\ell(x)} (\chi_\eps')^2$ can be bounded  by $\chi_\eps^2$, which is a random variable independent of $S'$, that satisfies the integrability \eqref{eq:Exp_corrector_bounds}. Combining this with Lemma~\ref{lem:CSestimates} applied to $|\partial_j g (\eps \nabla \phi_j^\eps) + e_j|^2$ with $t = \eps \ell$ and the bounds on the gradient of the corrector yields
\begin{equation}
\begin{split}
\norme{I}_\ell^2 & \lesssim \left(\int_{\R^d} \ell^{-d} \biggl(\sum_{j=1}^d \int_{D\cap B_{\eps \ell}(x)} |\partial_j g (\eps \nabla \phi_j^\eps + e_j)|^2 \biggr)\mathrm{d}x\right)\biggl(\eps \mu_d(\frac{1}{\eps}) (\chi_\eps)^2 \norme{u_0}_{W^{2, \infty}(D)}^2 \biggr)\\
& \lesssim \biggl(\eps^{d} \sum_{j=1}^d \int_{D} |\partial_j g (\eps \nabla \phi_j^\eps + e_j)|^2 \biggr) \biggl(\eps \mu_d(\frac{1}{\eps}) (\chi_\eps)^2 \norme{u_0}_{W^{2, \infty}(D)}^2 \biggr) \\
& \lesssim \eps^{d+1} \mu_d(\frac{1}{\eps}) \widetilde{\chi_\eps^4}^2 (\chi_\eps)^2 \norme{u_0}_{W^{2, \infty}(D)}^2 \norme{g}_{W^{2, \infty}(D)}^2,
\end{split}
\end{equation}
with $\widetilde{\chi_\eps^4}$ defined as:
$$\widetilde{\chi_\eps^4} := \left(\eps^d \sum_{z \in P_\eps(D)}(1+  r_*(z))^{2d} \right)^{\frac{1}{2}}.$$
In view of Lemma~\ref{lem:holder}, $\widetilde{\chi_\eps^4}^2 (\chi_\eps)^2$ satisfies the expected stochastic integrability.

Let 
$$\norme{II}_\ell^2 := \int_{\R^d} \ell^{-d} \sup_{S' \in \mathcal{A}_\ell(x)} \left|\sum_{j=1}^d \int_{D} -(\eps \phi_j^{\eps, *} \nabla \partial_j g + \nabla r_j)  \cdot \delta a_\eps \nabla u_\eps' + k^2  \delta n_\eps u_\eps'(r_j + \eps \phi_j^{\eps, *} \partial_j g) \right|^2 \mathrm{d}x.$$
Similarly to the analysis done with $v_\eps$ in the proof of Proposition~\ref{prop:2scErrDecay}, one has that $r_j$ satisfies a similar decay rate:
\begin{equation}
\norme{r_j}_{H^1(D)} \lesssim \eps^{\frac{1}{2}} \mu_d(\frac{1}{\eps})^{\frac{1}{2}} \widetilde{\chi_\eps^5} \norme{g}_{W^{2, \infty}(D)},
\end{equation}
for some random variable $\widetilde{\chi_\eps^5}$ satisfying \eqref{eq:Exp_corrector_bounds}. \\
To get rid of the dependency with respect to $S'$, note that $\delta u_\eps$ also satisfies
\begin{equation}
\left\{\begin{aligned}
&-\Delta \delta u_\eps - k^2 \delta u_\eps = 0 && \textrm{in}\, B_R \setminus \overline{D}, \\
& -\nabla\cdot\left(a'_\eps \nabla \delta u_\eps \right) - k^2 \delta u_\eps = -\nabla \cdot (\delta a_\eps \nabla u_\eps) - k^2 \delta n_\eps u_\eps && \textrm{in}\, D,\\
& \nabla \delta u_\eps^- \cdot \nu - a'_\eps \nabla \delta u_\eps^+ \cdot \nu = -\delta a_\eps \nabla u_\eps^+ \cdot \nu   && \textrm{on}\, \partial D,\\
& \nabla \delta u_\eps \cdot \nu = \Lambda(\delta u_\eps)  && \textrm{on}\, \partial B_R.\\
\end{aligned} \right.
\end{equation}
In particular, by Proposition~\ref{prop:stabAbsorb_smallKR}
\begin{equation} \label{eq:estimate_deltaU}
\norme{\delta u_\eps}_{H^1(D)} \lesssim \norme{u_\eps}_{H^1(B_{\eps \ell}(x) \cap D)}.
\end{equation}
This gives
\begin{equation} \label{eq:estimate_uxl}
\begin{split}
\sup_{S' \in \mathcal{A}_\ell(x)} \norme{u_\eps'}_{H^1(B_{\eps \ell}(x) \cap D)} & \lesssim \sup_{S' \in \mathcal{A}_\ell(x)} \norme{\delta u_\eps}_{H^1(D)} + \norme{u_\eps}_{H^1(B_{\eps \ell}(x) \cap D)} \\
& \lesssim \norme{u_\eps}_{H^1(B_{\eps \ell}(x) \cap D)}.
\end{split}
\end{equation}
We can finally compute $\norme{II}_\ell^2$ using the Lemma~\ref{lem:CSestimates} on $u_\eps$ and the bounds on the correctors which yields
\begin{equation}
\begin{split}
\norme{II}_\ell^2 & \lesssim \biggl(\sum_{j=1}^d \int_D |\nabla r_j|^2 + \eps^2|\phi_j^{\eps, *} \nabla \partial_j g|^2 + |r_j|^2 + \eps^2 |\phi_j^{\eps, *}\partial_j  g|^2 \biggr)  \\ 
& \hspace{3cm} \times \int_{\R^d} \ell^{-d}\biggl( \sup_{S' \in \mathcal{A}_\ell(x)} \norme{u_\eps'}_{H^1(B_{\eps \ell}(x) \cap D)}^2 \biggr) \mathrm{d}x \\
& \lesssim  \eps^{d+1} \mu_d(\frac{1}{\eps}) \widetilde{\chi_\eps^6}^2 \norme{u_0}_{W^{2, \infty}(D)}^2 \norme{g}_{W^{2, \infty}(D)}^2,
\end{split}
\end{equation}
for a random variable $\widetilde{\chi_\eps^6}$ satisfying the desired integrability. \\
\noindent Let 
\begin{equation}
\begin{split}
\norme{III}_\ell^2 = \int_{\R^d} \ell^{-d} \sup_{S' \in \mathcal{A}_\ell(x)}\\ &\hspace{-2.5cm} \left|\sum_{i,j=1}^d \int_D (\eps \phi_j^{\eps, *} \nabla \partial_j g \partial_i u_0 - \eps\nabla \cdot(\eta_\eps \phi_j^{\eps,*} \partial_j g \partial_i u_0) + \nabla R_{ij}) \cdot \delta a_\eps (\eps \nabla \phi_i^{\eps,\prime} + e_i)\right|^2 \mathrm{d}x. 
\end{split}\end{equation}
To estimate $\norme{III}_\ell^2$, we follow the steps of the proof of \cite[Proposition~2.6]{duerinckx2020robustness}. By a change of variable $y \mapsto \frac{y}{\eps}$ in the integral in $D$ and by \cite[Lemma~2.9]{duerinckx2020robustness}, we obtain
\begin{equation}
\begin{split}
& \norme{III}_\ell^2 \lesssim \eps^{2d} \int_{\R^d}\sum_{i,j=1}^d \fint_{B_{r_*(x)}(x)} \eps^2\Big[ |\phi_j^{\eps, *} \nabla \partial_j g \partial_i u_0 \mathbbm{1}_D|^2(\eps\cdot) \\ &  \hspace{1cm}+ \eps^2 |\nabla \cdot(\eta_\eps \phi_j^{\eps,*} \partial_j g \partial_i u_0 \mathbbm{1}_D|^2(\eps\cdot) + |\nabla R_{ij}|^2(\eps\cdot)
\Big] \sup_{S' \in \mathcal{A}_\ell(x)} \int_{B_{2 \ell+ r_*(x)}(x)} | \nabla \phi_i' + e_i|^2 \mathrm{d}x. \\
\end{split}
\end{equation}
Moreover from \cite[Proof~of~Theorem~4]{gloria2014regularity}, we obtain for $i \in \intInter{1,d}$,
$$\sup_{S'} \int_{B_{2 \ell+ r_*(x)}(x)} |\nabla \phi'_i + e_i|^2 \mathrm{d}x\lesssim  \int_{B_{2 \ell+ r_*(x)}(x)} |\nabla \phi_i + e_i|^2  \lesssim 2^d (\ell + r_*(x))^d \mathrm{d}x. $$
Thus, since for all $x \in \R^d$, $(\ell + r_*(x))^d \lesssim \ell^d r_*(x)^d$, we have
\begin{equation}
\begin{split}
& \norme{III}_\ell^2 \lesssim \eps^{2d} \ell^{d} \sum_{i,j=1}^d \int_{\R^d} r_*(x) \times \\
& \hspace{0.4cm}\left(\fint_{B_{r_*(x)}(x)} \eps^2 |\phi_j^{\eps, *} \nabla \partial_j g \partial_i u_0 \mathbbm{1}_D|^2(\eps\cdot) + \eps^2|\nabla \cdot(\eta_\eps \phi_j^{\eps,*} \partial_j g \partial_i u_0 \mathbbm{1}_D|^2(\eps\cdot) + |\nabla R_{ij}|^2(\eps\cdot) \right) \mathrm{d}x.\\
\end{split}
\end{equation}
We recall the following estimate \cite[(3.8)]{duerinckx2020robustness} in the form: 
If $v$ is the solution in $\dot{H}(\R^d)$ of $-\nabla \cdot a \nabla v = \nabla \cdot h$, with $h \in L^2(D)$, then for all $\alpha$ such that $d < \alpha < 3d $ and for all $T>1$
\begin{equation} \label{eq:CZ-mitia}
\begin{split}
\int_{\R^d} & r_*(x)^d \biggl(\fint_{B_{r_*(x)}(x)} |h|^2 + |\nabla v|^2\biggr)\\
&\lesssim_{\alpha} r_*(0)^{\frac{\alpha}{2}} \biggl(\int_{\R^d} r_*^{2d} \rho_T^{-\alpha} \biggr)^{\frac{1}{2}} \biggl(\int_{\R^d} \rho_T^{\alpha} \biggl(\int_{B(x)}|h|^2 \biggr)^2 \biggr)^{\frac{1}{2}},
\end{split}
\end{equation}
with $\rho_T$ defined in Lemma~\ref{lem:CSestimates}. Note that $R_{ij}$ verifies:
\begin{equation}
\begin{split}
- \nabla & \cdot (a \nabla R_{ij}(\eps \cdot)) \\
& = - \eps \nabla \cdot \biggl(\biggl(\biggl((a_\eps^* \phi_j^{\eps, *} - \sigma_j^{\eps, *}) \nabla(\partial_j g \partial_i u_0) \\
& \hspace{2.5cm} - (\nabla \cdot (\eta_\eps \sigma_j^{\eps, *} \partial_j g \partial_i u_0)) - a_\eps^* \nabla (\eta_\eps \phi_j^{\eps, *} \partial_j g \partial_i u_0) \biggr) \mathbbm{1}_D \biggr) (\eps \cdot) \biggr) \\
& := - \nabla (\cdot H_{ij}\mathbbm{1}_D(\eps \cdot)).
\end{split}
\end{equation}
Thus we can apply \eqref{eq:CZ-mitia} to $R_{ij}$ which yields for any $\alpha$ such that $d < \alpha < 3d $ and $T>1$
\begin{equation}
\begin{split}
\sum_{i,j=1}^d \int_{\R^d} r_*(x)  & \fint_{B_{r_*(x)}(x)} |\nabla R_{ij}|^2(\eps\cdot) \\
& \lesssim_{\alpha} \sum_{i,j=1}^d  r_*(0)^{\frac{\alpha}{2}} \biggl(\int_{\R^d} r_*^{2d} \rho_T^{-\alpha} \biggr)^{\frac{1}{2}} \biggl(\int_{\R^d} \rho_T^{\alpha} \biggl(\int_{B(x)}  |H_{ij} \mathbbm{1}_D|^2(\eps \cdot) \biggr)^2 \biggr)^{\frac{1}{2}} .
\end{split}
\end{equation}
By denoting:
$$ |(F_{ij} \mathbbm{1}_D) (\eps \cdot)|^2 := \eps^2 |\phi_j^{\eps, *} \nabla \partial_j g \partial_i u_0 \mathbbm{1}_D|^2(\eps\cdot) + \eps^2 |\nabla \cdot(\eta_\eps \phi_j^{\eps,*} \partial_j g \partial_i u_0 \mathbbm{1}_D|^2(\eps\cdot),$$
we also have
\begin{equation}
\begin{split}
\sum_{i,j=1}^d \int_{\R^d} r_*(x) & \fint_{B_{r_*(x)}(x)} |F_{ij} \mathbbm{1}_D|^2 (\eps \cdot) \\
& \lesssim_{\alpha} \sum_{i,j=1}^d  r_*(0)^{\frac{\alpha}{2}} \biggl(\int_{\R^d} r_*^{2d} \rho_T^{-\alpha} \biggr)^{\frac{1}{2}} \biggl(\int_{\R^d} \rho_T^{\alpha} \biggl(\int_{B(x)} |F_{ij} \mathbbm{1}_D|^2(\eps \cdot) \biggr)^2 \biggr)^{\frac{1}{2}} .
\end{split}
\end{equation}
Similarly to the analysis done in the proof of Proposition~\ref{prop:2scErrDecay}, one has that
\begin{equation}
\norme{F_{ij}}_{L^2(D)} + \norme{H_{ij}}_{L^2(D)}  \lesssim \eps^{\frac{1}{2}} \mu_d(\frac{1}{\eps})^{\frac{1}{2}} \widetilde{\chi_\eps^7} \norme{u_0}_{W^{2, \infty}(D)}^2 \norme{g}_{W^{2, \infty}(D)}^2,
\end{equation}
for a random variable $\widetilde{\chi_\eps^7}$ satisfying the desired integrability. \\
After changing variables to $x \mapsto \eps x$ in the inner and outer integral and applying \eqref{eq:CZ-2} from Lemma~\ref{lem:CSestimates} (with $U := F_{ij} + H_{ij}$) we obtain
\begin{equation}
\begin{split}
\norme{III}_\ell^2 & \lesssim_{\alpha} \eps^{d} \ell^d r_*(0)^{\frac{\alpha}{2}} \sum_{i,j=1}^d  \biggl(\int_{\R^d} r_*^{2d} \rho_T^{-\alpha} \biggr)^{\frac{1}{2}} \\& \hspace{4cm}\biggl(\int_{\R^d} \rho_T^{\alpha} \biggl(\int_{B_\eps (\eps x)} |F_{ij} \mathbbm{1}_D|^2 + |H_{ij} \mathbbm{1}_D|^2 \biggr)^2 \biggr)^{\frac{1}{2}} \\
& \lesssim_{\alpha} \eps^{d/2} \ell^d r_*(0)^{\frac{\alpha}{2}} \sum_{i,j=1}^d  \biggl(\int_{\R^d} r_*^{2d} \rho_T^{-\alpha} \biggr)^{\frac{1}{2}} \\& \hspace{4cm}  \biggl(\int_{\R^d} \rho_{\eps T}^{\alpha} \biggl(\int_{B_\eps (x)} |F_{ij} \mathbbm{1}_D|^2 + |H_{ij} \mathbbm{1}_D|^2 \biggr)^2 \biggr)^{\frac{1}{2}} \\
& \lesssim_{\alpha} \eps^{d} \ell^d r_*(0)^{\frac{\alpha}{2}} \sum_{i,j=1}^d  \biggl(\int_{\R^d} r_*^{2d} \rho_T^{-\alpha} \biggr)^{\frac{1}{2}}  \sup_{y \in D} \left(\frac{\eps + y}{\eps T} + 1 \right)^{\frac{\alpha}{2}}  \\& \hspace{4cm}\left(\int_{D} |F_{ij} \mathbbm{1}_D|^2 + |H_{ij} \mathbbm{1}_D|^2\right) \\
& \lesssim_{\alpha} \eps^{d + 1} \mu_d(\frac{1}{\eps}) \widetilde{\chi_\eps^7}^2
 \ell^d r_*(0)^{\frac{\alpha}{2}} \biggl(\int_{\R^d} r_*^{2d} \rho_T^{-\alpha} \biggr)^{\frac{1}{2}}  \sup_{y \in D} \left(\frac{\eps + y}{\eps T} + 1 \right)^{\frac{\alpha}{2}}.
\end{split}
\end{equation}
Choosing $T = \frac{1}{\eps}$ yields the desired result. Indeed, the random variables at stake $r_*$ and $ \widetilde{\chi_\eps^7}$ verify the desired stochastic integrability.

Let 
$$\norme{IV}_\ell^2 := \ell^{-d} \int_{\R^d} \sup_{S' \in \mathcal{A}_\ell(x)} \biggl|\int_D k^2 \delta n_\eps (u_\eps' - u_0) g - \eps k^2 \beta^\eps \delta u_\eps \cdot \nabla g \biggr|^2. $$
Proposition~\ref{prop:2scErrDecayL2} combined with the estimate \eqref{eq:Cprime} yields
$$\sup_{S' \in \mathcal{A}_\ell(x)} \norme{u_\eps' - u_0}_{L^2(D)} \lesssim \eps \mu_d(\frac{1}{\eps}) \widehat{\chi_\eps} \norme{u_0}_{W^{2, \infty}(D)}.$$
Therefore, using \eqref{eq:estimate_deltaU} and Lemma~\ref{lem:CSestimates} we get
\begin{equation}
\begin{split}
\norme{IV}_\ell^2 & \lesssim \ell^{-d} \int_{\R^d} \sup_{S' \in \mathcal{A}_\ell(x)} \norme{u_\eps' - u_0}_{L^2(D)}^2 \left( \int_{D \cap B_{\eps \ell}(x))} |g|^2\right) \\
& \hspace{1cm} + \eps^2 \ell^{-d} \int_{\R^d} \left(\int_{D \cap B_{\eps \ell}(x))} |u_\eps|^2 + |\nabla u_\eps|^2 \right) \norme{\beta^\eps \cdot \nabla g}_{L^2(D)}^2 \\
& \lesssim \eps^{d+2} \mu_d(\frac{1}{\eps})^2 \left(\widehat{\chi_\eps}^2 \norme{u_0}_{W^{2, \infty}(D)}^2 \norme{g}_{L^2(D)}^2 + \chi_{\eps}^2 \norme{g}_{W^{1, \infty}(D)}^2 \norme{u_\eps}_{H^1(D)}^2 \right) \\
& \lesssim \eps^{d+2} \mu_d(\frac{1}{\eps})^2 \widetilde{\chi_\eps^9}^2 \norme{u_0}_{W^{2, \infty}(D)}^2 \norme{g}_{W^{2, \infty}(D)}^2,
\end{split}
\end{equation}
with $\widetilde{\chi_\eps^9}^2 := \widehat{\chi_{\eps}}^2 + \chi_{\eps}^2$.
\end{proof}

\section{Numerical illustrations} \label{section:numericalSimul}
In this section, we illustrate numerically the different asymptotic expansion of $u_\eps$, \textit{i.e.} the results of Proposition~\ref{prop:2scErrDecayL2} and Corollary~\ref{cor:R1decay}. Especially, we recover the predicted convergence rates.

\subsection{Geometry and choice of parameters} 
We choose $D$ as the two-dimensional square $(-L_D/2,L_D/2)^2$, with $L_D = 5$. All the inclusions are disks of equal radius. The centers of the inclusions of size $1$ are sampled according to a Mat\`ern point process \cite[Section~6.5.2]{Matern} in a domain $\mathcal{Q}_L := (-\frac{L}{2}, \frac{L}{2})^d$ with $L \gg 1$. To compute the correctors and the associated homogenized coefficients, we will use periodization \cite{clozeau2023bias} and thus the Mat\`ern process is periodized in $\mathcal{Q}_L$. The different parameters chosen for the simulation are summarized in Table~\ref{tab:parameters}.   
\begin{table}[ht]
\centering
\begin{tabular}{|c | l|}
\hline
Parameter  &  Value \\
\hline
Angle of the incident wave $u^{inc}$ & $0$ (From left to right) \\
	$h$		  & $0.07$ 		 \\
	$k$		  & $5$ 		 \\
$(a_M, a_S)$  & $(2.0, 3.5)$ \\ 
$(n_M, n_S)$  & $(1.5, 0.5)$ \\ 
$\eps$ & Between $0.18$ and $0.09$ \\
Volumic fraction of inclusions & $\approx 22.6 \% $ \\
\hline
\end{tabular}
\caption{Parameters of the simulation}
\label{tab:parameters}
\end{table}
Note in particular that $L_D$ is of the order of a few wavelengths. \\
The solutions are computed with $XLiFE++$ \cite{xlifepp}, an open source FEM, BEM, and FEM-BEM solver. In order to avoid significant discretization errors and distinguish them from the homogenization error, the mesh step $h$ is taken sufficiently small, \textit{i.e.} much smaller than $\eps$. We choose a $P1$ mesh. All the equations defined in $B_R$ are implemented with a  classical FEM-BEM coupling to avoid numerically computing the corresponding Dirichlet-to-Neumann operators on $\partial B_R$. We choose $P2$ elements both for the FEM and the BEM unknown. 
For a single realization, with the set of parameters of Table~\ref{tab:parameters}, the problems $u_\eps$, $u_0$ and $\mathcal{U}_1$ could be simulated in a few minutes on a personal laptop. All the following simulations were obtained using a server with bi-processors, AMD EPYC Processor 7452 2.35 GHz with 128 threads, 2 chips, 32 cores/chip, 2 threads/core with RAM of 256 Go. 

\subsection{Computation of the reference solution}
We describe here the procedure to simulate $u_\eps$. The computation of $u_0$ and $\mathcal{U}_1$ will be done similarly. We solve simultaneously $u_\eps^{+} \in H^1(D)$, the solution of the equation for $u_\eps$ inside $D$ and the flux $p_\eps^{+} \in H^{-\frac{1}{2}}(\partial D)$,  
$$p_\eps^{+} := a_\eps \nabla u_\eps^+ \cdot \nu.$$
Since the outside domain is homogeneous, by knowing only $u_\eps^{+}$ and $p_\eps^{+}$, we can compute $u_\eps(y)$ for $y \in B_R \setminus \overline{D}$ using the Green function $\mathcal{G}$ of the free space:
\begin{equation}
\mathcal{G}(x,y) := \left\{\begin{aligned}
& \frac{i}{4}H_0^{(1)}(k|x-y|) && \textrm{if }\, d =2,\\
& \frac{\exp(ik|x-y|)}{4\pi|x-y|} && \textrm{if }\, d = 3,\\
\end{aligned} \right.
\end{equation}
where $H_0^{(1)}$ is the first Hankel function of the first kind \cite{ammari2013mathematical}. $u_\eps$ satisfies for $y \in B_R \setminus \overline{D}$ 
$$u_\eps(y) = u^{inc}(y) + \int_{\partial D} \nabla \mathcal{G}(\cdot, y) \cdot \nu u_\eps^- - \nabla u_\eps^- \cdot \nu \mathcal{G}(\cdot, y),$$
\textit{i.e.}
\begin{equation} \label{eq:uepsext}
u_\eps(y) = u^{inc}(y) + \int_{\partial D} \nabla \mathcal{G}(\cdot, y) \cdot \nu u_\eps^{+} - p_\eps^{+} \mathcal{G}(\cdot, y).
\end{equation}
Then $u_\eps^{+}$ is the solution in $H^1(D)$ of:
\begin{equation} \label{eq:uepsi}
\left\{\begin{aligned}
& -\nabla\cdot\left(a_\eps \nabla u_\eps^{+} \right) - k^2 n_\eps u_\eps^{+} = 0 && \textrm{in}\, D,\\
& a_\eps \nabla u_\eps^{+} \cdot \nu = p_\eps^{+} && \textrm{on}\, \partial D.\\
\end{aligned} \right.
\end{equation}
The equation for the flux $p_\eps^+$ on $\partial D$ is obtained by taking the normal trace of \eqref{eq:uepsext}. Using the classical jump formula for the single layer potential \cite[(2.64)]{ammari2013mathematical}, we have
\begin{equation} \label{eq:pepsi}
\frac{u_\eps^{+}(y)}{2} = u^{inc}(y) + \int_{\partial D} \nabla \mathcal{G}(\cdot, y) \cdot \nu u_\eps^{+} - p_\eps^{+} \mathcal{G}(\cdot, y).
\end{equation}
By coupling \eqref{eq:uepsi} and \eqref{eq:pepsi}, this yields the following variational formulation: 
$\textrm{find } (u_\eps^+, p_\eps^+) \in H^1(D) \times H^{-\frac{1}{2}}(D) \textrm{ such that for all }  (v, q) \in H^1(D) \times H^{\frac{1}{2}}(D)$,
\begin{equation}
\begin{split}
\int_D & a_\eps \nabla u_\eps^+ \cdot \nabla v - k^2 n^{hom} u_\eps^+ v - \int_{\partial D} p_\eps^+ v + \frac{u_\eps^{+}(y)}{2}q  \\
& + \iint_{\partial D \times \partial D} \biggl(- u_\eps^{+}(x) \nabla \mathcal{G}(x, y) \cdot \nu(x) q(y) + p_\eps^{+}(x) \mathcal{G}(x, y) q(y) \biggr) d\sigma(x) d\sigma(y) \\
& \hspace{2cm}= \int_{\partial D} u^{inc} q.
\end{split}
\end{equation}
The simulation of $u_0$ is done similarly by replacing the coefficient fields $a_\eps$ and $n_\eps$ with $a^{hom}$ and $n^{hom}$. \\

\subsection{Computation of the correctors and effective parameters}
As it is customary in stochastic homogenization, we choose to compute $\phi$ with periodic boundary condition and a regularization term. We compute the periodized correctors $\phi_i^{T, L}$ solutions in $$H^1_{per}(\mathcal{Q}_L) := \{\phi \in H^1_{loc}(\R^d)~|~\phi~\mathcal{Q}_L\textrm{-periodic}\}$$ 
of 
$$\frac{1}{T} \phi_i^{T,L} - \nabla \cdot a (\nabla \phi_i^{T,L} + e_i) = 0. $$
The massive term ensures $\int_{\mathcal{Q}_L} \phi_i^{T,L} = 0$. If one computes $\widetilde{a^{hom}}$ as
$$[\widetilde{a^{hom}}]_{i,j} :=  \E\left[\fint_{(-\frac{L}{2}, \frac{L}{2})^d} a (e_i + \nabla \phi_i^{T,L}) \cdot (e_j + \nabla \phi_j^{T,L})\right], $$
then one has that $\lim_{T,L \longrightarrow \infty} \left[\widetilde{a^{hom}} \right]_{i,j} = [a^{hom}]_{i,j} $. \\
Furthermore, from \cite[Proposition~2]{gloria2017quantitative}, we know that the corrector $\phi^T$ posed in the entire space $\R^d$ without periodic condition satisfies for $T \gg 1$,
\begin{equation}
\E[\vert \nabla \phi_T - \nabla \phi \vert^2] \lesssim \left\{\begin{aligned}
& T^{-1} && \textrm{if } d = 2, \\
& T^{-\frac{3}{2}} && \textrm{if } d = 3. \\
\end{aligned} \right.
\end{equation}   
Therefore, for $T$ and $L$ sufficiently large, $\phi_i^{T,L}$ is a good approximation of $\phi_i$ \cite{bourgeat2004approximations}. \\
To compute the numerical approximations of $a^{hom}$ and $n^{hom}$ that we call $a^{hom}_{num}$ and $n^{hom}_{num}$, we use a Monte-Carlo algorithm. For a fixed number $N$ of distinct periodic realizations we compute
$$[a^{hom}_{num}]_{ij} := \frac{1}{N}\sum_{m = 1}^N \int_{(-\frac{L}{2}, \frac{L}{2})^d} a^m (e_i + \nabla \phi_i^{T,L,m}) \cdot (e_j + \nabla \phi_j^{T,L,m}),$$
and
$$n^{hom}_{num} := \frac{1}{N}\sum_{m=1}^N \int_{(-\frac{L}{2}, \frac{L}{2})^d} n^m, $$
where $a^k$, $n^k$ and $\phi_i^{T,L,k}$ are respectively the coefficients for the $k$-th realization and the solution of the periodized corrector equation for the $k$-th realization.\\

We choose $T = 10^{7}$, $L= 50$ and $N = 20$. 
For the set of parameters described in Table~\ref{tab:parameters}, we find the following homogenized coefficients:
$$a^{hom}_{num} := \begin{bmatrix}
2.27054991565 & 0.000164757342405 \\
0.000164757342405 & 2.27054991565\\
\end{bmatrix},$$
and 
$$n^{hom}_{num} := 1.2735108046.$$


To simulate $\mathcal{U}_1$, we remark that $\mathcal{U}_1$ is the solution in $H^1(B_R)$ of
\begin{equation} 
\left\{\begin{aligned}
&-\Delta \mathcal{U}_1 - k^2 \mathcal{U}_1 = 0 && \textrm{in}\, B_R \setminus \overline{D}, \\
& -\nabla\cdot\left(a^{hom} \nabla \mathcal{U}_1 \right) - k^2 n^{hom} \mathcal{U}_1 = -\nabla \cdot \mathcal{H}_\eps - k^2 (n^{hom} - n_\eps)u_0 && \textrm{in}\, D,\\
&\mathcal{U}_1^- - \mathcal{U}_1^+ = 0 && \textrm{on}\, \partial D,\\
& \nabla \mathcal{U}_1^- \cdot \nu - a^{hom} \nabla \mathcal{U}_1^+ \cdot \nu = - \mathcal{H}_\eps \cdot \nu  && \textrm{on}\, \partial D,\\
& \nabla \mathcal{U}_1 \cdot \nu = \Lambda(\mathcal{U}_1) && \textrm{on}\, \partial B_R,\\
\end{aligned} \right.
\end{equation}
with $\mathcal{H}_\eps$ defined as
$$\mathcal{H}_\eps := (a^{hom} - a_\eps)(e_i + \nabla \phi_i^\eps) \partial_i u_0,$$
so that $\mathcal{U}_1$ can be simulated just as $u_0$ with the correct source term. 

\subsection{Numerical results}
We show here the results of the computations of $u_\eps$, $u_0$ and $\mathcal{U}_1$ on Figure~\ref{fig:Simu1} and Figure~\ref{fig:Simu2}. 
We plot the mesh and the solutions associated with one realization. We also plot and compare the error terms and the correction $\mathcal{U}_1$ to illustrate both Proposition~\ref{prop:2scErrDecay} and Theorem~\ref{thm:Rdecay}.

\begin{figure*}[!ht]
        \centering
        \begin{subfigure}[b]{0.49\textwidth}  
            \centering 
            \includegraphics[scale=0.1]{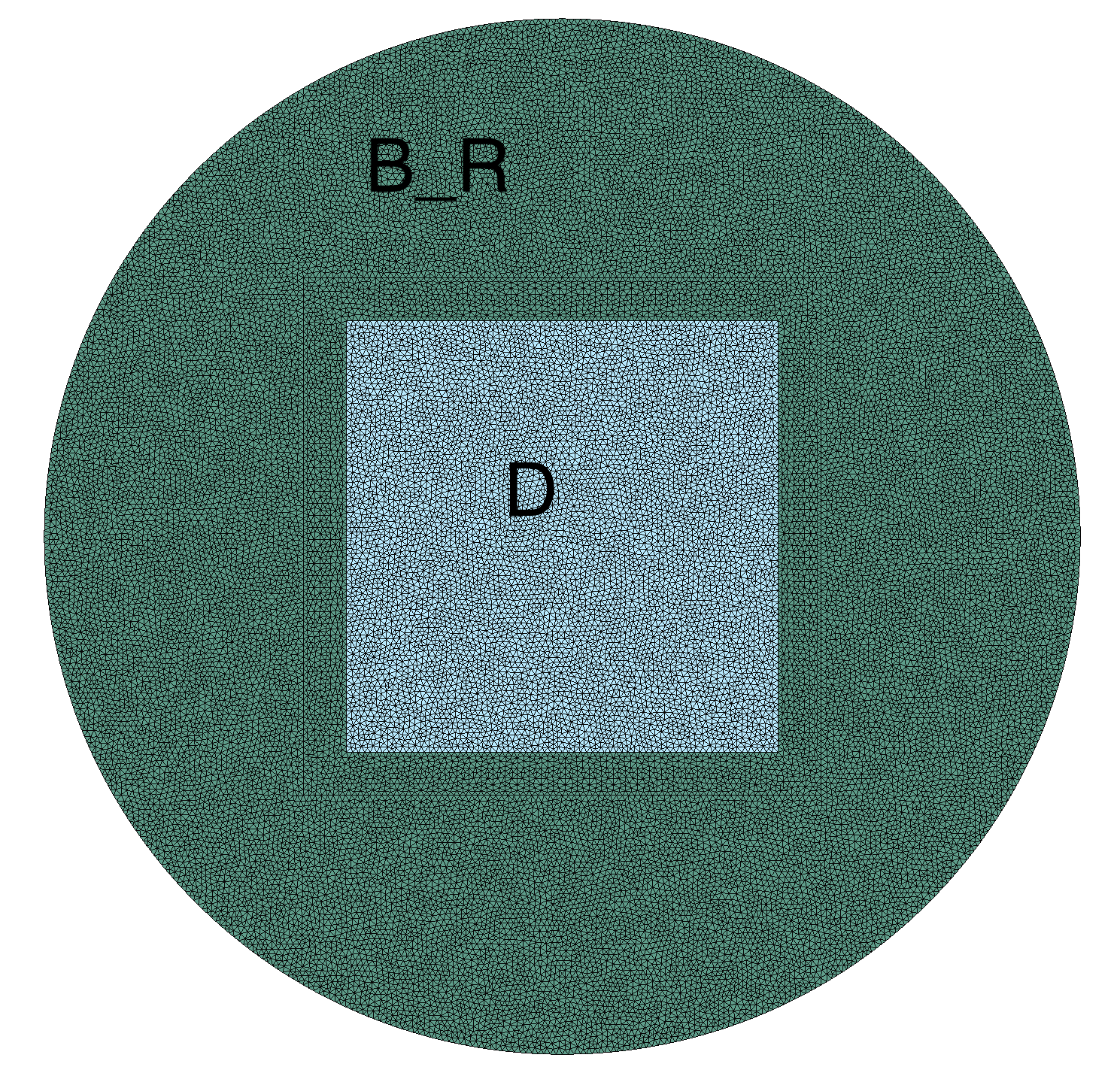}
            \caption{Mesh with domains}
  			\label{fig:mesh}
        \end{subfigure}
        \hfill
        \begin{subfigure}[b]{0.49\textwidth}
            \centering
            \includegraphics[width=\textwidth]{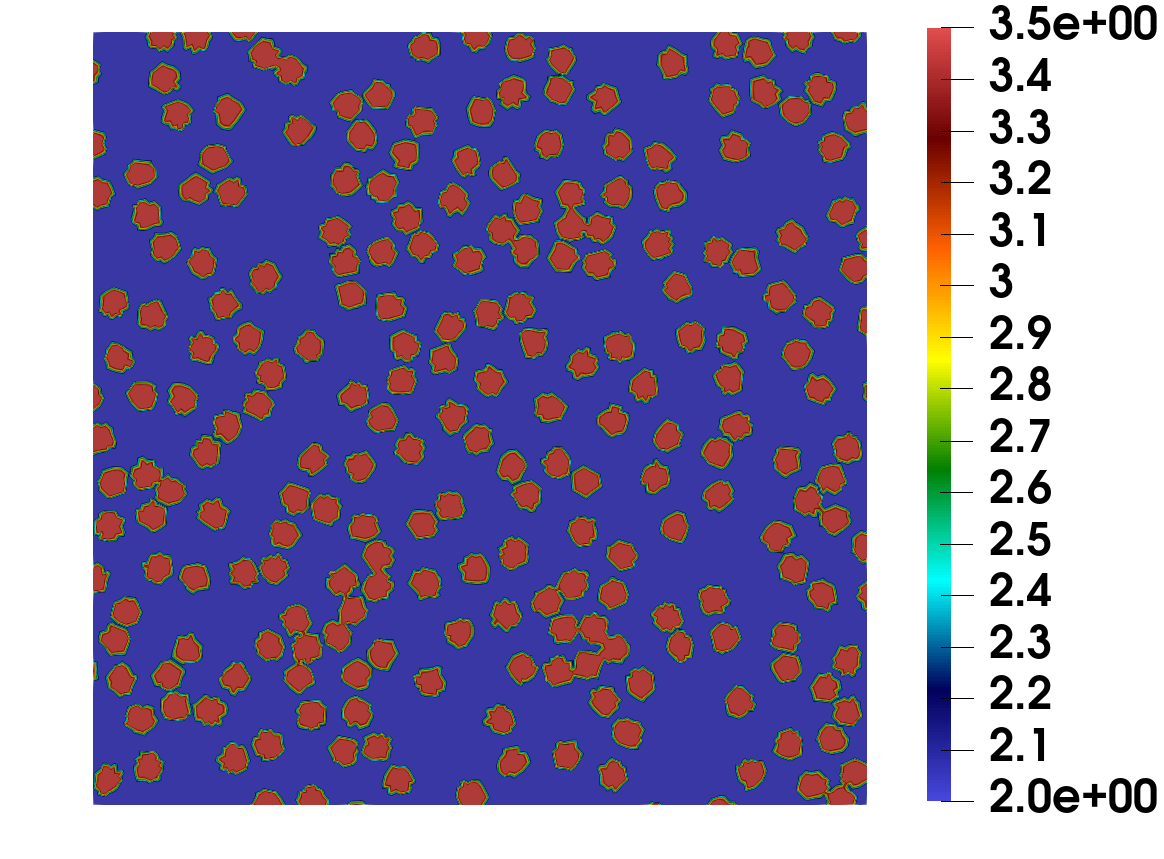}
            \caption{Realization of $a_\eps$ with $\eps = 0.09$ in $D$}
  			\label{fig:aeps}
        \end{subfigure}
        \vskip\baselineskip
        \begin{subfigure}[b]{0.49\textwidth}   
            \centering 
            \includegraphics[width=\textwidth]{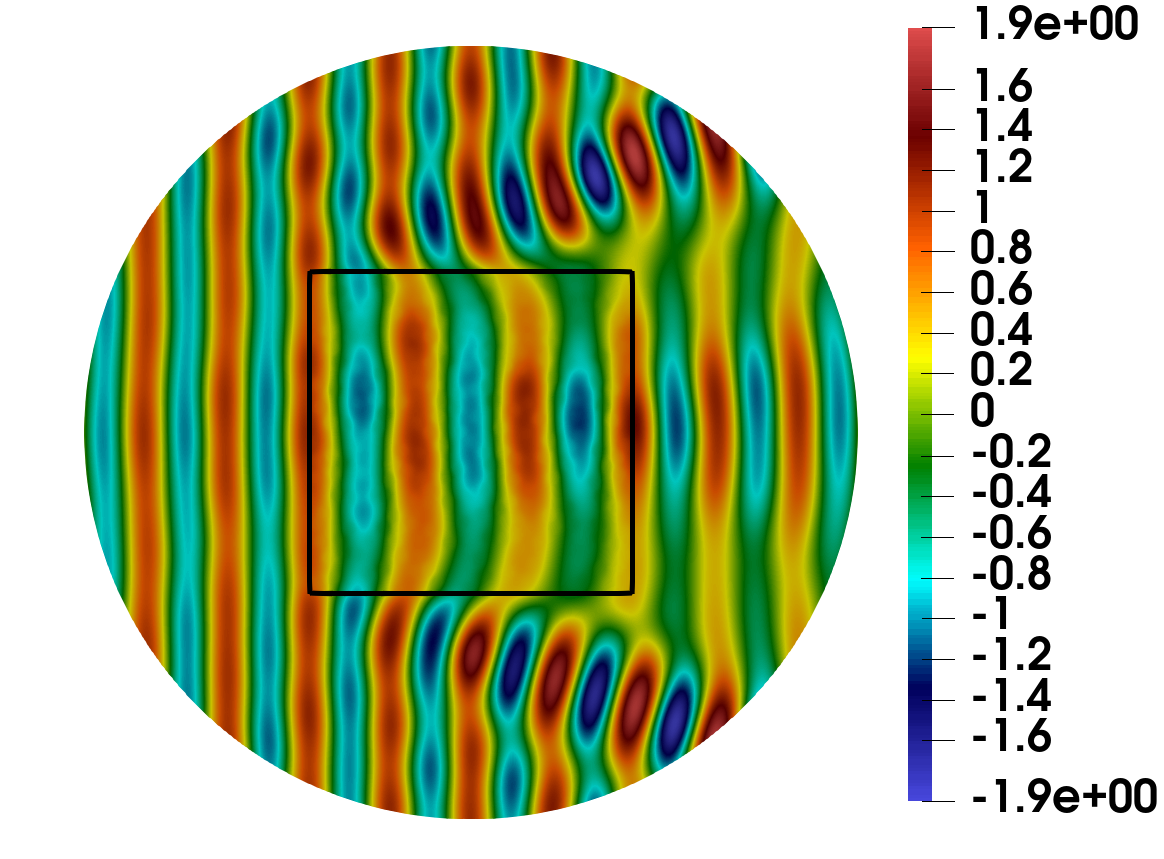}
            \caption{Solution $u_\eps$ with $\eps = 0.09$ }    
            \label{fig:ueps}
        \end{subfigure}
        \hfill
        \begin{subfigure}[b]{0.49\textwidth}   
            \centering 
            \includegraphics[width=\textwidth]{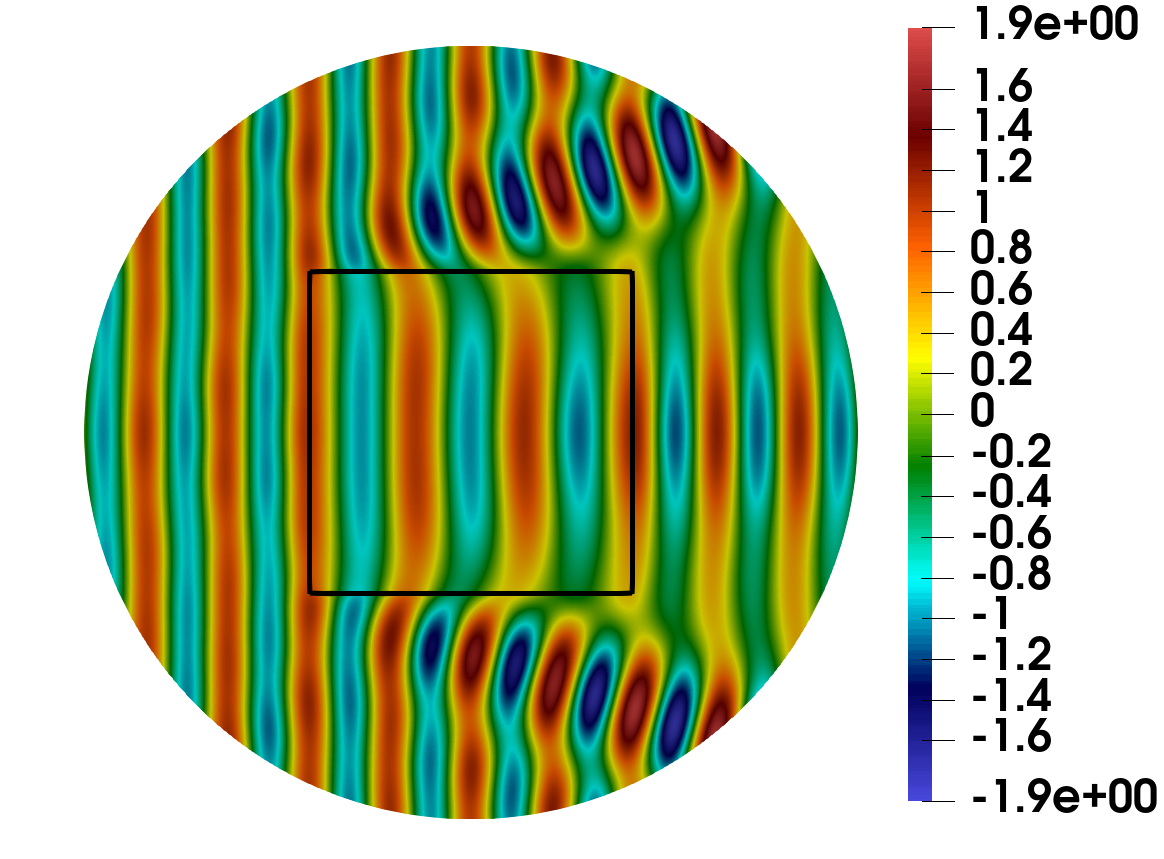}
            \caption{Solution $u_0$}
  			\label{fig:u0}
        \end{subfigure}
	\caption{(a) Mesh, (b) $a_\eps$ for a realization of $S$ with $\eps = 0.09$, (c) corresponding $u_\eps$ and (d) homogenized solution $u_0$ for an incident plane wave along $(1, 0)$}
	\label{fig:Simu1}
    \end{figure*}
    
    \begin{figure*}[!ht]
        \centering
        \begin{subfigure}[b]{0.49\textwidth}  
            \centering 
            \includegraphics[width=\textwidth]{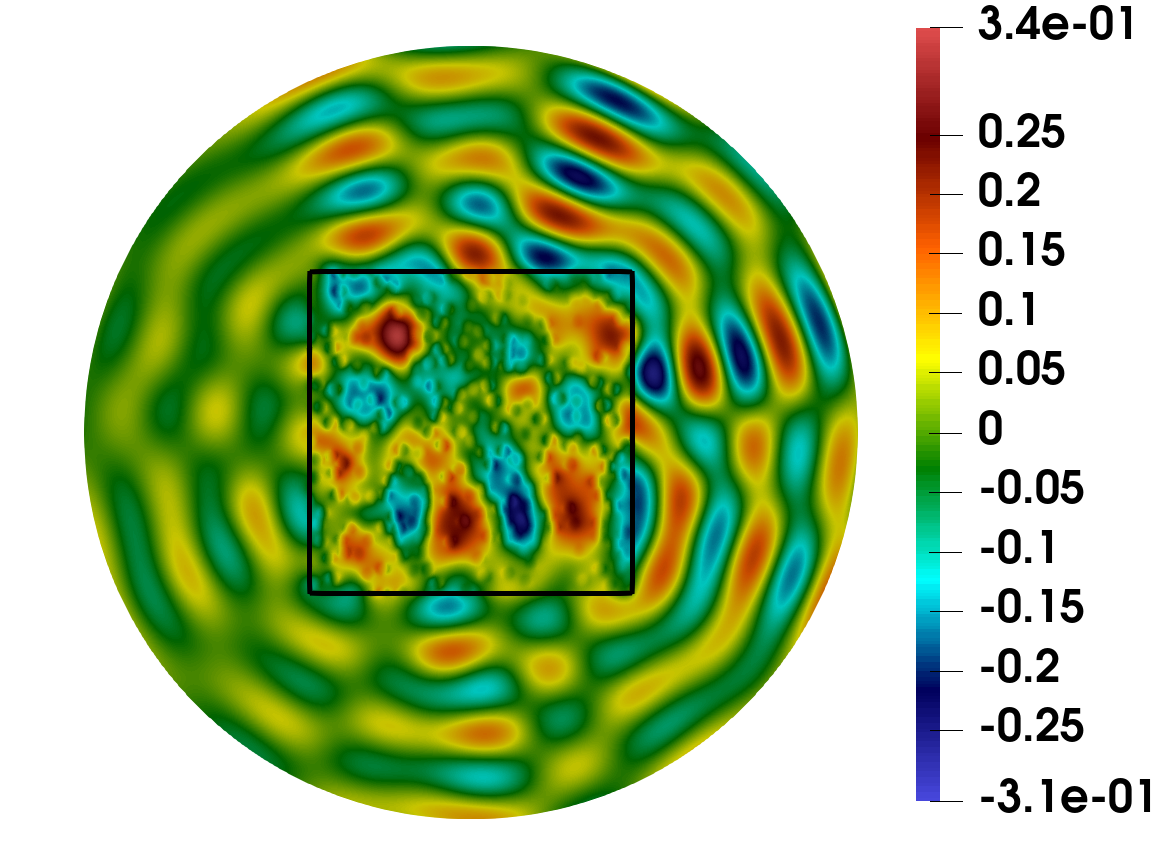}
            \caption{$u_\eps - u_0$ for the realization of $S$ in Figure~\ref{fig:aeps} ($\eps = 0.09$)}
  			\label{fig:err}
        \end{subfigure}
        \hfill
        \begin{subfigure}[b]{0.49\textwidth}
            \centering
            \includegraphics[width=\textwidth]{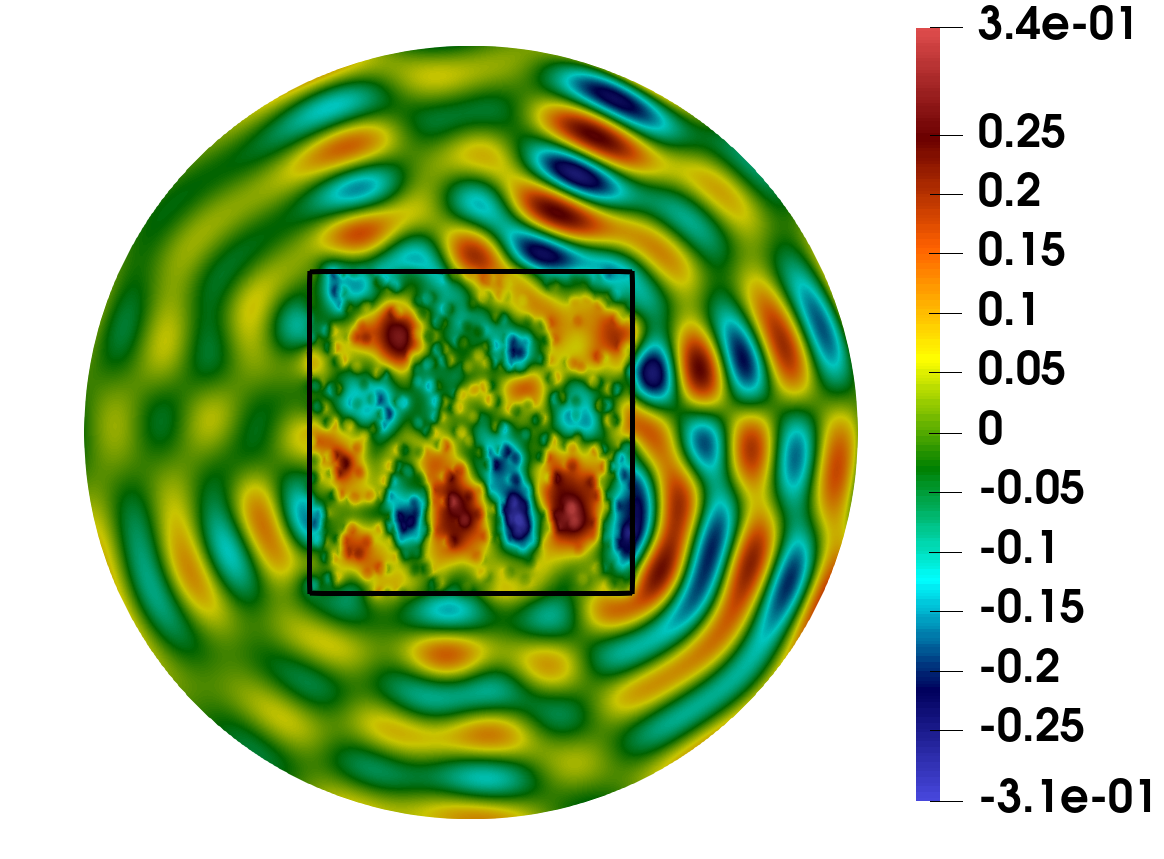}
            \caption{$\mathcal{U}_1$ for the realization of $S$ in Figure~\ref{fig:aeps} ($\eps = 0.09$)}
  			\label{fig:u1}
        \end{subfigure}
        \vskip\baselineskip
        \begin{subfigure}[b]{0.49\textwidth}   
            \centering 
            \includegraphics[width=\textwidth]{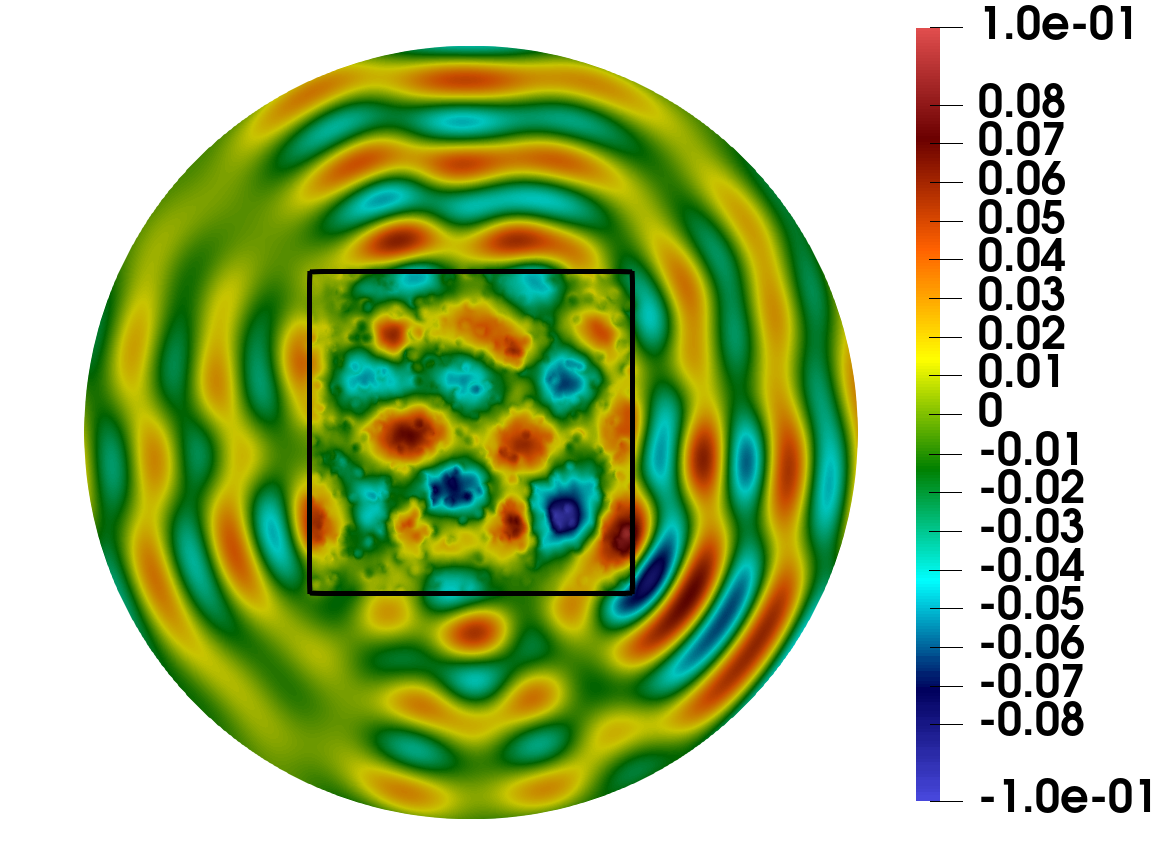}
           \caption{$u_\eps - u_0 - \mathcal{U}_1$ for the realization of $S$ in Figure~\ref{fig:aeps} ($\eps = 0.09$)}    
            \label{fig:err1}
        \end{subfigure}
	\caption{For the realization $S$ shown in Figure~\ref{fig:aeps}, (a) error term $u_\eps - u_0$, (b) correction term $\mathcal{U}_1$, (c) error term term $u_\eps - u_0 - \mathcal{U}_1$}
	\label{fig:Simu2}
    \end{figure*}

The Monte-Carlo process to compute the average error is done with $30$ realizations on Figure~\ref{fig:ErrorSlopes}. \\
One can see on Figure~\ref{fig:ErrorSlopes}, that the expected error decay of order $\eps^{\frac{d+1}{2}}\mu_d(\frac{1}{\eps})^{\frac{1}{2}} = \eps^{\frac{3}{2}}\vert \log(\eps)\vert^{\frac{1}{2}}$ in Theorem~\ref{thm:Rdecay} is obtained. For values of $\eps$ of order $0.1$, the asymptotic expansion $u_0 + \eps \mathcal{U}_1$ is already a very good approximation of the field $u_\eps$. \\
\begin{figure*}[!ht]
        \centering 
\includegraphics[width=\textwidth]{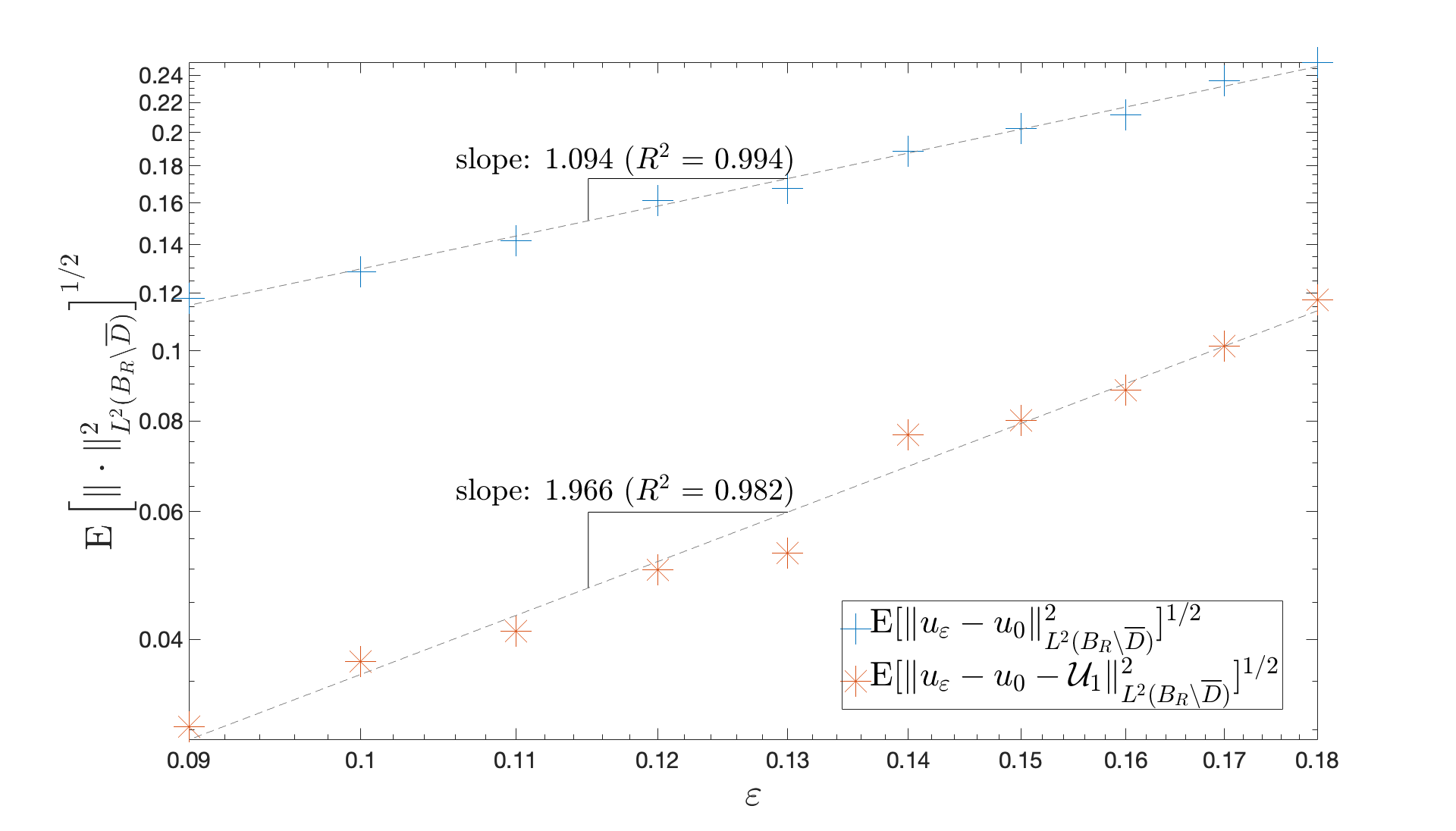}
           \caption{Error decay and linear regression}  
\label{fig:ErrorSlopes}
\end{figure*}

\bibliography{biblio}


\begin{thebibliography}{48}
\ifx \bisbn   \undefined \def \bisbn  #1{ISBN #1}\fi
\ifx \binits  \undefined \def \binits#1{#1}\fi
\ifx \bauthor  \undefined \def \bauthor#1{#1}\fi
\ifx \batitle  \undefined \def \batitle#1{#1}\fi
\ifx \bjtitle  \undefined \def \bjtitle#1{#1}\fi
\ifx \bvolume  \undefined \def \bvolume#1{\textbf{#1}}\fi
\ifx \byear  \undefined \def \byear#1{#1}\fi
\ifx \bissue  \undefined \def \bissue#1{#1}\fi
\ifx \bfpage  \undefined \def \bfpage#1{#1}\fi
\ifx \blpage  \undefined \def \blpage #1{#1}\fi
\ifx \burl  \undefined \def \burl#1{\textsf{#1}}\fi
\ifx \doiurl  \undefined \def \doiurl#1{\url{https://doi.org/#1}}\fi
\ifx \betal  \undefined \def \betal{\textit{et al.}}\fi
\ifx \binstitute  \undefined \def \binstitute#1{#1}\fi
\ifx \binstitutionaled  \undefined \def \binstitutionaled#1{#1}\fi
\ifx \bctitle  \undefined \def \bctitle#1{#1}\fi
\ifx \beditor  \undefined \def \beditor#1{#1}\fi
\ifx \bpublisher  \undefined \def \bpublisher#1{#1}\fi
\ifx \bbtitle  \undefined \def \bbtitle#1{#1}\fi
\ifx \bedition  \undefined \def \bedition#1{#1}\fi
\ifx \bseriesno  \undefined \def \bseriesno#1{#1}\fi
\ifx \blocation  \undefined \def \blocation#1{#1}\fi
\ifx \bsertitle  \undefined \def \bsertitle#1{#1}\fi
\ifx \bsnm \undefined \def \bsnm#1{#1}\fi
\ifx \bsuffix \undefined \def \bsuffix#1{#1}\fi
\ifx \bparticle \undefined \def \bparticle#1{#1}\fi
\ifx \barticle \undefined \def \barticle#1{#1}\fi
\bibcommenthead
\ifx \bconfdate \undefined \def \bconfdate #1{#1}\fi
\ifx \botherref \undefined \def \botherref #1{#1}\fi
\ifx \url \undefined \def \url#1{\textsf{#1}}\fi
\ifx \bchapter \undefined \def \bchapter#1{#1}\fi
\ifx \bbook \undefined \def \bbook#1{#1}\fi
\ifx \bcomment \undefined \def \bcomment#1{#1}\fi
\ifx \oauthor \undefined \def \oauthor#1{#1}\fi
\ifx \citeauthoryear \undefined \def \citeauthoryear#1{#1}\fi
\ifx \endbibitem  \undefined \def \endbibitem {}\fi
\ifx \bconflocation  \undefined \def \bconflocation#1{#1}\fi
\ifx \arxivurl  \undefined \def \arxivurl#1{\textsf{#1}}\fi
\csname PreBibitemsHook\endcsname

\bibitem[\protect\citeauthoryear{Shung and Thieme}{1992}]{shung1992ultrasonic}
\begin{bbook}
\bauthor{\bsnm{Shung}, \binits{K.K.}},
\bauthor{\bsnm{Thieme}, \binits{G.A.}}:
\bbtitle{Ultrasonic Scattering in Biological Tissues}.
\bpublisher{CRC press},
\blocation{{}}
(\byear{1992})
\end{bbook}
\endbibitem

\bibitem[\protect\citeauthoryear{Ueda and Ozawa}{1985}]{ueda1985spectral}
\begin{barticle}
\bauthor{\bsnm{Ueda}, \binits{M.}},
\bauthor{\bsnm{Ozawa}, \binits{Y.}}:
\batitle{{Spectral analysis of echoes for backscattering coefficient
  measurement}}.
\bjtitle{The Journal of the Acoustical Society of America}
\bvolume{77}(\bissue{1}),
\bfpage{38}--\blpage{47}
(\byear{1985})
\end{barticle}
\endbibitem

\bibitem[\protect\citeauthoryear{Norton and
  Linzer}{1981}]{norton1981ultrasonic}
\begin{botherref}
\oauthor{\bsnm{Norton}, \binits{S.J.}},
\oauthor{\bsnm{Linzer}, \binits{M.}}:
{Ultrasonic reflectivity imaging in three dimensions: exact inverse scattering
  solutions for plane, cylindrical, and spherical apertures}.
IEEE Transactions on biomedical engineering
(2),
202--220
(1981)
\end{botherref}
\endbibitem

\bibitem[\protect\citeauthoryear{Mamou and Oelze}{2013}]{mamou2013quantitative}
\begin{bbook}
\bauthor{\bsnm{Mamou}, \binits{J.}},
\bauthor{\bsnm{Oelze}, \binits{M.L.}}:
\bbtitle{Quantitative Ultrasound in Soft Tissues}.
\bpublisher{Springer},
\blocation{{}}
(\byear{2013})
\end{bbook}
\endbibitem

\bibitem[\protect\citeauthoryear{Lambert et~al.}{2020}]{lambert2020reflection}
\begin{barticle}
\bauthor{\bsnm{Lambert}, \binits{W.}},
\bauthor{\bsnm{Cobus}, \binits{L.A.}},
\bauthor{\bsnm{Couade}, \binits{M.}},
\bauthor{\bsnm{Fink}, \binits{M.}},
\bauthor{\bsnm{Aubry}, \binits{A.}}:
\batitle{{Reflection matrix approach for quantitative imaging of scattering
  media}}.
\bjtitle{Physical Review X}
\bvolume{10}(\bissue{2}),
\bfpage{021048}
(\byear{2020})
\end{barticle}
\endbibitem

\bibitem[\protect\citeauthoryear{Bensoussan
  et~al.}{2011}]{bensoussan2011asymptotic}
\begin{bbook}
\bauthor{\bsnm{Bensoussan}, \binits{A.}},
\bauthor{\bsnm{Lions}, \binits{J.-L.}},
\bauthor{\bsnm{Papanicolaou}, \binits{G.}}:
\bbtitle{Asymptotic Analysis for Periodic Structures}
vol. \bseriesno{374}.
\bpublisher{American Mathematical Soc.},
\blocation{{}}
(\byear{2011})
\end{bbook}
\endbibitem

\bibitem[\protect\citeauthoryear{Papanicolaou}{1979}]{papanicolaou1979boundary}
\begin{bchapter}
\bauthor{\bsnm{Papanicolaou}, \binits{G.C.}}:
\bctitle{{Boundary value problems with rapidly oscillating random
  coefficients}}.
In: \bbtitle{Colloquia Math. Soc., Janos Bolyai},
vol. \bseriesno{27},
pp. \bfpage{853}--\blpage{873}
(\byear{1979})
\end{bchapter}
\endbibitem

\bibitem[\protect\citeauthoryear{Jikov et~al.}{2012}]{jikov2012homogenization}
\begin{bbook}
\bauthor{\bsnm{Jikov}, \binits{V.V.}},
\bauthor{\bsnm{Kozlov}, \binits{S.M.}},
\bauthor{\bsnm{Oleinik}, \binits{O.A.}}:
\bbtitle{Homogenization of Differential Operators and Integral Functionals}.
\bpublisher{Springer},
\blocation{{}}
(\byear{2012})
\end{bbook}
\endbibitem

\bibitem[\protect\citeauthoryear{Armstrong
  et~al.}{2017}]{armstrong2017additive}
\begin{barticle}
\bauthor{\bsnm{Armstrong}, \binits{S.}},
\bauthor{\bsnm{Kuusi}, \binits{T.}},
\bauthor{\bsnm{Mourrat}, \binits{J.-C.}}:
\batitle{{The additive structure of elliptic homogenization}}.
\bjtitle{Inventiones mathematicae}
\bvolume{208}(\bissue{3}),
\bfpage{999}--\blpage{1154}
(\byear{2017})
\end{barticle}
\endbibitem

\bibitem[\protect\citeauthoryear{Gloria
  et~al.}{2015}]{gloria2015quantification}
\begin{barticle}
\bauthor{\bsnm{Gloria}, \binits{A.}},
\bauthor{\bsnm{Neukamm}, \binits{S.}},
\bauthor{\bsnm{Otto}, \binits{F.}}:
\batitle{{Quantification of ergodicity in stochastic homogenization: optimal
  bounds via spectral gap on Glauber dynamics}}.
\bjtitle{Inventiones mathematicae}
\bvolume{199}(\bissue{2}),
\bfpage{455}--\blpage{515}
(\byear{2015})
\end{barticle}
\endbibitem

\bibitem[\protect\citeauthoryear{Gloria and Otto}{2015}]{gloria2015corrector}
\begin{botherref}
\oauthor{\bsnm{Gloria}, \binits{A.}},
\oauthor{\bsnm{Otto}, \binits{F.}}:
{The corrector in stochastic homogenization: optimal rates, stochastic
  integrability, and fluctuations}.
arXiv preprint arXiv:1510.08290
(2015)
\end{botherref}
\endbibitem

\bibitem[\protect\citeauthoryear{Gloria et~al.}{2021}]{gloria2021quantitative}
\begin{barticle}
\bauthor{\bsnm{Gloria}, \binits{A.}},
\bauthor{\bsnm{Neukamm}, \binits{S.}},
\bauthor{\bsnm{Otto}, \binits{F.}}:
\batitle{{Quantitative estimates in stochastic homogenization for correlated
  coefficient fields}}.
\bjtitle{Analysis \& PDE}
\bvolume{14}(\bissue{8}),
\bfpage{2497}--\blpage{2537}
(\byear{2021})
\end{barticle}
\endbibitem

\bibitem[\protect\citeauthoryear{Duerinckx and
  Gloria}{2020}]{duerinckx2020multiscale}
\begin{barticle}
\bauthor{\bsnm{Duerinckx}, \binits{M.}},
\bauthor{\bsnm{Gloria}, \binits{A.}}:
\batitle{{Multiscale functional inequalities in probability: Constructive
  approach}}.
\bjtitle{Annales Henri Lebesgue}
\bvolume{3},
\bfpage{825}--\blpage{872}
(\byear{2020})
\end{barticle}
\endbibitem

\bibitem[\protect\citeauthoryear{Armstrong
  et~al.}{2019}]{armstrong2019quantitative}
\begin{bbook}
\bauthor{\bsnm{Armstrong}, \binits{S.}},
\bauthor{\bsnm{Kuusi}, \binits{T.}},
\bauthor{\bsnm{Mourrat}, \binits{J.-C.}}:
\bbtitle{Quantitative Stochastic Homogenization and Large-scale Regularity}
vol. \bseriesno{352}.
\bpublisher{Springer},
\blocation{{}}
(\byear{2019})
\end{bbook}
\endbibitem

\bibitem[\protect\citeauthoryear{Josien and
  Raithel}{2021}]{josien2021quantitative}
\begin{barticle}
\bauthor{\bsnm{Josien}, \binits{M.}},
\bauthor{\bsnm{Raithel}, \binits{C.}}:
\batitle{{Quantitative homogenization for the case of an interface between two
  heterogeneous media}}.
\bjtitle{SIAM Journal on Mathematical Analysis}
\bvolume{53}(\bissue{1}),
\bfpage{813}--\blpage{854}
(\byear{2021})
\end{barticle}
\endbibitem

\bibitem[\protect\citeauthoryear{Duerinckx
  et~al.}{2020a}]{duerinckx2020robustness}
\begin{barticle}
\bauthor{\bsnm{Duerinckx}, \binits{M.}},
\bauthor{\bsnm{Gloria}, \binits{A.}},
\bauthor{\bsnm{Otto}, \binits{F.}}:
\batitle{{Robustness of the pathwise structure of fluctuations in stochastic
  homogenization}}.
\bjtitle{Probability theory and related fields}
\bvolume{178},
\bfpage{531}--\blpage{566}
(\byear{2020})
\end{barticle}
\endbibitem

\bibitem[\protect\citeauthoryear{Duerinckx
  et~al.}{2020b}]{duerinckx2020structure}
\begin{barticle}
\bauthor{\bsnm{Duerinckx}, \binits{M.}},
\bauthor{\bsnm{Gloria}, \binits{A.}},
\bauthor{\bsnm{Otto}, \binits{F.}}:
\batitle{{The structure of fluctuations in stochastic homogenization}}.
\bjtitle{Communications in Mathematical Physics}
\bvolume{377}(\bissue{1}),
\bfpage{259}--\blpage{306}
(\byear{2020})
\end{barticle}
\endbibitem

\bibitem[\protect\citeauthoryear{Cakoni
  et~al.}{2016}]{cakoni2016homogenization}
\begin{barticle}
\bauthor{\bsnm{Cakoni}, \binits{F.}},
\bauthor{\bsnm{Guzina}, \binits{B.B.}},
\bauthor{\bsnm{Moskow}, \binits{S.}}:
\batitle{{On the homogenization of a scalar scattering problem for highly
  oscillating anisotropic media}}.
\bjtitle{SIAM Journal on Mathematical Analysis}
\bvolume{48}(\bissue{4}),
\bfpage{2532}--\blpage{2560}
(\byear{2016})
\end{barticle}
\endbibitem

\bibitem[\protect\citeauthoryear{Chaumont-Frelet and
  Spence}{2023}]{chaumont2023scattering}
\begin{barticle}
\bauthor{\bsnm{Chaumont-Frelet}, \binits{T.}},
\bauthor{\bsnm{Spence}, \binits{E.A.}}:
\batitle{{Scattering by Finely Layered Obstacles: Frequency-Explicit Bounds and
  Homogenization}}.
\bjtitle{SIAM Journal on Mathematical Analysis}
\bvolume{55}(\bissue{2}),
\bfpage{1319}--\blpage{1363}
(\byear{2023})
\end{barticle}
\endbibitem

\bibitem[\protect\citeauthoryear{Vinoles}{2016}]{vinoles2016problemes}
\begin{botherref}
\oauthor{\bsnm{Vinoles}, \binits{V.}}:
Probl{\`e}mes d'interface en pr{\'e}sence de m{\'e}tamat{\'e}riaux:
  mod{\'e}lisation, analyse et simulations.
PhD thesis,
Universit{\'e} Paris-Saclay (ComUE)
(2016)
\end{botherref}
\endbibitem

\bibitem[\protect\citeauthoryear{Cakoni et~al.}{2019}]{cakoni2019scattering}
\begin{barticle}
\bauthor{\bsnm{Cakoni}, \binits{F.}},
\bauthor{\bsnm{Guzina}, \binits{B.B.}},
\bauthor{\bsnm{Moskow}, \binits{S.}},
\bauthor{\bsnm{Pangburn}, \binits{T.}}:
\batitle{{Scattering by a bounded highly oscillating periodic medium and the
  effect of boundary correctors}}.
\bjtitle{SIAM journal on applied mathematics}
\bvolume{79}(\bissue{4}),
\bfpage{1448}--\blpage{1474}
(\byear{2019})
\end{barticle}
\endbibitem

\bibitem[\protect\citeauthoryear{Beneteau}{2021}]{beneteau2021modeles}
\begin{botherref}
\oauthor{\bsnm{Beneteau}, \binits{C.}}:
Modeles homog{\'e}n{\'e}is{\'e}s enrichis en pr{\'e}sence de bords: Analyse et
  traitement num{\'e}rique.
PhD thesis,
Institut polytechnique de Paris
(2021)
\end{botherref}
\endbibitem

\bibitem[\protect\citeauthoryear{Allaire and Amar}{1999}]{allaire1999boundary}
\begin{barticle}
\bauthor{\bsnm{Allaire}, \binits{G.}},
\bauthor{\bsnm{Amar}, \binits{M.}}:
\batitle{{Boundary layer tails in periodic homogenization}}.
\bjtitle{ESAIM: Control, Optimisation and Calculus of Variations}
\bvolume{4},
\bfpage{209}--\blpage{243}
(\byear{1999})
\end{barticle}
\endbibitem

\bibitem[\protect\citeauthoryear{G{\'e}rard-Varet and
  Masmoudi}{2012}]{gerard2012homogenization}
\begin{barticle}
\bauthor{\bsnm{G{\'e}rard-Varet}, \binits{D.}},
\bauthor{\bsnm{Masmoudi}, \binits{N.}}:
\batitle{{Homogenization and boundary layers}}.
\bjtitle{Acta mathematica}
\bvolume{209}(\bissue{1}),
\bfpage{133}--\blpage{178}
(\byear{2012})
\end{barticle}
\endbibitem

\bibitem[\protect\citeauthoryear{Prange}{2013}]{prange2013asymptotic}
\begin{barticle}
\bauthor{\bsnm{Prange}, \binits{C.}}:
\batitle{{Asymptotic analysis of boundary layer correctors in periodic
  homogenization}}.
\bjtitle{SIAM Journal on Mathematical Analysis}
\bvolume{45}(\bissue{1}),
\bfpage{345}--\blpage{387}
(\byear{2013})
\end{barticle}
\endbibitem

\bibitem[\protect\citeauthoryear{Gloria et~al.}{2014}]{gloria2014regularity}
\begin{botherref}
\oauthor{\bsnm{Gloria}, \binits{A.}},
\oauthor{\bsnm{Neukamm}, \binits{S.}},
\oauthor{\bsnm{Otto}, \binits{F.}}:
{A regularity theory for random elliptic operators}.
arXiv preprint arXiv:1409.2678
(2014)
\end{botherref}
\endbibitem

\bibitem[\protect\citeauthoryear{Pierce}{2007}]{pierce2007basic}
\begin{bbook}
\bauthor{\bsnm{Pierce}, \binits{A.D.}}:
\bbtitle{Basic Linear Acoustics}.
\bpublisher{Springer},
\blocation{{}}
(\byear{2007})
\end{bbook}
\endbibitem

\bibitem[\protect\citeauthoryear{Duerinckx and
  Gloria}{2017}]{DuerinckxGloria2017}
\begin{botherref}
\oauthor{\bsnm{Duerinckx}, \binits{M.}},
\oauthor{\bsnm{Gloria}, \binits{A.}}:
{Multiscale functional inequalities in probability: Constructive approach}.
arXiv: Probability
(2017)
\end{botherref}
\endbibitem

\bibitem[\protect\citeauthoryear{Illian et~al.}{2008}]{Matern}
\begin{bbook}
\bauthor{\bsnm{Illian}, \binits{J.}},
\bauthor{\bsnm{Penttinen}, \binits{A.}},
\bauthor{\bsnm{Stoyan}, \binits{H.}},
\bauthor{\bsnm{Stoyan}, \binits{D.}}:
\bbtitle{Statistical Analysis and Modelling of Spatial Point Patterns}.
\bpublisher{John Wiley \& Sons},
\blocation{{}}
(\byear{2008})
\end{bbook}
\endbibitem

\bibitem[\protect\citeauthoryear{Duerinckx and
  Gloria}{2017}]{duerinckx2017multiscale}
\begin{botherref}
\oauthor{\bsnm{Duerinckx}, \binits{M.}},
\oauthor{\bsnm{Gloria}, \binits{A.}}:
{Multiscale functional inequalities in probability: Concentration properties}.
arXiv preprint arXiv:1711.03148
(2017)
\end{botherref}
\endbibitem

\bibitem[\protect\citeauthoryear{}{2008}]{chandler2008wave}
\begin{botherref}
{Wave-number-explicit bounds in time-harmonic scattering},
  author={Chandler-Wilde, Simon N and Monk, Peter}.
SIAM Journal on Mathematical Analysis
\textbf{39}(5),
1428--1455
(2008)
\end{botherref}
\endbibitem

\bibitem[\protect\citeauthoryear{N{\'e}d{\'e}lec}{2001}]{nedelec2001acoustic}
\begin{bbook}
\bauthor{\bsnm{N{\'e}d{\'e}lec}, \binits{J.-C.}}:
\bbtitle{Acoustic and Electromagnetic Equations: Integral Representations for
  Harmonic Problems}
vol. \bseriesno{144}.
\bpublisher{Springer},
\blocation{{}}
(\byear{2001})
\end{bbook}
\endbibitem

\bibitem[\protect\citeauthoryear{Melenk and
  Sauter}{2010}]{melenk2010convergence}
\begin{barticle}
\bauthor{\bsnm{Melenk}, \binits{J.}},
\bauthor{\bsnm{Sauter}, \binits{S.}}:
\batitle{{Convergence analysis for finite element discretizations of the
  Helmholtz equation with Dirichlet-to-Neumann boundary conditions}}.
\bjtitle{Mathematics of Computation}
\bvolume{79}(\bissue{272}),
\bfpage{1871}--\blpage{1914}
(\byear{2010})
\end{barticle}
\endbibitem

\bibitem[\protect\citeauthoryear{Ball et~al.}{2012}]{ball2012uniqueness}
\begin{barticle}
\bauthor{\bsnm{Ball}, \binits{J.M.}},
\bauthor{\bsnm{Capdeboscq}, \binits{Y.}},
\bauthor{\bsnm{Tsering-Xiao}, \binits{B.}}:
\batitle{{On uniqueness for time harmonic anisotropic Maxwell's equations with
  piecewise regular coefficients}}.
\bjtitle{Mathematical Models and Methods in Applied Sciences}
\bvolume{22}(\bissue{11}),
\bfpage{1250036}
(\byear{2012})
\end{barticle}
\endbibitem

\bibitem[\protect\citeauthoryear{Chaumont-Frelet and Spence}{2021}]{spence}
\begin{botherref}
\oauthor{\bsnm{Chaumont-Frelet}, \binits{T.}},
\oauthor{\bsnm{Spence}, \binits{E.A.}}:
{Scattering by finely-layered obstacles: frequency-explicit bounds and
  homogenization}.
arXiv preprint arXiv:2109.11267
(2021)
\end{botherref}
\endbibitem

\bibitem[\protect\citeauthoryear{Bouchitt{\'e} and
  Felbacq}{2004}]{bouchitte2004homogenization}
\begin{barticle}
\bauthor{\bsnm{Bouchitt{\'e}}, \binits{G.}},
\bauthor{\bsnm{Felbacq}, \binits{D.}}:
\batitle{{Homogenization near resonances and artificial magnetism from
  dielectrics}}.
\bjtitle{Comptes Rendus Mathematique}
\bvolume{339}(\bissue{5}),
\bfpage{377}--\blpage{382}
(\byear{2004})
\end{barticle}
\endbibitem

\bibitem[\protect\citeauthoryear{Murat}{1977}]{tartar}
\begin{botherref}
\oauthor{\bsnm{Murat}, \binits{F.}}:
{H-Convergence, S{\'e}minaire d'Analyse Fonctionnelle et Num{\'e}rique de
  l'Universit{\'e} d'Alger, mimeographed notes. English translation: Murat and
  L. Tartar, H-Convergence}.
F. Topics in the Mathematical Modelling of Composite Materials, edited by A.
  Cherkaev and R. Kohn. Birkh{\"a}user Verlag, Boston. Series Progress in
  Nonlinear Differential Equations and their Applications
\textbf{31}
(1977)
\end{botherref}
\endbibitem

\bibitem[\protect\citeauthoryear{Cakoni et~al.}{2016}]{Cakoni2016OnTH}
\begin{barticle}
\bauthor{\bsnm{Cakoni}, \binits{F.}},
\bauthor{\bsnm{Guzina}, \binits{B.B.}},
\bauthor{\bsnm{Moskow}, \binits{S.}}:
\batitle{{On the Homogenization of a Scalar Scattering Problem for Highly
  Oscillating Anisotropic Media}}.
\bjtitle{SIAM J. Math. Anal.}
\bvolume{48},
\bfpage{2532}--\blpage{2560}
(\byear{2016})
\end{barticle}
\endbibitem

\bibitem[\protect\citeauthoryear{Gloria et~al.}{2019}]{Gloria2019}
\begin{botherref}
\oauthor{\bsnm{Gloria}, \binits{A.}},
\oauthor{\bsnm{Neukamm}, \binits{S.}},
\oauthor{\bsnm{Otto}, \binits{F.}}:
{Quantitative estimates in stochastic homogenization for correlated coefficient
  fields}.
arXiv: Analysis of PDEs
(2019)
\end{botherref}
\endbibitem

\bibitem[\protect\citeauthoryear{Brezis and
  Br{\'e}zis}{2011}]{brezis2011functional}
\begin{bbook}
\bauthor{\bsnm{Brezis}, \binits{H.}},
\bauthor{\bsnm{Br{\'e}zis}, \binits{H.}}:
\bbtitle{Functional Analysis, Sobolev Spaces and Partial Differential
  Equations}
vol. \bseriesno{2}.
\bpublisher{Springer},
\blocation{{}}
(\byear{2011})
\end{bbook}
\endbibitem

\bibitem[\protect\citeauthoryear{Adams and Fournier}{2003}]{adams2003sobolev}
\begin{bbook}
\bauthor{\bsnm{Adams}, \binits{R.A.}},
\bauthor{\bsnm{Fournier}, \binits{J.J.}}:
\bbtitle{Sobolev Spaces}.
\bpublisher{Elsevier},
\blocation{{}}
(\byear{2003})
\end{bbook}
\endbibitem

\bibitem[\protect\citeauthoryear{Monk}{2003}]{monk2003finite}
\begin{bbook}
\bauthor{\bsnm{Monk}, \binits{P.}}:
\bbtitle{Finite Element Methods for Maxwell's Equations}.
\bpublisher{Oxford University Press},
\blocation{{}}
(\byear{2003})
\end{bbook}
\endbibitem

\bibitem[\protect\citeauthoryear{Clozeau et~al.}{2023}]{clozeau2023bias}
\begin{botherref}
\oauthor{\bsnm{Clozeau}, \binits{N.}},
\oauthor{\bsnm{Josien}, \binits{M.}},
\oauthor{\bsnm{Otto}, \binits{F.}},
\oauthor{\bsnm{Xu}, \binits{Q.}}:
{Bias in the representative volume element method: periodize the ensemble
  instead of its realizations}.
Foundations of Computational Mathematics,
1--83
(2023)
\end{botherref}
\endbibitem

\bibitem[\protect\citeauthoryear{}{}]{xlifepp}
\begin{botherref}
XLiFE++ : eXtended Library of Finite Elements in C++.
accessed 2023 Feb. 09.
\url{https://uma.ensta-paris.fr/soft/XLiFE++/}
\end{botherref}
\endbibitem

\bibitem[\protect\citeauthoryear{Ammari et~al.}{2013}]{ammari2013mathematical}
\begin{bbook}
\bauthor{\bsnm{Ammari}, \binits{H.}},
\bauthor{\bsnm{Garnier}, \binits{J.}},
\bauthor{\bsnm{Jing}, \binits{W.}},
\bauthor{\bsnm{Kang}, \binits{H.}},
\bauthor{\bsnm{Lim}, \binits{M.}},
\bauthor{\bsnm{S{\o}lna}, \binits{K.}},
\bauthor{\bsnm{Wang}, \binits{H.}}:
\bbtitle{Mathematical and Statistical Methods for Multistatic Imaging}
vol. \bseriesno{2098}.
\bpublisher{Springer},
\blocation{{}}
(\byear{2013})
\end{bbook}
\endbibitem

\bibitem[\protect\citeauthoryear{Gloria and
  Otto}{2017}]{gloria2017quantitative}
\begin{barticle}
\bauthor{\bsnm{Gloria}, \binits{A.}},
\bauthor{\bsnm{Otto}, \binits{F.}}:
\batitle{{Quantitative results on the corrector equation in stochastic
  homogenization}}.
\bjtitle{Journal of the European mathematical society}
\bvolume{19}(\bissue{11}),
\bfpage{3489}--\blpage{3548}
(\byear{2017})
\end{barticle}
\endbibitem

\bibitem[\protect\citeauthoryear{Bourgeat and
  Piatnitski}{2004}]{bourgeat2004approximations}
\begin{bchapter}
\bauthor{\bsnm{Bourgeat}, \binits{A.}},
\bauthor{\bsnm{Piatnitski}, \binits{A.}}:
\bctitle{{Approximations of effective coefficients in stochastic
  homogenization}}.
In: \bbtitle{Annales de l'IHP Probabilit{\'e}s et Statistiques},
vol. \bseriesno{40},
pp. \bfpage{153}--\blpage{165}
(\byear{2004})
\end{bchapter}
\endbibitem

\bibitem[\protect\citeauthoryear{Costabel et~al.}{2010}]{costabel}
\begin{botherref}
\oauthor{\bsnm{Costabel}, \binits{M.}},
\oauthor{\bsnm{Dauge}, \binits{M.}},
\oauthor{\bsnm{Nicaise}, \binits{S.}}:
{Corner Singularities and Analytic Regularity for Linear Elliptic Systems. Part
  I: Smooth domains.}
{}
(2010)
\end{botherref}
\endbibitem

\end{thebibliography}

\begin{appendices}

%
%

\section{Well-posedness of the scattering problem and $H^s$-regularity} \label{app:regularity}

We show in this appendix that the scattering problems are well-posed in $H^1(B_R\setminus \overline{D}) \times H^1(D)$ with a control of $u_\eps$ that is independent on $\eps$ and the randomness. To do so, we suppose that the bilinear form associated to \eqref{eq:mainBR} is coercive. Finally, we prove that $u_0$ can be more regular than $H^1$ under regularity assumptions on the boundary of $D$ and the source terms. The coercivity of the bilinear form is not a restrictive hypothesis. It can be shown for example under either one of the following sufficient conditions \cite{chaumont2023scattering}
\begin{enumerate}
\item $\Im k >0$
\item $kR$ is small enough (low frequency).
\end{enumerate}

\begin{prop}[Uniform stability under coercivity assumption] \label{prop:stabAbsorb_smallKR}
Let $D \subset \R^d$ be a non-empty, open, and bounded set having $\mathcal{C}^2$- boundary $\partial D$ such that the exterior domain $\R^d \setminus \overline{D}$ is connected. Let 
$A : \overline{D} \mapsto \mathbb{C}^{d \times d}$ and $n : \overline{D} \mapsto \mathbb{C}$. We suppose that $A(x)$, $x \in  \overline{D}$, is a definite positive matrix that satisfies, $\xi \cdot A \xi \geq \Lambda_A^- \vert\xi \vert^2$ and $|A \xi| \leq \Lambda_A^+|\xi|$ for all $\xi  \in \mathbb{C}^3$ and $x \in D$, and that $\Lambda_n^+ \geq n \geq \Lambda_n^- > 0$ where $\Lambda_A^+$, $\Lambda_A^-$, $\Lambda_n^+$, $\Lambda_n^-$ are positive constants.  \\
Let $f \in L^2(D)$, $g \in H^{\frac{1}{2}}(\partial D)$ and $h \in H^{-\frac{1}{2}}(\partial D)$. 
Then, there exists a unique $u \in H^1(B_R \setminus \overline{D}) \times H^1(D)$ solution of the transmission problem
\begin{equation} \label{eq:TP}
\left\{\begin{aligned}
&-\Delta u^- - k^2 n_0 u^- = 0 && \textrm{in}\, B_R \setminus \overline{D}, \\
& -\nabla\cdot\left(A \nabla u^+ \right) - k^2 n u^+ = f && \textrm{in}\, D,\\
&u^- - u^+ = g  && \textrm{on}\, \partial D,\\
& \nabla u^- \cdot \nu - A \nabla u^+ \cdot \nu = h && \textrm{on}\, \partial D,\\
& \nabla u \cdot \nu = \Lambda(u) && \textrm{on}\, \partial B_R,\\
\end{aligned} \right.
\end{equation}
which satisfies the uniform control:
\begin{equation}
\norme{u}_{H^1(B_R \setminus \overline{D})} + \norme{u}_{H^1(D)} \lesssim \norme{f}_{L^2(D)} + \norme{g}_{H^{\frac{1}{2}}(\partial D)} + \norme{h}_{H^{-\frac{1}{2}}(\partial D)}.
\end{equation}
\end{prop}
We also need a regularity result on the homogenized solution $u_0$ that we recall here.
\begin{prop}[$H^s$- regularity for the transmission problem] \label{prop:regHk}

Let $s \geq 2$. Let $D$ be a bounded domain of class $\mathcal{C}^s$. If $A, n \in \mathcal{C}^{s-2}(\overline{D})$, $f \in H^{s-2}(D)$ $g \in H^{s - {\frac{1}{2}}}(\partial D)$ and $h \in H^{s - {\frac{3}{2}}}(\partial D)$, then the unique solution $u \in H^1(D) \times H^1(B_R \setminus \overline{D}) $ of \eqref{eq:TP} belongs to $H^s(D) \times H^s(B_R \setminus \overline{D})$. Moreover the following estimate holds:
\begin{equation}
\norme{u}_{H^s(B_R \setminus \overline{D})} + \norme{u}_{H^s(D)} \lesssim \norme{f}_{H^{s-2}(D)} + \norme{g}_{H^{s - \frac{1}{2}}(\partial D)} + \norme{h}_{H^{s -\frac{3}{2}}(\partial D)}.
\end{equation}

\end{prop}

\begin{proof}

We rely on elliptic regularity results proved in \cite{costabel} to establish our result.
We first prove that $u$ belongs to $H^2(D)$. 
Since $u \in H^1(D)$, its trace on $\partial D$ belongs to $H^{\frac{1}{2}}(\partial D)$. Let $\widetilde{u} \in H^1(D)$ be the unique solution of
\begin{equation} \label{eq:tildeu_reg}
\left\{\begin{aligned}
& -\nabla\cdot\left(A \nabla \widetilde{u} \right) - k^2 n \widetilde{u} = 0 && \textrm{in}\, D,\\
&\widetilde{u} = u && \textrm{on}\, \partial D.\\
\end{aligned} \right.
\end{equation}
Then $\widetilde{u}$ satisfies the hypotheses of \cite[Theorem~3.4.1]{costabel} and therefore $\widetilde{u}$ is in $H^2(D)$. \\
By uniqueness of the solution of \eqref{eq:tildeu_reg}, we also have:
$$\widetilde{u} = u\quad \textrm{in}\, D.$$
Therefore $u \in H^2(D)$. Using the same reasoning in $B_R \setminus \overline{D}$ with a Dirichlet-to-Neumann operator on the boundary of $B_R$, one concludes that $u \in H^2(B_R \setminus \overline{D})$.  \\
Similarly we can now apply \cite[Theorem~2.3.2~(ii)]{costabel} to show that $u$ belongs in fact to $H^s(D) \times H^s(B_R \setminus \overline{D})$ and get the estimate.
\end{proof}

\section{Qualitative homogenization} \label{app:qualitative_hom}
We detail here the proof of the convergence of $u_\eps$ towards $u_0$ strongly in $L^2(B_R)$ and weakly in $H^1(B_R)$ by the method of oscillating test functions.

\begin{prop}[Homogenization of the scattering problem in {$H^1(B_R)$}] \label{prop:CVueps}
Let $u_\eps$ be the a.s. unique solution in $H^1(B_R)$ of \eqref{eq:mainBR} and $u_0 \in H^1(B_R)$ be the solution of \eqref{eq:u0BR}. Then we have the following convergence results as $\eps$ goes to $0$ 
\begin{equation}
\left\{\begin{aligned}
& u_\eps \xrightarrow{L^2(B_R)} u_0, \\
& \nabla u_\eps \xrightharpoonup{L^2(B_R)} \nabla u_0, \\
& a_\eps \nabla u_\eps \xrightharpoonup{L^2(B_R)} a^{hom} \nabla u_0.
\end{aligned} \right.
\end{equation}
\end{prop}
\noindent Here, we extend $a^{hom}$ by $I$ in $B_R \setminus \overline{D}$.
\begin{proof}

Since a.s. $u_\eps$ is uniformly bounded in $H^1(B_R)$ independently of $\eps$, by Rellich-Kondrachov theorem, we can extract a subsequence, still denoted $u_\eps$ such that 
\begin{equation}
u_\eps \xrightharpoonup{H^1(B_R)} u,
\end{equation}
for a certain $u \in H^1(B_R)$. By Rellich's theorem we have then $u_\eps \xlongrightarrow{L^2(B_R)} u$. Similarly thanks to the uniform ellipticity of $a$, we have:

$$\norme{a_\eps \nabla u_\eps}_{L^2(B_R)} \leq \Lambda_a \norme{\nabla u_\eps}_{L^2(B_R)} \lesssim \norme{u^{inc}}_{H^1(B_R)}.$$
Therefore we can also extract a subsequence of $a_\eps \nabla u_\eps$ such that 
$$a_\eps \nabla u_\eps \xrightharpoonup{L^2(B_R)} F^\star $$
for some $F^\star \in L^2(B_R)$. \\
We show that $u = u_0$ and $F^\star = a^{hom} \nabla u_0$. \\
By Birkhoff's ergodic theorem and the strong convergence of $u_\eps$ to $u$ in $L^2(D)$, we have that
$$n_\eps u_\eps \xrightharpoonup{L^2(B_R)} \E[n] u = n^{hom} u.$$
Furthermore, the $DtN$ operator is continuous from $H^{\frac{1}{2}}(\partial B_R)$ to $H^{-\frac{1}{2}}(\partial B_R)$ and the trace operator is continuous from $H^1(B_R)$ to $H^{\frac{1}{2}}(\partial B_R)$. Thus
$$\Lambda(u_\eps) \xrightharpoonup{H^{-\frac{1}{2}}(\partial B_R)} \Lambda(u).$$
By passing to the limit inside the variational formulation of \eqref{eq:mainBR} for $u_\eps$, one finds that, for all $v \in H^1(B_R)$, 
\begin{equation} \label{eq:limF}
\begin{split}
\int_{D} F^\star \cdot \nabla \overline{v} - & k^2 n^{hom} u \overline{v} + \int_{B_R \setminus \overline{D}} F^\star \cdot \nabla \overline{v} - k^2 n_0 u \overline{v} \\
& \hspace{0.7cm} - \left \langle \Lambda(u), v \right\rangle_{H^{-\frac{1}{2}}(\partial B_R), H^{\frac{1}{2}}(\partial B_R) } = \left \langle \Lambda(u^{inc}), v \right\rangle_{H^{-\frac{1}{2}}(\partial B_R), H^{\frac{1}{2}}(\partial B_R) }.
\end{split}
\end{equation}
For $i \in \intInter{1,d}$, let $\psi_i \in H^1_{loc}(\R^d)$ be the adjoint corrector satisfying 
\begin{equation}
- \nabla \cdot a^*(\nabla \psi_i + e_i) = 0 \quad \textrm{in } \mathcal{D}'(\R^d),
\end{equation}
with the anchoring condition $\frac{1}{\vert \square_0 \vert }\int_{\square_0} \psi_i = 0$. \\
Moreover, $\nabla \psi_i$ is stationary, verifies $\E[\nabla \psi_i] = 0$ and admits finite second order moment. 
\noindent Now for all $x \in \R^d$, let
$$\alpha_i(x) := x_i + \psi_i(x), $$
and 
$$\alpha_i^\eps(x) := \eps \alpha_i(\frac{x}{\eps}) =  x_i + \eps \psi_i(\frac{x}{\eps}). $$
Thanks to the sublinearity of $\psi_i$, $\alpha_i^\eps \xrightarrow{L^2(B_R)} x_i$. Moreover by Birkhoff's theorem $\nabla \alpha_i^\eps \xrightharpoonup{L^2(B_R)} e_i$. Thus
$$\alpha_i^\eps \xrightharpoonup{H^1(B_R)} x_i.$$
Similarly by Birkhoff's theorem, $a_\eps^* \nabla \alpha_i^\eps \xrightharpoonup{L^2(B_R)} \E[a^* \nabla \alpha_i] = \E[a^* (e_i + \nabla \psi_i)]$. \\
Moreover, since $\E[\nabla \phi_j \cdot a^* (e_i + \nabla \psi_i)]  = \E[\nabla \psi_i \cdot a ( e_j + \nabla \phi_j)] = 0$ for $i,j \in  \intInter{1,d},$
\begin{equation}
\begin{split}
\E[e_j \cdot a^* \nabla \alpha_i] & = \E[(e_j + \nabla \phi_j)\cdot a^* ( e_i + \nabla \psi_i) \\
& =  \E[a (e_j + \nabla \phi_j)\cdot e_i  \\
& = a^{hom}_{ji} = e_i \cdot a^{hom} e_j.
\end{split}
\end{equation}
For $\zeta \in C_c^\infty(B_R)$ consider the variational formulation of the problem solved by $u_\eps$ with the test function $\zeta \alpha^\eps_i$,
\begin{equation}
\begin{split}
0 & = \int_{B_R} a_\eps \nabla u_\eps \cdot \nabla (\overline{\zeta} \alpha^\eps_i) - k^2 n_\eps u_\eps \overline{\zeta} \alpha^\eps_i \\
& = \int_{B_R} a_\eps \nabla u_\eps \cdot (\nabla \overline{\zeta}) \alpha^\eps_i - (\nabla \overline{\zeta}) u_\eps \cdot a_\eps^* \nabla \alpha^\eps_i - k^2 n_\eps u_\eps \overline{\zeta} \alpha^\eps_i. \\ 
\end{split}
\end{equation}
Then, by passing to the limit
\begin{equation}
\begin{split}
\int_{D} F^\star \cdot (\nabla \zeta) x_i - & (\nabla \zeta) u \cdot (a^{hom})^* e_i - k^2 n^{hom} u \zeta x_i \\
& +\int_{B_R \setminus \overline{D}} F^\star \cdot (\nabla \zeta) x_i - a^{hom} (\nabla \zeta) u \cdot  e_i - k^2 n_0 u \zeta x_i  = 0.\\ 
\end{split}
\end{equation}
Moreover by \eqref{eq:limF}
\begin{equation}
\begin{split}
\int_{B_R} F^\star \cdot (\nabla \overline{\zeta}) x_i & = \int_{B_R} F^\star \cdot \nabla (\overline{\zeta} x_i) - F \cdot \overline{\zeta} e_i \\
& = \int_{B_R \setminus \overline{D}}k^2 n_0 u \overline{\zeta} x_i + \int_D  k^2 n^{hom} u \overline{\zeta} x_i - \int_{B_R} F^\star \cdot \overline{\zeta} e_i.
\end{split}
\end{equation}
Since $\zeta u \in H^1_0(B_R)$, an integration by parts yields 
\begin{equation}
\begin{split}
\int_{B_R} a^{hom} (\nabla \overline{\zeta}) u \cdot e_i = \int_{B_R} - a^{hom} \nabla u \cdot \overline{\zeta} e_i  
\end{split}
\end{equation}
which yields that for any $\zeta \in C^\infty_c(B_R)$ and for any $i \in \intInter{1,d}$, 
\begin{equation}
\int_{B_R} \overline{\zeta} e_i \cdot (a^{hom} \nabla u- F^\star) = 0.
\end{equation}
This implies that a.s. $F^\star = a^{hom} \nabla u$ in $\mathcal{D}'(B_R)$, thus in $L^2(B_R)$. \\
Finally, \eqref{eq:limF} can then be rewritten as
\begin{equation} 
\begin{split}
\int_{B_R} a^{hom} \nabla u \cdot \nabla \overline{v} - k^2 n^{hom} u \overline{v} - \left \langle \Lambda(u), v \right\rangle_{H^{-\frac{1}{2}}(\partial B_R), H^{\frac{1}{2}}(\partial B_R) } =\\ & \hspace{-1.6cm}  \left \langle \Lambda(u^{inc}), v \right\rangle_{H^{-\frac{1}{2}}(\partial B_R), H^{\frac{1}{2}}(\partial B_R) }.
\end{split}
\end{equation}
We get $u = u_0$. Moreover by uniqueness of the limit, we proved convergence of $u_\eps$ and $a_\eps \nabla u_\eps$ and not just of a subsequence. 
\end{proof}

\section{Homogenization with a less regular solution} \label{appendix:homLessReg}

In Section~\ref{section:2scErr} Proposition~\ref{prop:2scErrDecay}, we proved an error estimate for the two-scale expansion when ${u_0}_{|_D} \in W^{2, \infty}(D)$. This result still holds for less regular $u_0$ as stated in Proposition~\ref{prop:2scErrDecaylessRegular}. \\
As done in \cite{armstrong2019quantitative}, we consider an extension of ${u_0}_{|_D}$ that we denote $\widehat{u_0} \in W^{1 + \alpha, p}(\R^d)$. $\widehat{u_0}$ is defined through the Sobolev extension theorem stated below.
\begin{lem}[Sobolev extension theorem {\cite[Proposition~B.14]{armstrong2019quantitative}}] \label{lem:extensionThm}
Let $D$ be a bounded Lipschitz domain, $\alpha \in (0, \infty)$ and $p \in (1, \infty)$. The restriction operator ${W^{\alpha, p}(\R^d) \rightarrow W^{\alpha, p}(D)}$ has a bounded linear right inverse. That is, there exists a linear operator
$$\text{Ext}: W^{\alpha, p}(D) \rightarrow W^{\alpha, p}(\R^d), $$ 
such that, for every $u \in W^{\alpha, p}(D)$, 
$$\text{Ext}(u) = u~a.e.~ \text{ in } D,$$
and 
$$ \norme{\text{Ext}(u)}_{W^{\alpha, p}(\R^d)} \lesssim \norme{u}_{W^{\alpha, p}(D)}.$$
\end{lem}

We derive a convergence rate of $u_\eps$ towards the two-scale expansion when ${{u_0}_{|_D} \in  W^{1 + \alpha, p}(D)}$.
\begin{prop}[$H^1$- convergence of the two-scale expansion for ${{u_0}_{|_D} \in  W^{1 + \alpha, p}(D)}$] \label{prop:2scErrDecaylessRegular}
For $p \in (2, \infty]$, $\alpha \in (\frac{1}{p},1]$, suppose that $u_0 \in H^1(B_R)$ such that ${{u_0}_{|_D} \in W^{1 + \alpha, p}(D)}$ then
\begin{equation} \label{eq:2scBClessReg} 
\begin{split} 
\norme{u_\eps - \right.&\left. u_0}_{H^1(B_R \setminus \overline{D})} + \norme{u_\eps - u_0 - \widehat{u_{1,\eps}}}_{H^1(D)} \lesssim \eps^{\frac{1}{2}} \mu_d(\frac{1}{\eps})^{\frac{1}{2}} \widehat{\chi_{\eps,p}}\norme{u_0}_{W^{1 + \alpha, p}(D)},
\end{split}
\end{equation}
where $\widehat{u}_{1,\eps}$ is defined by
$$\widehat{u}_{1,\eps}(x) := \mathbbm{1}_{D}(x) \sum_{i=1}^d \phi_i \biggl(\frac{x}{\eps} \biggr) \partial_i \widehat{u_0} * \xi_\eps(x)  \quad \textrm{for } x \in B_R,$$
with the standard mollifier $\xi_\eps$ defined by
\begin{equation} \label{eq:mollifier}
\xi_\eps(x) := \eps^{-d} \left\{
\begin{array}{ll}
\vsd\dst c_d \exp(-\frac{1}{1-|\frac{x}{\eps}|^2}) &\textrm{for}\,\,|\frac{x}{\eps}|< 1,\\
\vsd\dst 0 &\textrm{for }\,\,|\frac{x}{\eps}| \geq 1,\\
\end{array}
\right.
\end{equation}
and $c_d$ is such that 
$$\int_{\R^d} \xi_\eps(x) \mathrm{d}x = 1.$$  
Here $\widehat{\chi_{\eps,p}}$ denotes a random variable satisfying the stochastic integrability \eqref{eq:Exp_corrector_bounds}.
\end{prop}

In order to prove the previous theorem, we introduce the boundary corrector and start by proving the result with the boundary corrector. 
\begin{prop}[$H^1$- convergence of the two-scale expansion with the boundary corrector for ${{u_0}_{|_D} \in W^{1 + \alpha, p}(D)}$] \label{prop:2scErrDecayBClessRegular}
For $p \in (2, \infty]$, $\alpha \in (0,1]$, suppose that $u_0 \in H^1(B_R)$ such that ${{u_0}_{|_D} \in W^{1 + \alpha, p}(D)}$ then
\begin{equation}
\norme{u_\eps - u_0 - \widehat{u}_{1,\eps} - \widehat{v_\eps}}_{H^1(B_R)} \lesssim \eps^{\alpha} \mu_d(\frac{1}{\eps}) \chi_{\eps,p} \norme{u_0}_{W^{1 + \alpha, p}(D)},
\end{equation}
where the boundary corrector ${\widehat{v_\eps} \in H^1(B_R \setminus \overline{D}) \times H^1(D)}$ is the solution of
\begin{equation}  \label{eq:hatVeps}
\left\{\begin{aligned} 
& - \Delta \widehat{v_\eps} - k^2 \widehat{v_\eps} = 0 &&\textrm{in}\, B_R \setminus\overline{D},\\
&-\nabla\cdot a_\eps \nabla \widehat{v_\eps} - k^2 n_\eps \widehat{v_\eps} = 0 &&\textrm{in}\, D,\\
& \widehat{v_\eps}^- - \widehat{v_\eps}^+ = \eps \widehat{u}_{1,\eps} &&\textrm{on}\, \partial D, \\
& \nabla \widehat{v_\eps}^- \cdot \nu - a_\eps \nabla \widehat{v_\eps}^+ \cdot \nu = \eps \left(\nabla \cdot(\sigma_i^\eps \partial_i \widehat{u_0}* \xi_\eps)^+\right) \cdot \nu\\
& \hspace{5cm} - k^2 \eps (\beta^\eps \widehat{u_0} * \xi_\eps)^+ \cdot \nu   &&\textrm{on}\, \partial D, \\
& \nabla \widehat{v_\eps} \cdot \nu = \Lambda(\widehat{v_\eps}) && \textrm{on}\, \partial B_R.
 \end{aligned}
\right.
\end{equation}
$\chi_{\eps,p}$ is a random variable defined as
\begin{equation}
\chi_{\eps,p} := \left(\eps^d \sum_{z \in P_\eps(D)}  \mathcal{C}(z)^{\frac{2p}{p-2}}\right)^{\frac{p-2}{2p}},
\end{equation}
with $\mathcal{C}$ denoting the constant in Proposition~\ref{prop:corrbounds}. Moreover $\chi_{\eps,p}$ satisfies the stochastic integrability \eqref{eq:Exp_corrector_bounds}.

\end{prop}

Both results of Proposition~\ref{prop:2scErrDecaylessRegular} and Proposition~\ref{prop:2scErrDecayBClessRegular} were established for the Poisson equation in a bounded domain with Dirichlet or Neumann condition in \cite[Chapter~6]{armstrong2019quantitative}. The proofs below use similar arguments as the ones developed in \cite{armstrong2019quantitative}. 

\begin{proof}[Proof of Proposition~\ref{prop:2scErrDecayBClessRegular}]

We denote $\widehat{Z_\eps} := u_\eps - u_0 - \eps \widehat{u_{1,\eps}}$. \\
As in Proposition~\ref{prop:2scErrDecayBC}, the boundary layer $\widehat{v_\eps}$ solution of \eqref{eq:hatVeps} is constructed such that 
$\widehat{Z_\eps} - \widehat{v_\eps}$ is the unique solution in $H^1(B_R)$ of
\begin{equation} 
\left\{\begin{aligned} 
&-\nabla\cdot a_\eps \nabla (\widehat{Z_\eps} - \widehat{v_\eps}) - k^2 n_\eps (\widehat{Z_\eps} - \widehat{v_\eps}) = \nabla \cdot \widehat{F_\eps} + k^2 \widehat{G_\eps} &&\textrm{in}\, B_R,\\
& \nabla (\widehat{Z_\eps} - \widehat{v_\eps}) \cdot \nu = \Lambda(\widehat{Z_\eps} - \widehat{v_\eps}) &&\textrm{on}\, \partial B_R,\\
 \end{aligned}
\right.
\end{equation}
where $\widehat{F_\eps}$ and $\widehat{G_\eps}$ are defined by
\begin{equation}
\widehat{F_\eps} := \eps  (a_\eps \phi_i^\eps -\sigma_i^\eps)\nabla (\partial_i \widehat{u_0} * \xi_\eps) + (a_\eps - a^{hom}) \nabla (\widehat{u_0} * \xi_\eps - u_0) + \eps k^2 \beta^\eps \widehat{u_0} * \xi_\eps,
\end{equation}
and
\begin{equation}
\widehat{G_\eps} := \eps \left((n_\eps \phi_i^\eps - \beta_i^\eps) \partial_i \widehat{u_0} * \xi_\eps\right) + (n_\eps - n^{hom})(u_0 - \widehat{u_0} * \xi_\eps).
\end{equation}
Moreover, $\widehat{Z_\eps} - \widehat{v_\eps}$ verifies a.s.
\begin{equation} 
\norme{\widehat{Z_\eps} - \widehat{v_\eps}}_{H^1(B_R)} \lesssim \norme{\widehat{F_\eps}}_{L^2(D)} + \norme{\widehat{G_\eps}}_{L^2(D)}.    
\end{equation}
To prove \eqref{eq:2scBClessReg}, we hence need to prove that
$$\norme{\widehat{F_\eps}}_{L^2(D)} + \norme{\widehat{G_\eps}}_{L^2(D)} \lesssim \eps^{\alpha} \mu_d(\frac{1}{\eps}) \chi_{\eps,p} \norme{u_0}_{W^{1 + \alpha, p}(D)}. $$
By the triangle inequality, we immediately get
\begin{equation} \label{eq:H1err0}
\begin{split}
& \norme{\widehat{F_\eps}}_{L^2(D)} + \norme{\widehat{G_\eps}}_{L^2(D)} \\
& \hspace{1cm}\lesssim \eps \norme{ |\phi^\eps - \sigma^\eps| |\nabla \nabla \widehat{u_0} * \xi_\eps| }_{L^2(D)} + \eps \norme{ |\beta^\eps| |\widehat{u_0} * \xi_\eps| }_{L^2(D)} \\
& \hspace*{0cm}+ \eps \norme{ |\phi^\eps - \beta^\eps| \nabla \widehat{u_0} * \xi_\eps }_{L^2(D)} + \norme{\nabla (\widehat{u_0} * \xi_\eps - u_0)}_{L^2(D)} + \norme{\widehat{u_0} * \xi_\eps - u_0}_{L^2(D)}.
\end{split}
\end{equation}
It remains to estimate these five terms. \\
We recall a useful Lemma, proved in \cite{armstrong2019quantitative}, which allows us to estimate the three first terms of \eqref{eq:H1err0}.
\begin{lem}[{\cite[Lemma~6.8]{armstrong2019quantitative}}] \label{lem:lemma68}
Fix $\alpha \in (0, 1]$ and $p \in (2, \infty)$. Let $f \in L^2(D + 2 \eps \square_0)$, $g \in L^p(D + 2 \eps \square_0)$ and its Sobolev extension $\widehat{g} \in L^p(\R^d)$. Then
\begin{equation} \label{eq:lemma68_1}
\norme{f |\widehat{g} * \xi_\eps| }_{L^2(D)} \lesssim \left(\eps^d \sum_{z \in P_\eps(D)} \norme{f}_{\underline{L}^2 (z + 2 \eps \square_0)}^{\frac{2p}{p-2}} \right)^{\frac{p-2}{2p}}\norme{g}_{L^p(D + 2 \eps \square_0)},
\end{equation}
where $\norme{f}_{\underline{L}^2(z + 2 \eps \square_0)} := \norme{f}_{\underline{L}^2(z + 2 \eps \square_0)} := \left(\fint_{z + 2 \eps \square_0}  |f|^2 \right)^{\frac{1}{2}}$. \\
Moreover, if $g \in W^{\alpha, p}(D + 2 \eps \square_0)$, then
\begin{equation}  \label{eq:lemma68_2}
\norme{f |\nabla(\widehat{g} * \xi_\eps)| }_{L^2(D)} \lesssim \eps^{\alpha-1} \left(\eps^d \sum_{z \in P_\eps(D)}\norme{f}_{\underline{L}^2 (z + 2 \eps \square_0)}^{\frac{2p}{p-2}} \right)^{\frac{p-2}{2p}}\norme{g}_{W^{\alpha,p}(D + 2 \eps \square_0)}.
\end{equation}

\end{lem}
Now, using \eqref{eq:lemma68_2}, with $f = \phi^\eps$ and $g = \nabla \widehat{u_0}$ and the corrector estimate of Proposition~\ref{prop:corrbounds}, we obtain
\begin{equation} \label{eq:H1err1}
\begin{split}
\norme{ |\phi^\eps | \right.&\left. |\nabla (\nabla \widehat{u_0} * \xi_\eps)| }_{L^2(D)} \\
& \lesssim \eps^{\alpha-1} \left(\eps^d \sum_{z \in P_\eps(D)} \norme{\phi^\eps}_{\underline{L}^2(z + 2\eps \square_0))}^{\frac{2p}{p-2}} \right)^{\frac{p-2}{2p}}\norme{\nabla \widehat{u_0}}_{W^{\alpha,p}(D + 2 \eps \square_0)} \\
& \lesssim \eps^{\alpha-1} \left(\eps^d \sum_{z \in P_\eps(D)} \norme{\phi}_{\underline{L}^2 (2\square_0)}^{\frac{2p}{p-2}} \right)^{\frac{p-2}{2p}}\norme{\nabla u_0}_{W^{\alpha,p}(D)} \\
& \lesssim \eps^{\alpha-1}\mu_d(\frac{1}{\eps}) \left(\eps^d \sum_{z \in P_\eps(D)} \mathcal{C}(z)^{\frac{2p}{p-2}} \right)^{\frac{p-2}{2p}}\norme{\nabla u_0}_{W^{\alpha,p}(D)}.
\end{split}
\end{equation}
Similarly, with $f = \phi^\eps - \beta^\eps$ and $g = \widehat{u_0}$,
\begin{equation} \label{eq:H1err2}
\norme{ |\phi^\eps - \beta^\eps| | \nabla \widehat{u_0} * \xi_\eps| }_{L^2(D)} \lesssim \eps^{\alpha-1}\mu_d(\frac{1}{\eps}) \left(\eps^d \sum_{z \in P_\eps(D)} \mathcal{C}(z)^{\frac{2p}{p-2}} \right)^{\frac{p-2}{2p}} \norme{u_0}_{W^{\alpha,p}(D)}.
\end{equation}
Now, using \eqref{eq:lemma68_1} with $f = \beta^\eps$ and $g = u_0$, we get
\begin{equation} \label{eq:H1err3}
\norme{|\beta^\eps| |u_0 * \xi_\eps| }_{L^2(D)} \lesssim  \mu_d(\frac{1}{\eps}) \left(\eps^d \sum_{z \in P_\eps(D)} \mathcal{C}(z)^{\frac{2p}{p-2}} \right)^{\frac{p-2}{2p}} \norme{u_0}_{L^p(D)}.
\end{equation}
To estimate the last two terms of \eqref{eq:H1err0}, we recall another useful lemma.
\begin{lem}[{\cite[Lemma~6.7]{armstrong2019quantitative}}] \label{lem:lemma67}
Fix $1 \leq q \leq p < \infty$ and $0 < \alpha \leq 1$. Let ${g \in W^{\alpha, p}(D + 2 \eps \square_0)}$ and its Sobolev extension $\widehat{g} \in W^{\alpha, p}(\R^d)$. Then
\begin{equation}
\norme{g - (\widehat{g} * \xi_\eps)}_{L^q(D)} \lesssim |D|^{\frac{1}{q} - \frac{1}{p}} \eps^{\alpha} \norme{g}_{W^{\alpha, p}(D + 2 \eps \square_0)}.
\end{equation} 

\end{lem}
Using this Lemma, with $g = \nabla u_0$, $q = 2$, and $p > 2$, we have
\begin{equation} \label{eq:H1err4}
\norme{\nabla u_0 - (\nabla \widehat{u_0}) * \xi_\eps}_{L^2(D)} \lesssim \eps^\alpha \norme{\nabla \widehat{u_0}}_{W^{\alpha, p}(D + 2 \eps \square_0)} \lesssim \eps^\alpha \norme{u_0}_{W^{1+\alpha, p}(D)}
\end{equation}
and with $g = u_0$ we obtain similarly
\begin{equation} \label{eq:H1err5}
\norme{u_0 - \widehat{u_0} * \xi_\eps}_{L^2(D)} \lesssim \eps^\alpha \norme{\widehat{u_0}}_{W^{\alpha, p}(D + 2 \eps \square_0)} \lesssim \eps^\alpha \norme{u_0}_{W^{\alpha, p}(D)}.
\end{equation}
Inserting \eqref{eq:H1err1}, \eqref{eq:H1err2}, \eqref{eq:H1err3}, \eqref{eq:H1err4}, \eqref{eq:H1err5} into \eqref{eq:H1err0} gives us \eqref{eq:2scBClessReg}, concluding the proof of Proposition~\ref{prop:2scErrDecayBC}.

\end{proof}

By estimating the $H^1$-norm of the boundary corrector $\widehat{v_\eps}$, we can now prove Proposition~\ref{prop:2scErrDecaylessRegular}.

\begin{proof}[Proof of Proposition~\ref{prop:2scErrDecaylessRegular}]
We consider $\widehat{V_\eps} := \widehat{v_\eps} - \eps \eta_\eps \widehat{u_{1, \eps}}$ the a.s. unique solution in $H^1(B_R)$ of 
\begin{equation} 
\left\{\begin{aligned} 
& -\Delta \widehat{V_\eps} - k^2 \widehat{V_\eps} = 0 &&\textrm{in}\, B_R\setminus\overline{D},\\
& -\nabla \cdot a_\eps \nabla \widehat{V_\eps} - k^2 n_\eps \widehat{V_\eps} = -\eps \nabla \cdot a_\eps \nabla(\eta_\eps \widehat{u_{1, \eps}}) + \eps k^2 n_\eps \eta_\eps \widehat{u_{1, \eps}}   && \textrm{in}\, D,\\
& \nabla \widehat{V_\eps}^- \cdot \nu - a_\eps \nabla \widehat{V_\eps}^+ \cdot \nu = \eps a_\eps \nabla( \eta_\eps \widehat{u_{1, \eps}}) \cdot \nu  + \eps \left(\nabla \cdot(\sigma_i^\eps \partial_i \widehat{u_0} * \xi_\eps)^+\right) \cdot \nu\\
& \hspace{5cm} - k^2 \eps (\beta^\eps \widehat{u_0} * \xi_\eps)^+ \cdot \nu   &&\textrm{on}\, \partial D, \\
& \nabla \widehat{V_\eps} \cdot \nu = \Lambda(\widehat{V_\eps}) && \textrm{on}\, \partial B_R. \\
\end{aligned}
\right.
\end{equation}
As in the proof of Proposition~\ref{prop:2scErrDecay}, we estimate $\norme{\widehat{V_\eps}}_{H^1(B_R)}$ by writing the variational formulation. For $w \in H^1(B_R)$,
\begin{equation}
\begin{split}
\int_{B_R} & a_\eps \nabla \widehat{V_\eps} \cdot \nabla \overline{w} - k^2 n_\eps \widehat{V_\eps} \overline{w} - \left \langle \Lambda(\widehat{V_\eps}), w \right\rangle_{H^{-\frac{1}{2}}(\partial B_R), H^{\frac{1}{2}}(\partial B_R) }  \\
& = \int_D - \eps a_\eps\nabla(\eta_\eps \widehat{u_{1,\eps}}) \cdot \nabla \overline{w} + \eps k^2 n_\eps(\eta_\eps \widehat{u_{1,\eps}}) - \eps \nabla \cdot (\sigma_i^\eps \eta_\eps \partial_i \widehat{u_0} * \xi_\eps) \cdot \nabla \overline{w} \\
& \hspace{0.5cm} - k^2 \eps \nabla \cdot (\beta^\eps \widehat{u_0} * \xi_\eps \eta_\eps) \overline{w} + k^2 \eps \beta^\eps \widehat{u_0} * \xi_\eps \eta_\eps \cdot \nabla \overline{w}.
\end{split}
\end{equation}
In particular, by the coercivity of the sesquilinear form, we get
\begin{equation}
\begin{split}
\norme{\widehat{v_\eps}}_{H^1(D)} + \norme{\widehat{v_\eps}}_{H^1(B_R \setminus \overline{D})} & \lesssim \eps \norme{\eta_\eps \widehat{u_{1,\eps}}}_{H^1(D)} + \eps \norme{\sum_{i=1}^d \nabla \cdot (\sigma_i^\eps \partial_i \widehat{u_0}*\xi_\eps \eta_\eps)}_{L^2(D)} + \\
& \hspace{1cm} \eps \norme{\nabla \cdot (\beta^\eps \widehat{u_0}*\xi_\eps \eta_\eps)}_{L^2(D)} + \eps \norme{\beta^\eps \widehat{u_0}*\xi_\eps \eta_\eps}_{L^2(D)} . 
\end{split}
\end{equation}
Let us now estimate $\norme{\eta_\eps \widehat{u_{1,\eps}}}_{H^1(D)} = \norme{\eta_\eps \phi_i^\eps \partial_i \widehat{u_0} * \xi_\eps}_{H^1(D)}$. The three other terms can then be estimated using similar arguments.
First
\begin{equation}
\begin{split}
\norme{\nabla (\eta_\eps \phi_i^\eps \partial_i \widehat{u_0} * \xi_\eps)}_{L^2(D)} & \lesssim \norme{(\nabla \eta_\eps) \phi_i^\eps \partial_i \widehat{u_0} * \xi_\eps + \eta_\eps \nabla (\phi_i^\eps \partial_i \widehat{u_0} * \xi_\eps)}_{L^2(D)} \\
& \hspace{-1.5cm}\lesssim \norme{\left(\frac{1}{\mu_d(\frac{1}{\eps})}|\nabla \widehat{u_0} * \xi_\eps| + \eps|\nabla (\nabla \widehat{u_0} * \xi_\eps)|\right) |\phi^\eps| + |\nabla \phi| |\nabla \widehat{u_0} * \xi_\eps|}_{L^2(\mathcal{S}_{\eta_\eps})}.
\end{split}
\end{equation}
Eq \eqref{eq:lemma68_1} combined with the bounds on the corrector implies
\begin{equation} \label{eq:tildeu11}
\begin{split}
\norme{\frac{1}{\mu_d(\frac{1}{\eps})}|\nabla \widehat{u_0} * \xi_\eps| |\phi^\eps| }_{L^2(\mathcal{S}_{\eta_\eps})} & \lesssim \eps^{\frac{p-2}{2p}} \mu_d(\frac{1}{\eps})^{\frac{p-2}{2p}-1} \widetilde{\chi_{\eps, p}^1} \norme{\nabla \widehat{u_0}}_{L^p(\mathcal{S}_{\eta_\eps} + 2 \eps \square_0)} \\
\end{split}
\end{equation}
and
\begin{equation} \label{eq:tildeu12}
\begin{split}
\biggl\Vert |\nabla \widehat{u_0} * \xi_\eps| |\nabla \phi^\eps| \biggr\Vert_{L^2(\mathcal{S}_{\eta_\eps})} & \lesssim \eps^{\frac{p-2}{2p}} \mu_d(\frac{1}{\eps})^{\frac{p-2}{2p}} \widetilde{\chi_{\eps, p}^2} \norme{\nabla \widehat{u_0}}_{L^p(\mathcal{S}_{\eta_\eps} + 2 \eps \square_0)} \\
\end{split}
\end{equation}
where the random variables $\widetilde{\chi_{\eps, p}^1}$ and $\widetilde{\chi_{_\eps,p}^2}$ are defined as
\begin{equation} \label{eq:ChiEpsP}
\left\{\begin{aligned} 
& \widetilde{\chi_{\eps, p}^1} := \left(\frac{\eps^d}{\eps \mu_d(\frac{1}{\eps})} \sum_{z \in P_\eps(\mathcal{S}_{\eta_\eps})} \mathcal{C}(z)^{\frac{2p}{p-2}} \right)^{\frac{p-2}{2p}}, \\
& \widetilde{\chi_{\eps, p}^2} := \left(\frac{\eps^d}{\eps \mu_d(\frac{1}{\eps})} \sum_{z \in P_\eps(\mathcal{S}_{\eta_\eps})} (1+ r_*(z))^{d\frac{2p}{p-2}}\right)^{\frac{p-2}{2p}}.
\end{aligned}
\right.
\end{equation}
We use the following Lemma to estimate $\norme{\nabla \widehat{u_0}}_{L^p(\mathcal{S}_{\eta_\eps} + 2 \eps \square_0)}$.
\begin{lem}[{\cite[Lemma~6.12]{armstrong2019quantitative}}] \label{lem:lemma612}
Fix $p \in (1, \infty)$, $\alpha > \frac{1}{p}$, $q \in [1, p]$ and $\beta \in \biggl(0, \frac{1}{q}\biggr]$. For every $f \in W^{\alpha, p}(\R^d)$ and $r \in (0, 1]$, 
\begin{equation} \label{eq:Lp_Walphap}
\norme{f}_{L^q(\partial D + B_r)} \lesssim r^\beta \norme{f}_{W^{\alpha, p}(\R^d)}.
\end{equation} 

\end{lem}
\noindent Applying \eqref{eq:Lp_Walphap} with $f = \nabla \widehat{u_0}$, $r = 4 \eps \mu_d(\frac{1}{\eps})$, $q = p$, $\alpha > \frac{1}{p}$, $\beta = \frac{1}{q} = \frac{1}{p}$ yields
\begin{equation} \label{eq:tildeu13}
\norme{\nabla \widehat{u_0}}_{L^p(\mathcal{S}_{\eta_\eps} + 2 \eps \square_0)}  \lesssim \eps^{\frac{1}{p}} \mu_d(\frac{1}{\eps})^{\frac{1}{p}}\norme{\nabla u_0}_{W^{\alpha,p}(D)}.
\end{equation}
Furthermore, using \eqref{eq:lemma68_2}, we get
\begin{equation} \label{eq:tildeu14}
\begin{split}
\biggl\Vert \eps|\nabla (\nabla \widehat{u_0} * \xi_\eps)| |\phi^\eps| \biggr\Vert_{L^2(\mathcal{S}_{\eta_\eps})} & \lesssim \eps^\alpha \mu_d(\frac{1}{\eps}) \eps^{\frac{p-2}{2p}} \mu_d(\frac{1}{\eps})^{\frac{p-2}{2p}} \widetilde{\chi_{\eps, p}^1} \norme{\nabla u_0}_{W^{\alpha, p}(\mathcal{S}_{\eta_\eps} + 2 \eps \square_0)}\\
& \lesssim \eps^{\frac{1}{2}} \mu_d(\frac{1}{\eps})^{\frac{1}{2}} \eps^{\alpha - \frac{1}{p}} \mu_d(\frac{1}{\eps})^{1-\frac{1}{p}} \widetilde{\chi_{\eps, p}^1} \norme{\nabla u_0}_{W^{\alpha, p}(D)}.
\end{split}
\end{equation}
Combining the last estimates \eqref{eq:tildeu11}, \eqref{eq:tildeu12}, \eqref{eq:tildeu13} and \eqref{eq:tildeu14}, and the fact that $\alpha > \frac{1}{p}$, one has that
\begin{equation} \label{eq:tildeu1H1}
\begin{split}
\norme{\nabla (\eta_\eps \phi_i^\eps \right.&\left.\partial_i \widehat{u_0} * \xi_\eps)}_{L^2(D)}\\
& \lesssim \eps^{\frac{1}{2}}\mu_d(\frac{1}{\eps})^{\frac{1}{2}}\left(\mu_d(\frac{1}{\eps})^{-1}\chi_{\eps, p}^1 + \chi_{\eps, p}^2 + \eps^{\alpha - \frac{1}{p}} \mu_d(\frac{1}{\eps})^{1-\frac{1}{p}} \chi_{\eps, p}^1 \right) \norme{\nabla u_0}_{W^{\alpha, p}(D)} \\
& \lesssim \eps^{\frac{1}{2}}\mu_d(\frac{1}{\eps})^{\frac{1}{2}} \widetilde{\chi_{\eps, p}^3} \norme{u_0}_{W^{1+ \alpha, p}(D)},
\end{split}
\end{equation}
where the random variable $\widetilde{\chi_{\eps, p}^3}$ is defined as
$$\widetilde{\chi_{\eps, p}^3} := \mu_d(\frac{1}{\eps})^{-1}\widetilde{\chi_{\eps, p}^1} + \widetilde{\chi_{\eps, p}^2} + \eps^{\alpha - \frac{1}{p}} \mu_d(\frac{1}{\eps})^{1-\frac{1}{p}} \widetilde{\chi_{\eps, p}^1},$$
and satisfies the stochastic integrability \eqref{eq:Exp_corrector_bounds}.\\
We finally proved that
\begin{equation}
\norme{\widetilde{v_1}}_{H^1(D)} + \norme{\widetilde{v_1}}_{H^1(B_R \setminus\overline{D})} \lesssim \eps^{\frac{1}{2}}\mu_d(\frac{1}{\eps})^{\frac{1}{2}} \widetilde{\chi_{\eps, p}^3} \norme{u_0}_{W^{1+ \alpha, p}(D)}.
\end{equation}

\end{proof}

\section{Proof of Lemma~\ref{lem:CSestimates}} \label{sec:lemCSestimates}


\begin{proof}[Proof of Lemma~\ref{lem:CSestimates}]
(a) First note that for $y \in \R^d$,
$$\int_{\R^d} \mathbbm{1}_{B_t(x)}(y) \mathrm{d}x = \int_{\R^d} \mathbbm{1}_{B_t(y)}(x) \mathrm{d}x = C t^d, $$
where $C$ depends only on $d$.\\
We have also for $y,z \in \R^d$,
$$\int_{\R^d} \mathbbm{1}_{B_t(x)}(y) \mathbbm{1}_{B_t(x)}(z) \mathrm{d}x \leq \int_{\R^d} \mathbbm{1}_{B_t(x)}(y) \mathrm{d}x = C t^d. $$

Let $U \in L^1(D)$ and $t >0$. By Fubini's theorem,
\begin{equation}
\begin{split}
\int_{\R^d} \left(\int_{B_t(x) \cap D} |U| \right) \mathrm{d}x & = 
\int_{\R^d} \int_D |U(y)| \mathbbm{1}_{B_t(x)}(y) \mathrm{d}y \mathrm{d}x\\
& = \int_D |U(y)| \left(\int_{\R^d} \mathbbm{1}_{B_t(x)}(y) \mathrm{d}x \right) \mathrm{d}y \\ 
& \leq C t^d \int_D |U(y)| \mathrm{d}y \\
& \leq C t^d \left(\int_D |U|\right).
\end{split}
\end{equation}

(b) Similarly,
\begin{equation}
\begin{split}
\int_{\R^d} \rho_T(x)^{\alpha} & \left(\int_{B_t(x) \cap D} |U| \right)^2 \mathrm{d}x \\
& = \int_{\R^d} \rho_T(x)^\alpha \left(\int_D |U(y)| \mathbbm{1}_{B_t(x)}(y) \mathrm{d}y\right) \left(\int_D |U(z)| \mathbbm{1}_{B_t(x)}(z) \mathrm{d}z \right)  \mathrm{d}x\\
& = \int_D |U(y)| \int_D |U(z)| \left(\int_{\R^d} \rho_T(x)^{\alpha} \mathbbm{1}_{B_t(x)}(y) \mathbbm{1}_{B_t(x)}(z)  \mathrm{d}x \right) \mathrm{d}y \mathrm{d}z. \\
\end{split}
\end{equation}

Since $\mathbbm{1}_{B_t(x)}(y) \mathbbm{1}_{B_t(x)}(z) = 0$ if $|x - y| > t$ or $|x - z|>t$, one can bound $\rho_T(x)^\alpha$ in the third integral by $\sup_{y \in D}\biggl(\frac{t + y}{T} + 1\biggr)^{\alpha}$ yielding the result.

\end{proof}

\end{appendices}

\end{document}